\documentclass[a4paper,12pt]{amsart}
\usepackage{amsmath, amssymb, amsthm}
\usepackage{verbatim}
\usepackage{enumerate}
\usepackage{hyperref}
\usepackage{graphicx}
\usepackage[all]{xy}
\usepackage[normalem]{ulem}

\newcounter{theorem}
\newtheorem{thm}[theorem]{Theorem}
\newtheorem{lemma}[theorem]{Lemma}
\newtheorem{prop}[theorem]{Proposition}
\newtheorem{cor}[theorem]{Corollary}
\newtheorem{defn}[theorem]{Definition}

\theoremstyle{remark}
\newtheorem*{remark*}{Remark}
\newtheorem{remark}[theorem]{Remark}
\newtheorem{example}[theorem]{Example}
\newtheorem{question}[theorem]{Question}

\numberwithin{equation}{section}
\numberwithin{theorem}{section}

\newcommand{\e}{\epsilon}
\newcommand{\dl}{\delta}

\newcommand{\R}{\mathbb{R}}

\newcommand{\N}{\mathbb{N}}
\renewcommand{\setminus}{\backslash}
\renewcommand{\emptyset}{\varnothing}

\newcommand{\her}{\mathrm{her}}

\newcommand{\id}{\mathrm{id}}
\newcommand{\diag}{\mathrm{diag}}
\newcommand{\ev}{\mathrm{ev}}

\newcommand{\Hom}{\mathrm{Hom}}
\newcommand{\Prim}{\mathrm{Prim}}
\newcommand{\Ad}{\mathrm{Ad}}
\newcommand{\dr}{\mathrm{dr}}

\newcommand{\el}{l}

\newcommand{\labelledthing}[2]{\hspace{4pt}\buildrel {#2} \over #1 \hspace{3pt}} 

\newcommand{\labelledrightarrow}{\labelledthing{\longrightarrow}}

\newcommand{\mathrmm}{\text}

\include{hyph}

\begin{document}

\title[Decomposition rank of ASH algebras]{Decomposition rank of approximately subhomogeneous $\mathrmm C^*$-algebras}

\author[G.\ A.\ Elliott]{George A.\ Elliott}
\author[Z.\ Niu]{Zhuang Niu}
\author[L.\ Santiago]{Luis Santiago}
\author[A.\ Tikuisis]{Aaron Tikuisis}

\address{\hskip-\parindent
George A.\ Elliott, Department of Mathematics, University of Toronto, Toronto, Ontario, Canada, M5S 2E4}
\email{elliott@math.toronto.edu}
\address{\hskip-\parindent
Zhuang Niu, Department of Mathematics, University of Wyoming, Laramie, WY, 82071, USA.}
\email{zniu@uwyo.edu}
\address{\hskip-\parindent
Luis Santiago, Department of Mathematical Sciences, Lakehead University, Thunderbay, Ontario, Canada, 	P7B 5E1.}
\email{luissant@gmail.com}
\address{\hskip-\parindent Aaron Tikuisis, Department of Mathematics and Statistics, University of Ottawa, Ottawa, Canada, K1N 6N5.}
\email{aaron.tikuisis@uottawa.ca}

\begin{abstract}
It is shown that every Jiang-Su stable approximately subhomogeneous $\mathrmm C^*$-algebra has finite decomposition rank.
This settles a key direction of the Toms--Winter conjecture for simple approximately subhomogeneous $\mathrmm C^*$-algebras.
A key step in the proof is that subhomogeneous $\mathrmm C^*$-algebras are locally approximated by a certain class of more tractable subhomogeneous algebras, namely, a non-commutative generalization of the class of cell complexes.
The result is applied, in combination with other recent results, to show classifiability of crossed product $\mathrmm C^*$-algebras associated to minimal homeomorphisms with mean dimension zero.
\end{abstract}



\maketitle

\renewcommand*{\thetheorem}{\Alph{theorem}}

\section{Introduction}

Decomposition rank is a concept which has proven crucial in recent developments in the structure and classification of $\mathrmm C^*$-algebras.
Defined by Kirchberg and Winter, it is a marriage of the topological concept of Lebesgue covering dimension with the functional-analytic ideas of nuclearity and quasidiagonality.
This notion gained currency in the programme of classifying $\mathrmm C^*$-algebras, where finite decomposition rank has been shown to imply strong structural conditions on a $\mathrmm C^*$-algebra \cite{BrownWinter:quasitraces,NgWinter:CFP,TomsWinter:V1,Winter:RR0drClass,Winter:drZstable,Winter:CrProducts}.
A spectacular classification result of Gong, Lin, and two of us (G.A.E.\ and Z.N.) uses finite decomposition rank as a key hypothesis \cite{EGLN:Class2}.

The concept of decomposition rank is enshrined in the so-called Toms-Winter conjecture about simple, separable, unital, nuclear, finite $\mathrmm C^*$-algebras, which says, in particular, that finite decomposition rank and tensorial absorption of the Jiang-Su algebra ($\mathcal Z$) should be equivalent for such $\mathrmm C^*$-algebras.
The component ``$\mathcal Z$-absorption implies finite decomposition rank'' is a pivotal problem, which has been solved using classification in certain limited cases (using tracial approximation hypotheses \cite{Lin:AsympClassification,GongLinNiu}).
In this article we prove this direction of the Toms--Winter conjecture for $\mathrmm C^*$-algebras that are locally subhomogeneous.

The class of simple locally subhomogeneous $\mathrmm C^*$-algebras is very broad: it is known to include many natural examples (see the introduction to \cite{Toms:rigidity}), and it exhausts the range of the Elliott invariant \cite{Elliott:ashrange}.
As a consequence of our main result (discussed below), the class includes crossed product $\mathrmm C^*$-algebras associated to minimal homeomorphisms of mean dimension zero (and more generally, the $\mathcal Z$-stabilization of a crossed product given by a minimal homeomorphism).
There are no simple, separable, nuclear, stably finite $\mathrmm C^*$-algebras which are known to lie outside the class of locally subhomogeneous $\mathrmm C^*$-algebras.

Our main result is as follows.

\begin{thm}\label{thm:MainThmA}
Let $A$ be a locally subhomogeneous, $\mathcal Z$-stable $\mathrmm C^*$-algebra.
Then $\dr A \leq 2$.
\end{thm}

This result is novel even in the case of simple, unital, separable, locally subhomogeneous, $\mathcal Z$-stable $\mathrmm C^*$-algebras.

The problem of showing directly that $\mathcal Z$-absorption implies finite decomposition rank, as is done here, (and in particular, avoiding appealing to simplicity or the Universal Coefficient Theorem) has proven pivotal, and has spurred the introduction of novel techniques in two very different streams of attack.
On the one hand, for non-simple $\mathrmm C^*$-algebras of a very particular form---namely, algebras $\mathrmm C(X) \otimes \mathcal Z$ (and by permanence properties, also limits of $\mathrmm C^*$-algebras Morita equivalent to these)---a direct computation, giving decomposition rank at most two, was obtained by one of us (A.T.) and Winter \cite{TW:Zdr}, using quasidiagonality of the cone over $\mathcal O_2$, together with a one-dimensional approximation result inside $\mathrmm C(Y) \otimes \mathcal O_2$ due to Kirchberg and R\o rdam \cite{KirchbergRordam:pi3}.
(The present paper uses this result.)
On the other hand, for simple, quasidiagonal $\mathrmm C^*$-algebras, techniques inspired by---and building on---Connes's proof that injective von Neumann algebras are hyperfinite have been developed by Matui and Sato in the unique trace case \cite{MatuiSato:dr}, and subsequently by Bosa, Brown, Sato, White, Winter, and one of us (A.T.) in the case of compact extreme trace space (where the hypothesis of quasidiagonality needs to be strengthened slightly) \cite{BBSTWW} (see also \cite{SWW:Znucdim}).
For the class considered in \cite{BBSTWW}, the optimal decomposition rank estimate of one was proven.

The present article builds on the $\mathcal Z$-stabilized commutative case dealt with in \cite{TW:Zdr}, showing how to replace the algebra $\mathrmm C(X)$ by an arbitrary subhomogeneous algebra, i.e., a $\mathrmm C^*$-algebra for which there is a finite bound on the (Hilbert space) dimension of its irreducible representations.
While there is some overlap between the class of $\mathrmm C^*$-algebras in Theorem \ref{thm:MainThmA} and in \cite{BBSTWW} (namely, all simple, locally subhomogeneous algebras with compact extreme trace space), the two results mostly complement each other.
The present theorem allows non-simple $\mathrmm C^*$-algebras, and even in the simple case, allows $\mathrmm C^*$-algebras with non-compact extreme trace space.
On the other hand, the result of \cite{BBSTWW}, where it applies, gives an optimal bound of one instead of two for the decomposition rank.
(Also, the result in \cite{BBSTWW} holds under (potentially) much weaker hypotheses than local subhomogeneity---not even the Universal Coefficient Theorem is required.)

In an application of our main result, we show in Section \ref{sec:CrossedProd} that it applies to the UHF-stabilization of a minimal $\mathbb Z$-crossed product, by using an argument involving Berg's technique.
It follows that $\mathcal Z$-stable minimal $\mathbb Z$-crossed products satisfy the hypotheses of a far-reaching classification result due to Elliott, Gong, Lin, and Niu \cite{EGLN:Class2}.
We summarize this consequence in the following statement (the proof can be found in Section \ref{sec:CrossedProd}).

\begin{cor}
\label{cor:CrProdClass}
Let $X$ be a compact metrizable space and let $\alpha:X \to X$ be a minimal homeomorphism.
Set $A := C(X) \rtimes_\alpha \mathbb Z$, if $\alpha$ has mean dimension zero, or $(C(X) \rtimes_\alpha \mathbb Z) \otimes \mathcal Z$, otherwise.
Then $A$ is an inductive limit of subhomogeneous $\mathrmm C^*$-algebras, and is classifiable in the following sense.

Let $B$ be another $\mathrmm C^*$-algebra of the same form, or a simple separable unital $\mathrmm C^*$-algebra with finite decomposition rank and which satisfies the Universal Coefficient Theorem.
Then $A \cong B$ if and only if $\mathrm{Ell}(A) \cong \mathrm{Ell}(B)$ (where $\mathrm{Ell}(\cdot)$ denotes the Elliott invariant, consisting of $K$-theory paired with traces).
\end{cor}

In fact, the abstract classification in \cite{EGLN:Class2} applies to all $\mathcal Z$-stable locally subhomogeneous $\mathrmm C^*$-algebras, again using Theorem \ref{thm:MainThmA} to verify a key hypothesis (see also \cite{EGLN:ASH}).

Our proof of Theorem \ref{thm:MainThmA} consists of two major steps.
The first step consists in proving the following result, which is perhaps of independent interest.

\begin{thm}\label{thm:ApproxCellA}
Let $A$ be a unital subhomogeneous $\mathrmm C^*$-algebra.
Then $A$ is locally approximated by non-commutative cell complexes.
Moreover, when $A$ is separable, the topological dimension of the approximating non-commutative cell complexes can be made equal to the topological dimension of $A$.
\end{thm}

The non-commutative cell complexes in this result, defined formally in Section \ref{sec:NCcellDef}, are special subhomogeneous algebras, akin to the non-commutative CW complexes of Eilers-Loring-Pedersen, though slightly more general for technical reasons (see Remark \ref{rmk:NCcellDef} (ii)).
As the name suggests, they arise by gluing together matrix algebras over spaces $D^n$ along the boundaries $S^{n-1}$.
What is most important here is that the boundary $S^{n-1}$ is a neighbourhood retract in $D^n$.

The topological dimension of a separable subhomogeneous $\mathrmm C^*$-algebra measures a sort of dimension of its primitive ideal space, and agrees with the decomposition rank of the $\mathrmm C^*$-algebra; see Definition \ref{def:TopDim}.

Theorem \ref{thm:ApproxCellA} is a non-commutative analogue of the result that $n$-di\-men\-sion\-al compact metrizable spaces are inverse limits of finite $n$-dimensional CW complexes.
Just as CW complexes form a class of spaces that are amenable to deep analysis, so too are the non-com\-mu\-ta\-tive cell complexes of Theorem \ref{thm:ApproxCellA}.
Their tractable structure allows them to be used to prove Theorem \ref{thm:MainThmA}.
The author A.T.\ has also made use of Theorem \ref{thm:ApproxCellA}, in an argument that shows that $C(X,\mathcal Q)$ is not locally approximated by subhomogeneous $\mathrmm C^*$-algebras of topological dimension less than the dimension of $X$, where $\mathcal Q$ is the universal UHF algebra \cite{CXDecomp}.

Generalizing the result that cell complexes are absolute neighbourhood retracts, we demonstrate that our non-commutative cell complexes satisfy a restricted semiprojectivity condition (Theorem \ref{thm:CellRSHSemiproj}); this condition then plays a key role in the proof of Theorem \ref{thm:ApproxCellA}.
From the proof that a subhomogeneous $\mathrmm C^*$-algebra is locally approximated by non-commutative cell complexes, ideas are borrowed to also prove that any $\mathcal W$-stabilized, locally subhomogeneous $\mathrmm C^*$-algebra is an inductive limit of point--line algebras (Theorem \ref{thm:WASH}), where $\mathcal W$ is the stably projectionless $\mathrmm C^*$-algebra studied in \cite{Jacelon:R}; this result builds on the case of $\mathrmm C(X) \otimes \mathcal W$, which was handled by one of us (L.S.) in \cite{Santiago:DimRed}.

The second step in proving Theorem \ref{thm:MainThmA} is to show that the $\mathcal Z$-stabilization of a non-commutative cell complex has decomposition rank at most $2$.
In fact, in this part, we are able to directly reduce to the problem of the decomposition rank of $\mathcal Z$-stabilized commutative $\mathrmm C^*$-algebras (which was solved in \cite{TW:Zdr}), by turning coloured c.p.c.\ approximations (as in the definition of decomposition rank) of certain commutative $\mathrmm C^*$-algebras into coloured c.p.c.\ approximations of non-commutative cell complexes (Theorem \ref{thm:drBound}).
The commutative $\mathrmm C^*$-algebras arise in a sort of mapping cone construction involving the eigenvalue patterns of the gluing maps.
The cell structure is absolutely essential, as it allows us to relate these non-commutative $\mathrmm C^*$-algebras to such commutative algebras.

The latest results concerning particular values of the decomposition rank suggest that the only possible values in the simple case are $0,1,$ and $\infty$ (to be precise, this is what is shown in \cite{BBSTWW} in the case of compact extreme tracial boundary).
It is unclear whether the upper bound of two, proven here, is optimal for non-simple, $\mathcal Z$-stable, locally subhomogeneous $\mathrmm C^*$-algebras.
Our result shows, however, that this depends purely on whether a better bound can be found for $\mathrmm C(X) \otimes \mathcal Z$.
That is, if the bound of two found in \cite{TW:Zdr} is improved to one (which would be optimal, since decomposition rank zero implies AF), then the same estimate applies to $\mathcal Z$-stable, locally subhomogeneous $\mathrmm C^*$-algebras.

Building on the results here, the authors G.A.E.\ and Z.N., in collaboration with Guihua Gong and Huaxin Lin, have gone on to prove that simple, separable, unital, $\mathcal Z$-stable, locally subhomogeneous $\mathrmm C^*$-algebras are rationally tracially approximated by point--line algebras \cite{EGLN:ASH}, and are thereby classified by the results of \cite{GongLinNiu}.
Their result depends on both Theorems \ref{thm:MainThmA} and \ref{thm:ApproxCellA}.
It follows that every simple, separable, unital, $\mathcal Z$-stable, locally subhomogeneous $\mathrmm C^*$-algebra is an inductive limit of subhomogeneous $\mathrmm C^*$-algebras of topological dimension at most two (this does not constitute an alternative proof of Theorem \ref{thm:MainThmA} since it uses said theorem).

\subsection{Posterior results}
Since this paper was originally written, there have been substantial developments in the structure and classification of C*-algebras, some of which partially generalize the results of this paper.
We summarize the state-of-the art in relation to these results.

A more general classification result than \cite{EGLN:ASH} is contained in \cite{EGLN:Class2} (still using \cite{GongLinNiu}), where they 
prove that simple, separable, stably finite, unital $\mathrmm C^*$-algebras with finite decomposition rank are rationally tracially approximated by point--line algebras (and finite decomposition rank can be further weakened to finite nuclear dimension using \cite{TWW:qd}).
Improvements have also been made in establishing finite nuclear dimension and decomposition rank; in \cite{CETWW,CastillejosEvington}, it is shown that every simple nuclear $\mathcal Z$-stable C*-algebra has finite nuclear dimension, and finite decomposition rank provided the algebra is finite and all traces on it are quasidiagonal.
The results in this paper are still the only way to see that \emph{nonsimple} approximately subhomogeneous $\mathcal Z$-stable C*-algebras have finite decomposition rank (or even finite nuclear dimension).

\subsection{Acknowledgements}

A.T.\ had numerous long discussions with Wilhelm Winter about the main problem solved in this article.
We would like to thank Wilhelm Winter for these, and for comments on early versions of the article.
We would also like to acknowledge Rob Archbold, Etienne Blanchard, Huaxin Lin, Chris Phillips, and Stuart White for discussions and comments that helped to shape this paper.
Finally, we would like to thank the referees for helpful comments and suggestions.

G.A.E.\ has been supported by NSERC.
Z.N.\ has been supported by NSERC, a start-up grant from the University of Wyoming, and a Simons Foundation collaboration grant.
L.S.\ has been supported by the University of Toronto and a start-up grant from the University of Aberdeen.
A.T.\ has been supported by NSERC and a start-up grant from the University of Aberdeen.
All authors were supported by the Fields Institute through the ``Thematic program on abstract harmonic analysis, Banach and operator algebras.''
Work on this article advanced perceptibly at the EPSRC- and LMS-funded conference ``Classification, structure, amenability, and regularity'' in Glasgow, and at the BIRS workshop ``Dynamics and $\mathrmm C^*$-algebras: amenability and soficity.''

\renewcommand*{\thetheorem}{\roman{theorem}}
\setcounter{theorem}{0}
\numberwithin{theorem}{section}

\subsection{Preliminaries and notation}
\label{sec:Notation}

Let $A$ be a $\mathrmm C^*$-algebra.
Denote the positive cone of $A$ by $A_+$.
For $a,b \in A$ and $\e > 0$, write $a \approx_\e b$ to mean $\|a-b\|<\e$.
If $\mathcal F,B\subset A$ with $\mathcal F$ finite, and $\e > 0$, write $\mathcal F \subset_\e B$ to mean that for every $a \in \mathcal F$,
\begin{equation} \mathrm{dist}(a,B) < \e. \end{equation}

For a unital $\mathrmm C^*$-algebra $A$, let $U(A)$ denote the unitary group.
Write $M_m$ for $M_m(\mathbb C)$ and $U_m$ for $U(M_m(\mathbb C))$.

If $A,B$ are unital $\mathrmm C^*$-algebras, let us consider $\Hom(A,B)$, the space of \textit{unital} ${}^*$-homomorphisms from $A$ to $B$, with the point-norm topology.
Provided that $A$ is separable, this topology on $\Hom(A,B)$ is metrizable.

For a compact Hausdorff space $X$, there is a one-to-one correspondence between $\Hom(A,\mathrmm C(X,B))$ and $\mathrmm C(X,\Hom(A,B))$, defined as follows.
For $\phi \in \Hom(A,\mathrmm C(X,B))$, define $\hat\phi \in \mathrmm C(X,\Hom(A,B))$ by
\begin{equation} \hat\phi(x)(a) := \phi(a)(x),\quad x\in X,\ a \in A. \end{equation}
Thus, if $Y$ is a closed subspace of $X$ and $r_Y\colon \mathrmm C(X,B) \to \mathrmm C(Y,B)$ denotes the restriction map, then for $\phi \in \Hom(A,\mathrmm C(X,B))$,
\begin{equation} (r_Y \circ \phi)\hspace*{-3pt}\hat{\phantom{)}} = \hat\phi|_Y. \end{equation}

\begin{defn}[\cite{KirchbergWinter:CovDim}]
Let $A$ be a $\mathrmm C^*$-algebra and let $n \in \N$.
One says that the \textbf{decomposition rank} of $A$ is at most $n$ (in abbreviated form, $\dr A \leq n$) if, for every finite subset $\mathcal F$ of $A$ and every $\e > 0$, there exist finite dimensional $\mathrmm C^*$-algebras $F_0,\dots,F_n$ and c.p.c.\ maps
\begin{equation} A \labelledrightarrow{\psi} F_0 \oplus \cdots \oplus F_n \labelledrightarrow{\phi} A \end{equation}
such that $\phi|_{F_i}$ is orthogonality preserving (also called order zero), and $\phi(\psi(a)) \approx_\e a$ for every $a \in \mathcal F$.
\end{defn}

\begin{defn}
A $\mathrmm C^*$-algebra $A$ is \textbf{subhomogeneous} if there is a finite upper bound on the dimension of the irreducible representations of $A$.
\end{defn}

Let $A$ be a subhomogeneous $\mathrmm C^*$-algebra and denote by $\Prim(A)$ the space of primitive ideals of $A$ with the hull-kernel topology (see \cite[Chapter 3]{Pedersen:CstarBook}).
Then the subset
\begin{equation} \Prim_{\leq k}(A) := \{\ker\pi \mid \pi \in \Hom(A,M_{k'})\text{ is irreducible, }k' \leq k\} \end{equation}
is a closed subspace of $\Prim(A)$.
Moreover,
\begin{equation} \Prim_k(A) = \Prim_{\leq k}(A) \setminus \Prim_{\leq k-1}(A) \end{equation}
is a Hausdorff subspace of $\Prim(A)$, and the subquotient
\begin{equation} A|_{\Prim_k(A)} \end{equation}
is a $k$-homogeneous $\mathrmm C^*$-algebra (i.e., every irreducible representation has dimension exactly $k$).
(To define $A|_{\Prim_k(A)}$, recall that closed sets in $\Prim(A)$ correspond to ideals, in a containment-reversing fashion; thus, $A|_{\Prim_k(A)}$ is the quotient of the ideal corresponding to $\Prim_{\leq k-1}(A)$ by the ideal corresponding to $\Prim_{\leq k}(A)$.)

\begin{defn}
\label{def:TopDim}
If $A$ is a separable subhomogeneous $\mathrmm C^*$-algebra then the \textbf{topological dimension} of $A$ refers to the value
\begin{equation} \max_k \dim \Prim_k(A). \end{equation}
\end{defn}

We stick to the separable case in the above definition, due to the following subtlety in the non-separable case:
If a locally compact Hausdorff space $X$ is not second countable, the nuclear dimension of $C_0(X)$ is not equal to the covering dimension of $X$.
If the definition of Lebesgue covering dimension is modified by replacing ``an open set of $X$'' by ``a set of the form $f^{-1}(U)$ where $f\in C_0(X,\R)$ and $U\subseteq \R$ is open,'' then the resulting notion of ``dimension'' of $X$ agrees with the nuclear dimension of $C_0(X)$.
This notion of dimension would also be the correct one to use in Definition \ref{def:TopDim}, in order for, say Theorem \ref{thm:drSHOrig} to generalize to non-separable algebras.

Rob Archbold pointed out to us the following connection to the primitive ideal space (as a topological space without additional structure).
Brown and Pedersen gave another definition of topological dimension in \cite[2.2 (v)]{BrownPedersen:Limits}, valid for any $\mathrm C^*$-algebra, which agrees with the definition just given in the subhomogeneous case, by \cite[Proposition 2.4]{BrownPedersen:Limits}.
Incidentally, this shows that the topological dimension is a function of the primitive ideal space---although it is not the covering dimension of the primitive ideal space.

\begin{thm}[\cite{Winter:drSH}]
\label{thm:drSHOrig}
If $A$ is a separable subhomogeneous $\mathrmm C^*$-algebra then the topological dimension, nuclear dimension, and decomposition rank of $A$ coincide.
\end{thm}

As pointed out above, separability is not needed, provided that topological dimension is defined appropriately in the non-separable case.
In Corollary \ref{cor:drSH}, we give an alternative proof of this theorem.

\begin{defn}
Let $\mathcal C$ be a class of $\mathrmm C^*$-algebras and let $A$ be a $\mathrmm C^*$-algebra.
Let us say that $A$ is \textbf{locally approximated} by (algebras in) $\mathcal C$ (or simply, locally $\mathcal C$) if, for every finite subset $\mathcal F$ of $A$ and every $\e > 0$, there exists a subalgebra $C \subseteq A$ such that $C \in \mathcal C$ and $\mathcal F \subset_\e C$.
In particular, \textbf{locally subhomogeneous} means locally approximated by the class of subhomogeneous $\mathrmm C^*$-algebras.
\end{defn}

Here is a well-known result which is clear from the definition of decomposition rank.

\begin{prop}
\label{prop:drLocal}
If $A$ is locally approximated by a class $\mathcal C$ then
\begin{equation} \dr(A) \leq \sup_{C \in \mathcal C} \dr(C). \end{equation}
\end{prop}

Pull-backs exist in the category of $\mathrmm C^*$-algebras, and in many places we shall use an explicit realization of them, which we describe now.
Let $A,B,C$ be $\mathrmm C^*$-algebras and let $\alpha\colon A \to C, \beta\colon B \to C$ be ${}^*$-ho\-mo\-mor\-phisms.
Set
\begin{equation}
D := \{(a,b) \in A \oplus B \mid \alpha(a) = \beta(b)\},
\end{equation}
and let $\pi_A\colon D \to A$ and $\pi_B\colon D \to B$ denote the first and second coordinate projections.
Then
\begin{equation}
\xymatrix{
D \ar[r]^-{\pi_A} \ar[d]_{\pi_B} & A \ar[d]^-{\alpha} \\
B \ar[r]_-{\beta} & C}
\end{equation}
is a pull-back diagram.

\section{Non-commutative cell complexes}
Non-commutative cell complexes are introduced in this section, and it is shown that all subhomogeneous algebras can be locally approximated by these.
This result can be viewed as a (non-commutatively) generalized---though weakened---version of the result that $n$-di\-men\-sion\-al compact metrizable spaces are inverse limits of $n$-dimensional simplicial complexes.
In particular, as in this topological result, our result has the feature that the topological dimension of the approximating subalgebras is controlled by the dimension of the given subhomogeneous algebra.

With regard to how it is weakened, the (non-commutative) result is more like saying that $n$-dimensional compact metrizable spaces are inverse limits of $n$-dimensional cell complexes.
As with the commutative case, the result provides building blocks for (approximately) subhomogeneous algebras that are much more tractable and amenable to further analysis (e.g., Theorem \ref{thm:MainThmA}, \cite{EGLN:ASH,Robert:NCCW,CXDecomp}), compared to general subhomogeneous algebras---or even to Phillips's recursive subhomogeneous algebras.

In proving this approximation theorem, a key result is that non-commutative cell complexes are semiprojective with respect to $\mathrmm C^*$-algebras of the form $\mathrmm C(X,M_m)$, for any fixed $m\in\N$.
This result is a generalization of the classical fact that finite cell complexes are absolute neighbourhood retracts; recalling the fact that a compact Hausdorff space $X$ is an absolute neighbourhood retract if and only if $\mathrmm C(X)$ is semiprojective with respect to the class of unital commutative $\mathrmm C^*$-algebras, we see that this classical result is the special case that the non-commutative cell complex is commutative and $m=1$.
In fact, the proof of the new semiprojectivity result is inspired by the proof in the classical case (see \cite[Appendix A]{Hatcher}, particularly Corollary A.10, which deals with CW complexes, although the arguments work for finite cell complexes).
The classical proof entails showing (i) that cell complexes embed into Euclidean space (equivalently, they have finite topological dimension) and (ii) that cell complexes are locally contractible.
Then \cite[Theorem A.7]{Hatcher} implies that they embed into Euclidean space as neighbourhood retracts, which implies that they are absolute neighbourhood retracts.

Here, using the correspondence between ${}^*$-homomorphisms $A\to \mathrmm C(X,M_m)$ and continuous maps $X \to \Hom(A,M_m)$, the semiprojectivity result is equivalent to showing that $\Hom(A,M_m)$ is an absolute neighbourhood retract, and as in the classical case, we appeal to \cite[Theorem A.7]{Hatcher}, requiring us to show that $\Hom(A,M_m)$ has finite topological dimension and is locally contractible.
Showing that $\Hom(A,M_m)$ is finite dimensional is fairly straightforward.
Showing that it is locally contractible is somewhat more involved, although the idea behind the commutative case underpins even that argument.

\subsection{Definition of non-commutative cell complexes}
\label{sec:NCcellDef}

Non-com\-mu\-ta\-tive cell complexes are defined as recursive subhomogeneous algebras (as in \cite[Definition 1.1]{Phillips:RSH}) for which the gluing pairs are always of the form $S^{n-1} \subset D^n$.
Here is a formal definition.

\begin{defn}
The class of (unital) \textbf{non-commutative (NC) cell complexes} is the smallest class $\mathcal C$ of $\mathrmm C^*$-algebras such that:

(i) every finite dimensional algebra is in $\mathcal C$; and

(ii) if $B \in \mathcal C$, $k,n \in \N$, $\phi\colon B \to \mathrmm C(S^{n-1},M_k)$ is a unital ${}^*$-ho\-mo\-mor\-phism, and $A$ is given by the pull-back diagram
\begin{equation}
\label{eq:NCcellDefpb}
\xymatrix{
A \ar[d] \ar[r] & \mathrmm C(D^n,M_k) \ar[d]^{f \mapsto f|_{S^{n-1}}} \\
B \ar[r]_-{\phi} & \mathrmm C(S^{n-1}, M_k),
}
\end{equation}
then $A \in \mathcal C$.
\end{defn}

\begin{remark}
\label{rmk:NCcellDef}
(i)
Of course, every NC cell complex can be constructed by finitely many iterated pull-backs as in \eqref{eq:NCcellDefpb}.
The topological dimension of an NC cell complex is precisely the largest value of $n$ such that $D^n$ appears in one of these pull-backs.
(Note that one need not take the minimum over all possible decompositions, since if at one stage, $D^n$ appears, then the algebra at this stage is a quotient of the final algebra, and its topological dimension is at least $n$.)

(ii)
A commutative $\mathrmm C^*$-algebra is an NC cell complex if and only if it is isomorphic to $C(X)$ for some finite cell complex $X$.

(iii)
The definition of NC cell complexes is closely related to the definition, due to Eilers-Loring-Pedersen, of NCCW complexes \cite[Section 2.4]{EilersLoringPedersen:NCCW}: all NCCW complexes are NC cell complexes, although the converse does not hold (see example below).

It is not hard to see that all NC cell complexes of topological dimension at most one are, in fact, one-dimensional NCCW complexes (these have elsewhere been called point--line algebras and occur prominently in the classification results of \cite{GongLinNiu} and \cite{Robert:NCCW}).
This fails, even in the commutative case, in dimension two.

For example, let $\alpha\colon S^1 \to D^2$ be a space-filling curve, and then define $X$ by the push-forward diagram
\begin{equation}
\xymatrix{
X & D^2 \ar[l] \\
D^2 \ar[u] & S^1, \ar[l]_-{\alpha} \ar@{^{(}->}[u]
}
\end{equation}
so that $X$ is a (finite) cell complex (i.e., $\mathrmm C(X)$ is an NC cell complex) but not a CW complex (i.e., $\mathrmm C(X)$ is not an NCCW complex---since every commutative NCCW complex is equal to $\mathrmm C(X)$ for a CW complex $X$).
\end{remark}

\subsection{Basic facts about $\Hom(A,M_m)$}
In preparation for the proof of our semiprojectivity result for NC cell complexes, we will prove a number of results about $\Hom(A,M_m)$, particularly applicable when $A$ is subhomogeneous.

A finite dimensional representation of a $\mathrmm C^*$-algebra decomposes as a direct sum of irreducible representations; we introduce notation to keep combinatorial track of such decompositions.

Fix $m \in \N$.
Define
\begin{align}
\notag T_m &:= \{\alpha = (k_1,\alpha_1,\dots,k_p,\alpha_p) \in \mathbb N_{>0}^{2p} \mid \\
&\qquad p \in \N\text{ and }k_1\alpha_1 + \cdots + k_p\alpha_p = m\}.
\end{align}
Let $A$ be a unital $\mathrmm C^*$-algebra.
Then for $\alpha=(k_1,\alpha_1,\dots,k_p,\alpha_p) \in T_m, u \in U_m$, and $\pi_i \in \Hom(A,M_{k_i})$ for $i=1,\dots,p$, define
\begin{align}
\notag \sigma_{\alpha,u,(\pi_1,\dots,\pi_p)} &:= \Ad(u) \circ \diag(\pi_1 \otimes 1_{\alpha_1},\dots,\pi_p \otimes 1_{\alpha_p}) \\
\label{eq:sigmaDef}
&\qquad \in \Hom(A,M_m).
\end{align}
For $\alpha=(k_1,\alpha_1,\dots,k_p,\alpha_p) \in T_m$, define
\begin{align}
\notag
H_\alpha := \{\sigma_{\alpha,u,(\pi_1,\dots,\pi_p)} \mid & u \in U_m\text{ and } \\
&\pi_i \in \Hom(A,M_{k_i})\text{ for }i=1,\dots,p\},
\end{align}
and let $H_\alpha^{\mathrm{pure}}$ denote the subset of $H_\alpha$ consisting of those $\sigma_{\alpha,u,(\pi_1,\dots,\pi_p)}$ for which $\pi_1,\dots,\pi_p$ are mutually inequivalent, irreducible representations.
For a $\mathrmm C^*$-algebra $A$ which is not assumed to be unital, use $\Hom_1(A,M_m)$ to denote the set of ${}^*$-homomorphisms whose image contains the unit (this coincides with $\Hom(A,M_m)$ if $A$ is unital), and define $\sigma_{\alpha,u,(\pi_1,\dots,\pi_p)}$, $H_\alpha$, and $H^{\mathrm{pure}}_\alpha$ as in the unital case, with $\Hom_1(A,M_m)$ in place of $\Hom(A,M_m)$.

\begin{prop}
\label{prop:sigmaCts}
Fix $\alpha = (k_1,\alpha_1,\dots,k_p,\alpha_p) \in T_m$ and a unital $\mathrmm C^*$-algebra $A$.
Then the map $U_m \times \Hom(A,M_{k_1}) \times \cdots \times \Hom(A,M_{k_p}) \to \Hom(A,M_m)$ given by
\begin{equation} (u,\pi_1,\dots,\pi_p) \mapsto \sigma_{\alpha,u,(\pi_1,\dots,\pi_p)} \end{equation}
is continuous.
\end{prop}

\begin{proof} Obvious. \end{proof}

\begin{lemma}
\label{lem:sigmaUnitaryKernel}
Let $A$ be a $\mathrmm C^*$-algebra.
Fix $\alpha = (k_1,\alpha_1,\dots,k_p,\alpha_p) \in T_m$, and fix pairwise inequivalent irreducible representations $\pi_i \in \Hom_1(A,M_{k_i})$ for $i=1,\dots,p$.
Then for $u,v \in U_m$,
\begin{equation} \sigma_{\alpha,u,(\pi_1,\dots,\pi_p)} \approx \sigma_{\alpha,v,(\pi_1,\dots,\pi_p)} \end{equation}
if and only if $uv^*$ is approximately contained in
\begin{equation} U(1_{k_1} \otimes M_{\alpha_1} \oplus \cdots \oplus 1_{k_p} \otimes M_{\alpha_p}) = 1_{k_1} \otimes U_{\alpha_1} \oplus \cdots \oplus 1_{k_p} \otimes U_{\alpha_p}. \end{equation}
To be more precise:

(i) Given $\mathcal F \subset A$ finite and $\e > 0$, there exists $\dl > 0$ such that if $u,v \in U_m$ and
\begin{equation} d(uv^*,1_{k_1} \otimes U_{\alpha_1} \oplus \cdots \oplus 1_{k_p} \otimes U_{\alpha_p}) < \dl \end{equation}
then
\begin{equation} \sigma_{\alpha,u,(\pi_1,\dots,\pi_p)}(a) \approx_\e \sigma_{\alpha,v,(\pi_1,\dots,\pi_p)}(a) \end{equation}
for all $a \in \mathcal F$; and

(ii)
Given $\e > 0$, there exists $\mathcal F \subset A$ finite and $\dl > 0$, such that if $u,v \in U_m$ and
\begin{equation} \sigma_{\alpha,u,(\pi_1,\dots,\pi_p)}(a) \approx_\dl \sigma_{\alpha,v,(\pi_1,\dots,\pi_p)}(a) \end{equation}
for all $a \in \mathcal F$, then
\begin{equation} 
\label{eq:sigmaUnitaryKernelEq}
d(uv^*,1_{k_1} \otimes U_{\alpha_1} \oplus \cdots \oplus 1_{k_p} \otimes U_{\alpha_p}) < \e. \end{equation}
\end{lemma}

\begin{proof}
Note that, by \eqref{eq:sigmaDef},
\begin{equation} \sigma_{\alpha,u,(\pi_1,\dots,\pi_p)}(a) \approx_\e \sigma_{\alpha,v,(\pi_1,\dots,\pi_p)}(a) \end{equation}
is equivalent to saying that
\begin{equation} \|[uv^*, \diag(\pi_1(a) \otimes 1_{\alpha_1},\dots,\pi_p(a) \otimes 1_{\alpha_p})]\| < \e. \end{equation}

(i): Given $\mathcal F$ finite and $\e > 0$, set $M:= \max_{a \in \mathcal F} \|a\|$ and then
\begin{equation} \dl := \frac\e{2M}. \end{equation}
Then if $u,v \in U_m$ and
\begin{equation} uv^* \approx_\dl w \in 1_{k_1} \otimes U_{\alpha_1} \oplus \cdots \oplus 1_{k_p} \otimes U_{\alpha_p}, \end{equation}
then for $a \in A$,
\begin{align}
\notag
& \hspace*{-10mm} uv^* \diag(\pi_1(a) \otimes 1_{\alpha_1},\dots,\pi_p(a) \otimes 1_{\alpha_p}) \\
\notag
&\qquad \approx_{\e/2} w \diag(\pi_1(a) \otimes 1_{\alpha_1},\dots,\pi_p(a) \otimes 1_{\alpha_p}) \\
\notag
&\qquad = \diag(\pi_1(a) \otimes 1_{\alpha_1},\dots,\pi_p(a) \otimes 1_{\alpha_p})w \\
&\qquad \approx_{\e/2} \diag(\pi_1(a) \otimes 1_{\alpha_1},\dots,\pi_p(a) \otimes 1_{\alpha_p})uv^*,
\end{align}
as required.

(ii): On the other hand, suppose we are given $\e > 0$.
Since being a unitary is a stable relation, it suffices to find $\mathcal F$ and $\dl>0$ such that, instead of \eqref{eq:sigmaUnitaryKernelEq}, we have
\begin{equation} 
d(uv^*,1_{k_1} \otimes M_{\alpha_1} \oplus \cdots \oplus 1_{k_p} \otimes M_{\alpha_p}) < \e. \end{equation}

By our hypothesis on $\pi_1,\dots,\pi_p$, we may choose $\mathcal F \subset A$ finite such that the set
\begin{equation} \{\diag(\pi_1(a) \otimes 1_{\alpha_1},\dots,\pi_p(a) \otimes 1_{\alpha_p}) \mid a \in \mathcal F\} \end{equation}
generates
\begin{equation}
D:= M_{k_1} \otimes 1_{\alpha_1} \oplus \cdots \oplus M_{k_p} \otimes 1_{\alpha_p}
\end{equation}
as a $\mathrmm C^*$-algebra.
Therefore (and using the fact that $D$ is finite dimensional), we may pick $\dl>0$ so that if $w$ is unitary and
\begin{equation} \|[w, \diag(\pi_1(a) \otimes 1_{\alpha_1},\dots,\pi_p(a) \otimes 1_{\alpha_p})]\| < \dl \end{equation}
for every $a \in \mathcal F$ then
\begin{equation} \|[w,x]\| < \e \end{equation}
for every contraction $x \in D$.
By applying the conditional expectation onto
\begin{equation} M_m \cap D' = 1_{k_1} \otimes M_{\alpha_1} \oplus \cdots \oplus 1_{k_p} \otimes M_{\alpha_p}, \end{equation}
we see that this implies that
\begin{equation} d(w,M_m \cap D') < \e, \end{equation}
as required.
\end{proof}

\begin{prop}
\label{prop:TOrder}
Let $\alpha = (k_1,\alpha_1,\dots,k_p,\alpha_p) ,\beta = (\el_1,\beta_1,\dots,\el_q,\beta_q) \in T_m$.
The following statements are equivalent:
\begin{enumerate}
\item $H_\alpha \subseteq H_\beta$;
\item There exists $S \in M_{p \times q}(\mathbb N)$ such that
\begin{equation} \begin{pmatrix} k_1 \\ \vdots \\ k_p \end{pmatrix} = S\begin{pmatrix} \el_1 \\ \vdots \\ \el_q \end{pmatrix}\quad\text{and}\quad
\begin{pmatrix} \beta_1 \\ \vdots \\ \beta_q \end{pmatrix} = S^{\text{tr}} \begin{pmatrix} \alpha_1 \\ \vdots \\ \alpha_p \end{pmatrix}. \end{equation}
\end{enumerate}
If $H^{\text{pure}}_\alpha$ is non-empty, then these are also equivalent to
\begin{enumerate}
\item[(iii)] $H^{\text{pure}}_\alpha \cap H_\beta \neq \emptyset$.
\end{enumerate}
Also, $H_\alpha = H_\beta$ if and only if $p=q$ and there exists a permutation $\rho$ on $\{1,\dots,p\}$ such that
\begin{equation} k_i = \el_{\rho(i)} \quad\text{and}\quad \alpha_i=\beta_{\rho(i)} \end{equation}
for all $i=1,\dots,p$.
\end{prop}

\begin{proof}
Straightforward.
\end{proof}

\begin{prop}
\label{prop:AllRepsPure}
Let $A$ be a $\mathrmm C^*$-algebra.
Then
$\Hom_1(A,M_m) = \bigcup_{\alpha \in T_m} H^{\mathrm{pure}}_\alpha$.
\end{prop}

\begin{proof}
Obvious.
\end{proof}

\subsection{Dimension of $\Hom(A,M_m)$}

\begin{lemma}
\label{lem:HomDecomp}
Let $A$ be a unital separable subhomogeneous $\mathrmm C^*$-algebra such that $A|_{\Prim_k(A)}$ is locally trivial for each $k$.
Let $m\in \N$.
Then there exist $q$ and a finite increasing sequence of sets,
\begin{equation} \emptyset = Z_0 \subset Z_1 \subset \cdots \subset Z_q = \Hom(A,M_m) \end{equation}
such that:

(i)
$Z_i$ is closed for each $i$; and

(ii)
For each $\sigma \in Z_i \setminus Z_{i-1}$, there exists a neighbourhood $W$ of $\sigma$ in $Z_i \setminus Z_{i-1}$ (that is, $W$ is relatively open in $Z_i$) which is homeomorphic to
\begin{equation} V \times W_1 \times \cdots \times W_p \end{equation}
where

(a) $V = 
U_m/G$ where $G$ is a Lie subgroup of $U_m$;

(b) $p \leq m$; and

(c) For each $j=1,\dots,m'$, $W_j$ is a relatively open subset of $\Prim_k(A)$ for some $k$.
\end{lemma}

\begin{proof}
Consider the equivalence relation $\sim$ on $T_m$ given by $\alpha \sim \beta$ if $H_\alpha = H_\beta$.
Enumerate one representative for each equivalence class as $\alpha^1,\dots,\alpha^q \in T_m$, in such a way that for each $i,j$, if $H_{\alpha^i} \subseteq H_{\alpha^j}$ then $i \leq j$.
Then, define
\begin{equation} Z_i := \bigcup_{j \leq i} H_{\alpha_j}. \end{equation}

Since $U_m$ and $\Hom(A,M_{k})$ are compact spaces, Proposition \ref{prop:sigmaCts} implies that $H_\alpha$ is closed for each $\alpha$.
Consequently, $Z_i$ is closed.

By Propositions \ref{prop:TOrder} and \ref{prop:AllRepsPure},
\begin{equation} Z_i \setminus Z_{i-1} = H^{\text{pure}}_{\alpha^i}. \end{equation}
So, in order to establish (ii), we must show that every element of $H^{\text{pure}}_\alpha$ has a neighbourhood in $H^{\text{pure}}_\alpha$ of the form described in (ii).

For this, fix $\alpha = (k_1,\alpha_1,\dots,k_p,\alpha_p)$ and let $\sigma_{\alpha,u,(\pi_1,\dots,\pi_p)} \in H^{\text{pure}}_\alpha$.
For each $i$, let $W_i$ be a neighbourhood of $[\pi_i]$ in $\Prim_{k_i}(A)$ such that $A|_{W_i}$ is trivial, and such that $W_i \cap W_j = \emptyset$ whenever $i \neq j$.
Since $A|_{W_i}$ is trivial, let $\dot W_i \subset \Hom(A,M_{k_i})$ be such that $\pi_i\in \dot W_i$ and $\pi \mapsto [\pi]$ is a homeomorphism from $\dot W_i$ onto $W_i$.

$W_i$ being open in $\Prim_{k_i}(A)$ means that $W_i \cup \Prim_{>k_i}(A)$ is open in $\Prim(A)$.
Using the fact that $A$ is separable, we may therefore find $a_i \in A$ such that, for each irreducible representation $\pi$ of $A$, $\pi(a_i) = 0$ if and only if $[\pi] \in \Prim_{\leq k_i}(A) \setminus W_i$.
Set
\begin{equation} V:=U_m/(1_{k_1} \otimes U_{\alpha_1} \oplus \cdots \oplus 1_{k_p} \otimes U_{\alpha_p}). \end{equation}

Define
\begin{equation} W := \{\sigma_{\alpha,v,(\pi_1',\dots,\pi_p')} \mid [v] \in V, \pi_i' \in \dot W_i\text{ for all }i=1,\dots,p\}. \end{equation}

Let us first show that $W$ is a neighbourhood of $\sigma_{\alpha,u,(\pi_1,\dots,\pi_p)}$ in $H^{\text{pure}}_\alpha$.
Let $\sigma_{\alpha,v,(\pi_1',\dots,\pi_p')} \in H^{\text{pure}}_\alpha$ be close to $\sigma_{\alpha,u,(\pi_1,\dots,\pi_p)}$.
If it is close enough, then
\begin{equation} \sigma_{\alpha,v,(\pi_1',\dots,\pi_p')}(a_i) \neq 0, \end{equation}
and therefore, for some $j$, $[\pi_j'] \in W_i$.
By the pigeonhole principle, permuting indices, we may assume that $[\pi_i'] \in W_i$ for each $i$.
Replacing $\pi_i'$ by a unitarily equivalent homomorphism (altering $v$ in the process), we may then assume that $\pi_i' \in \dot W_i$ for each $i$.

This concludes the proof that $W$ is a neighbourhood of $\sigma_{\alpha,u,(\pi_1,\dots,\pi_p)}$.

Now define $\Phi\colon V \times \dot W_1 \times \cdots \times \dot W_p \to W$ in the obvious way, by
\begin{equation} \Phi([v],\pi_1',\dots,\pi_p') = \sigma_{\alpha,v,(\pi_1',\dots,\pi_p')}. \end{equation}
This is continuous by Proposition \ref{prop:sigmaCts}, and it is clearly surjective.
Since the sets $W_1,\dots,W_p$ are disjoint, and by Lemma \ref{lem:sigmaUnitaryKernel}, $\Phi$ is also injective.

Let us show that that $\Phi$ is open.
Suppose that $\sigma_{\alpha,v,(\pi_1',\dots,\pi_p')} \approx \sigma_{\alpha,w,(\pi_1'',\dots,\pi_p'')}$.
Then the argument for why $W$ is a neighbourhood of $\sigma_{\alpha,u,(\pi_1,\dots,\pi_p)}$ shows that, up to a permutation, we have $\pi_i' \approx \pi_i''$ for all $i$.
Thus
\begin{equation}
\sigma_{\alpha,v,(\pi_1',\dots,\pi_p')} \approx \sigma_{\alpha,w,(\pi_1'',\dots,\pi_p'')} \approx \sigma_{\alpha,w,(\pi_1',\dots,\pi_p')},
\end{equation}
so by Lemma \ref{lem:sigmaUnitaryKernel}, $vw^*$ is approximately contained in $1_{k_1} \otimes U_{\alpha_1} \oplus \cdots \oplus 1_{k_p} \otimes U_{\alpha_p}$, i.e., $[v] \approx [w]$ in $V$.
This shows that $\Phi$ is open.

Since $W_i$ is homeomorphic to $\dot W_i$, we conclude that $W$ is homeomorphic to $V \times W_1 \times \cdots \times W_p$, as required.
\end{proof}

\begin{cor}
\label{cor:HomndDim}
Let $A$ be a unital separable subhomogeneous $\mathrmm C^*$-algebra of finite topological dimension, and let $m\in \N$.
Then $\Hom(A,M_m)$ has finite dimension.
\end{cor}

\begin{proof}
Suppose that $A$ has topological dimension at most $n$, so that $\Prim_k(A)$ has dimension at most $n$ for each $k$, and hence by \cite[Theorem 2.16]{Phillips:RSH}, it satisfies the hypotheses of Lemma \ref{lem:HomDecomp}.
Note also that any quotient of $U_m$ by a Lie subgroup has dimension at most $m^2$.
From Proposition \ref{lem:HomDecomp} and standard permanence properties for dimension, it follows directly that
\begin{equation} \dim \Hom(A,M_m) \leq m^2n^m. \end{equation}
\end{proof}

\subsection{Local contractibility of $\Hom(A,M_m)$}

To outline the argument here, consider the purely commutative case: $\Hom(A,\mathbb C)$ where $A$ is commutative (note that $\Hom(C(X),\mathbb C) \cong X$).
In order to show that a cell complex is locally contractible, one may use induction, showing that any point on the boundary of a newly-attached cell has small neighbourhoods, each of which can be retracted to a contractible neighbourhood of the old complex (see \cite[Proposition A.4]{Hatcher} and its proof).
This is the idea behind showing that $\Hom(A,M_m)$ is locally contractible when $A$ is a non-commutative cell complex, although even in the case that $A$ is commutative, the argument is more complicated: if $A=\mathrmm C(Y)$, then homomorphisms from $A$ to $M_m$ are parametrized by an $m$-tuple of points in $Y$ together with a unitary $u\in M_m$ (each homomorphism is the direct sum of point-evaluations, conjugated by a unitary), so part of the argument is that the space of $m$-multisets of points in $Y$ is locally contractible.
For more general non-commutative cell complexes, we have a similar picture, but with ``points in $Y$'' replaced by irreducible representations.
These need to be handled carefully, chiefly because of the non-Hausdorff nature of the space of irreducible representations.

The result, in the case that $A$ is commutative (or a matrix algebra over such), turns out to be a necessary stepping stone to the overall result, and we prove this special case now.
Recall that $\Hom_1(A,M_m)$ denotes the space of ${}^*$-homomorphisms whose image contains the unit.
We recall from topology that a \textbf{deformation retraction} of a space $X$ onto a subspace $Y$ is a continuous function
\begin{equation} r = (r_t)_{t\in [0,1]}\colon [0,1] \times X \to X
\end{equation}
such that $r_0 = \mathrm{id}_X$, $r_1(X) = Y$, and $r_1|_Y = \mathrm{id}_Y$.
Thus, a space $X$ is contractible to a point $x_0$ exactly when there exists a deformation retraction from $X$ to $\{x_0\}$.

\begin{lemma}
\label{lem:HomXMkLocContr}
Let $X$ be a locally compact Hausdorff space.
If $X$ is locally contractible then so also is $\Hom_1(\mathrmm C_0(X,M_k),M_{sk})$ for any $k,s \in \N$.
\end{lemma}

\begin{proof}
Using $\alpha := (1,1,\dots,1,1) \in \N^{2s}$, note that $\Hom_1(\mathrmm C_0(X,M_k),M_{sk}) = H_\alpha$.
To keep notation more concise, write
\begin{align}
\check\sigma_{u,(x_1,\dots,x_s)} &:= \sigma_{\alpha,u,(\ev_{x_1},\dots,\ev_{x_s})}
\in \Hom_1(\mathrmm C_0(X,M_k),M_{ks})
\end{align}
for $x_1,\dots,x_s \in X$ and $u \in U_{sk}$.

Consider a point $\sigma = \check\sigma_{u,(x_1,\dots,x_s)} \in \Hom_1(\mathrmm C_0(X,M_k),M_{sk})$.
Choose, for each $x \in \{x_1,\dots,x_s\}$, an arbitrarily small contractible neighbourhood $V_x$ of $x$, such that $V_x \cap V_y = \emptyset$ whenever $x \neq y$.
Fix a deformation retraction $\alpha_x\colon [0,1] \times V_x \to V_x$ of $V_x$ onto $\{x\}$, 
(Note that, by our choice of notation, if $x_i=x_j$ then $V_{x_i}=V_{x_j}$ and $\alpha_{x_i}=\alpha_{x_j}$.)
Let $\e > 0$ be small (to be determined).

Define $\tilde U \subset U_{sk} \times X^s$ to consist of those points $(v,(y_1,\dots,y_s))$ for which there exist a unitary $u' \in U_{sk}$ such that
\begin{enumerate}
\item $\sigma = \check\sigma_{u',(x_1,\dots,x_s)}$;
\item $\|v-u'\| < \e$; and
\item $y_i \in V_{x_i}$ for each $i$.
\end{enumerate}
The set
\begin{equation} U:= \{\check\sigma_{v,(y_1,\dots,y_s)} \in \Hom_1(\mathrmm C_0(X,M_k),M_{sk}) \mid (v,(y_1,\dots,y_s)) \in \tilde U\} \end{equation}
is an open neighbourhood of $\alpha$, and by making $V_1,\dots,V_s$ and $\e$ small enough, we make this neighbourhood arbitrarily small.

Define $\widetilde W \subset U_{sk} \times X^s$ to consist of those points $(v,(x_1,\dots,x_s))$ for which there exist a unitary $u' \in U_{sk}$ such that
\begin{enumerate}
\item $\sigma = \check\sigma_{u',(x_1,\dots,x_s)}$; and
\item $\|v-u'\| < \e$,
\end{enumerate}
and set 
\begin{equation} W:= \{\check\sigma_{v,(x_1,\dots,x_s)} \in \Hom_1(\mathrmm C_0(X,M_k),M_{sk}) \mid (v,(x_1,\dots,x_s)) \in \widetilde W\}. \end{equation}

Define $\tilde\beta = (\tilde\beta_t)\colon [0,1] \times \tilde U \to \tilde U$ as follows.
Given $(v,(y_1,\dots,y_s)) \in \tilde U$, 
define
\begin{equation} \tilde\beta_t(v,(y_1,\dots,y_s)) = (v,(\alpha_{x_1}(t, y_1),\dots,\alpha_{x_s}(t,y_s))), \end{equation}
where we recall that $\alpha_x$ is a deformation retract of $V_x$ onto $\{x\}$.
It is not hard to see that $\tilde\beta$ is a deformation retraction of $\tilde U$ onto $\widetilde W$.
It is also not hard to see that, via the surjection $\tilde U \to U$, $\tilde\beta$ induces a continuous map $\beta:[0,1] \times U \to U$, which is therefore a deformation retraction of $U$ onto $W$.

Next, define $\dot W \subset U_{sk} \times X^s$ to consist of those points $(v,(x_1,\dots,x_s))$ such that $\|v-u\| < \e$.
Note that $\check\sigma$ induces a homeomorphism of $\dot W$ onto $W$.

For unitaries $u,v$ and $t \in [0,1]$, define
\begin{equation} \lambda_t(u,v) := tu + (1-t)v. \end{equation}
If $u$ and $v$ are sufficiently close then $\lambda_t(u,v)$ is invertible, and we may define
\begin{equation} \theta_t(u,v) := \lambda_t(u,v)(\lambda_t(u,v)^*\lambda_t(u,v))^{-1/2}, \end{equation}
which is unitary.
Thus $(\theta(u,v))_{t\in [0,1]}$ is a path of unitaries from $u$ to $v$.

For $(v,(x_1,\dots,x_s)) \in \dot W$, define $\dot\gamma = (\dot\gamma_t)\colon [0,1] \times \dot W \to \dot W$ by
\begin{equation} \dot\gamma_t(v,(x_1,\dots,x_s)) = (\theta_t(u,v),(x_1,\dots,x_s)) \end{equation}
(assume $\e$ is sufficiently small so that $\theta_t(u,v)$ is always defined.)
This is a deformation retraction of $\dot W$ onto $\{(u,(x_1,\dots,x_s))\}$, so by the homeomorphism $\dot W \to W$, we obtain a deformation retraction $\gamma$ of $W$ onto $\{\sigma\}$.

Combining the deformation retracts $\beta$ and $\gamma$ provides a deformation retraction of $U$ onto $\{a\}$, as required.
\end{proof}

The next lemma generalizes the commutative argument that, for points on the boundary of a cell attached to a cell complex, every relative neighbourhood in the old cell complex is a deformation retract of a neighbourhood in the enlarged cell complex.
This lemma does not, in itself, imply that $\Hom(A,M_m)$ is locally contractible when $A$ is an NC cell complex, since a ${}^*$-homomorphism $A \to M_m$ can contain a mix of irreducible representations from the boundary of the new cell, the interior of the new cell, and purely from the old NC cell complex.

We continue to use $\Hom_1(A,M_m)$ to denote the set of ${}^*$-homomorphisms $A \to M_m$ whose images contain the unit of $M_m$.

\begin{lemma}
\label{lem:GluingDefRetract}
Let $A$ be given by the pull-back diagram
\begin{equation}
\xymatrix{
A \ar[d]_{\lambda} \ar[r]^-{\rho} & \mathrmm C_0(X \times [0,1),M_k) \ar[d]^{f \mapsto f|_{X \times \{0\}}} \\
B \ar[r]_-{\phi} & \mathrmm C(X,M_k),
}
\end{equation}
where $B$ is a unital $\mathrmm C^*$-algebra and $\phi$ is a unital ${}^*$-homomorphism.
Let $U$ be an open set in $\Hom(B,M_m)$.
Then there exists an open set $V$ in $\Hom_1(A,M_m)$ such that
\begin{equation}
\label{eq:GluingDefRetractURestricts}
 U \circ \lambda = V \cap (\Hom(B,M_m) \circ \lambda),
\end{equation}
and this set is a deformation retract of $V$.
\end{lemma}

\begin{proof}
To keep notation more concise, for $\el \in \N$ such that $\el k \leq m$, denote
\begin{equation} \alpha_\el := (1,1,\dots,1,1,m-k\el,1) \in T_m \end{equation}
(where $1,1$ appears $\el$ times), and for $v_1,\dots,v_\el \in X \times [0,1)$, $\pi_B \in \Hom(B,M_{m-k\el})$, and $u \in U_m$, denote
\begin{align}
\notag
\check\sigma_{u,(v_1,\dots,v_\el),\pi_B} &:= \sigma_{\alpha_\el,u,(\ev_{v_1} \circ \rho,\dots,\ev_{v_\el} \circ \rho,\pi_B \circ \lambda)} \\
\notag
&= \Ad(u) \circ \diag(\ev_{v_1} \circ \rho,\dots,\ev_{v_\el} \circ \rho,\pi_B \circ \lambda) \\
&\in \Hom_1(A,M_m).
\end{align}

Define
\begin{align}
\notag
 \hspace*{-10mm} \dot U := \{&(u,(x_1,\dots,x_\el), \pi_B) \mid \el \in \mathbb N, \pi_B \in \Hom(B,M_{m-k\el}), \\
&\quad x_1,\dots,x_\el \in X, u \in U_m,\check\sigma_{u,((x_1,0),\dots,(x_\el,0)),\pi_B} \in U \circ \lambda\}.
\end{align}
Note that
\begin{equation} \check\sigma_{u,((x_1,0),\dots,(x_\el,0)),\pi_B} = \sigma_{\alpha_\el,u,(\ev_{x_1} \circ \phi,\dots,\ev_{x_\el} \circ \phi,\pi_B)} \circ \lambda, \end{equation}
where $\sigma_{\alpha_\el,u,(\ev_{x_1}\circ \phi,\dots, \ev_{x_\el} \circ \phi, \pi_B)}$ is in $\Hom(B,M_m)$.

Set
\begin{align}
\notag
\dot V := \{&(u,((x_1,s_1),\dots,(x_\el,s_\el)),\pi_B) \mid \el\in \mathbb N, \\
&\qquad (u,(x_1,\dots,x_\el),\pi_B) \in \dot U, s_1,\dots,s_\el \in [0,1)\}.
\end{align}
Define $\dot\theta=(\dot\theta_t)\colon [0,1] \times \dot V \to \dot V$ by
\begin{align}
\notag
&\vspace*{-10mm} \dot\theta_t(u,((x_1,s_1),\dots,(x_\el,s_\el)),\pi_B) \\
&\qquad := (u,((x_1,ts_1),\dots,(x_\el,ts_\el)),\pi_B).
\end{align}
Define
\begin{align}
\notag
V := \{&\check\sigma_{u,((x_1,s_1),\dots,(x_\el,s_\el)),\pi_B} \mid \\
&\qquad (u,((x_1,s_1),\dots,(x_\el,s_\el)),\pi_B) \in \dot V\}.
\end{align}
We shall check two things: (i) that $V$ is open, and (ii) that $\dot\theta$ induces a homotopy $\theta$ on $V$, by
\begin{equation} \theta_t(\check\sigma_{u,((x_1,s_1),\dots,(x_\el,s_\el)),\pi_B}) = \check\sigma_{\dot\theta_t(u,((x_1,s_1),\dots,(x_\el,s_\el)),\pi_B)}, \quad t \in [0,1]. \end{equation}
It is obvious from the definition that \eqref{eq:GluingDefRetractURestricts} holds and that $\theta$ is a deformation retraction onto $U$.

(i):
To see that $V$ is open, let $\check\sigma_{u,((x_1,s_1),\dots,(x_\el,s_\el)),\pi_B} \in V$ for some $(u,((x_1,s_1),\dots,(x_\el,s_\el)),\pi_B) \in \dot V$.
By possibly modifying the choice of lift in $\dot V$, we may arrange that $\pi_B$ has no subrepresentation that factors through $\phi$, that $s_1,\dots,s_{\el_1}>0$, and that $s_{\el_1+1}=\dots=s_{\el}=0$.
Suppose that $\sigma$ is a representation near to $\check\sigma_{u,((x_1,s_1),\dots,(x_\el,s_\el)),\pi_B}$.
Let us write $\sigma$ as
\begin{equation} \check\sigma_{v,((y_1,t_1),\dots,(y_p,t_p)),\pi_B'}. \end{equation}
Assuming that $\sigma$ is near enough, taking into account how far each subrepresentation of $\pi_B$ is from a representation that factors through $\phi$, and by rearranging, we may assume that $\el_1 \leq p \leq \el$ and 
\begin{equation}
\label{eq:GluingDefRetractApprox0}
(y_i,t_i) \approx (x_i,s_i), \quad i \leq p.
\end{equation}
It follows that, for all $i,j$, if $(x_i,s_i) \neq (x_j,s_j)$ then $(y_i,t_i) \neq (y_j,t_j)$, and that
\begin{equation}
\label{eq:GluingDefRetractApprox1}
 \check\sigma_{v,((x_1,s_1),\dots,(x_p,s_p)),\pi_B'} \approx \check\sigma_{u,((x_1,s_1),\dots,(x_\el,s_\el)),\pi_B}. \end{equation}

By the same argument as used to prove Lemma \ref{lem:sigmaUnitaryKernel}, there exists a unitary $u' \in U_m$ such that
\begin{align}
\notag
\check\sigma_{u,((x_1,s_1),\dots,(x_\el,s_\el)),\pi_B} &=
\check\sigma_{u',((x_1,s_1),\dots,(x_\el,s_\el)),\pi_B} \quad \text{and} \\
\label{eq:GluingDefRetractApprox2}
(u')^*v &\approx \left( \begin{array}{cc} 1_{kp}&0 \\ 0&w_2\end{array} \right) =: w
 \end{align}
for some $w_2 \in U_{m-kp}$.
Thus,
\begin{eqnarray}
\notag
\check\sigma_{u',((x_1,s_1),\dots,(x_\el,s_\el)),\pi_B}
&\stackrel{\eqref{eq:GluingDefRetractApprox1}}\approx& \check\sigma_{v,((x_1,s_1),\dots,(x_p,s_p)),\pi_B'} \\
&\stackrel{\eqref{eq:GluingDefRetractApprox2}}\approx& \check\sigma_{u'w,((x_1,s_1),\dots,(x_p,s_p)),\pi_B'},
\end{eqnarray}
so conjugating by $(u')^*$, and looking at $(m-kl) \times (m-kl)$ block on the bottom-right, one obtains
\begin{equation}
\label{eq:GluingDefRetractApprox3}
 \diag(\ev_{x_{p+1}} \circ \phi,\dots,\ev_{x_\el} \circ \phi,\pi_B) \approx \Ad(w_2) \circ \pi_B'. \end{equation}
It follows that
\begin{eqnarray}
\notag
\sigma_{\alpha_p,v,(\ev_{y_1} \circ \phi, \dots, \ev_{y_p} \circ \phi),\pi_B'} 
&\stackrel{\eqref{eq:GluingDefRetractApprox2}}\approx& \sigma_{\alpha_p,u',(\ev_{y_1} \circ \phi, \dots, \ev_{y_p} \circ \phi),\Ad(w_2) \circ \pi_B'} \\
\notag
&\stackrel{\eqref{eq:GluingDefRetractApprox0}}\approx& \sigma_{\alpha_p,u',(\ev_{x_1} \circ \phi, \dots, \ev_{x_p} \circ \phi),\Ad(w_2) \circ \pi_B'} \\
&\stackrel{\eqref{eq:GluingDefRetractApprox3}}\approx& \sigma_{\alpha_l,u',(\ev_{x_1} \circ \phi, \dots, \ev_{x_l} \circ \phi),\pi_B} \in U.
\end{eqnarray}
These approximations can be made arbitrarily close by asking at the outset that $\sigma$ is sufficiently close to $\check\sigma_{u,((x_1,s_1),\dots,(x_l,s_l)),\pi_B}$.
It then follows that $\sigma_{\alpha_p,v,(\ev_{y_1} \circ \phi, \dots, \ev_{y_p} \circ \phi),\pi_B'} \in U$, i.e., $(v,(y_1,\dots,y_p),\pi_B') \in \dot U$.
Hence,
\begin{equation}
(v,((y_1,t_1),\dots,(y_p,t_p)),\pi_B') \in \dot V, \quad \text{i.e.,} \quad \sigma \in V.
\end{equation}
This establishes that $V$ is open.

(ii):
To check that $\dot\theta$ induces a homotopy $\theta=(\theta_t)_{t\in[0,1]}$ on $V$, suppose that 
\begin{equation} (u,((x_1,s_1),\dots,(x_\el,s_\el)),\pi_B),(v,((y_1,t_1),\dots,(y_p,t_p)),\pi_B') \in \dot V \end{equation}
give rise to the same element of $V$, i.e., that
\begin{align}
\check\sigma_{u,((x_1,s_1),\dots,(x_\el,s_\el)),\pi_B} = \check\sigma_{v,((y_1,t_1),\dots,(y_p,t_p)),\pi_B'}.
\end{align}
It is not hard to see that we may assume that $s_i,t_j > 0$ for all $i,j$.
Then it follows that 
\begin{equation} \{(x_1,s_1),\dots,(x_\el,s_\el)\} = \{(y_1,t_1),\dots,(y_p,t_p)\} \end{equation}
as multisets, and so $\el=p$.
Replacing a unitary by its composition with a permutation, if necessary, we may assume that $(x_i,s_i)=(y_i,t_i)$ for all $i$.
When this is true, it is clear from the formula for $\dot\theta_t$ that
\begin{equation} \dot\theta_t(u,(x_1,s_1),\dots,(x_\el,s_\el),\pi_B)) \end{equation}
and
\begin{equation} \theta_t(v,(y_1,t_1),\dots,(y_p,t_p),\pi_P)) \end{equation}
induce the same representation of $A$.
\end{proof}

Finally, we prove the local contractibility result.

\begin{prop}
\label{prop:NCCellLocContr}
Let $A$ be an NC cell complex.
Then $\Hom(A,M_m)$ is locally contractible.
\end{prop}

\begin{proof}
Using induction, we need only prove that if $\Hom(B,M_m)$ is locally contractible for every $m$ and $A$ is given by the pull-back diagram
\begin{equation}
\xymatrix{
A \ar[d]_\lambda \ar[r]^-{\rho} & \mathrmm C(D^n,M_k) \ar[d]^{f \mapsto f|_{S^n \times \{0\}}} \\
B \ar[r]_-{\phi} & \mathrmm C(S^{n-1},M_k)
}
\end{equation}
then $\Hom(A,M_m)$ is also locally contractible for every $m$.

Therefore, let us take a point $\pi \in \Hom(A,M_m)$.
Up to unitary equivalence, $\pi$ decomposes as
\begin{equation} \pi = \diag(\pi_B \circ \lambda, \pi_D \circ \rho), \end{equation}
where $\pi_B \in \Hom(B,M_{m_B})$ and $\pi_D \in \Hom_1(\mathrmm C_0(D^n \setminus S^{n-1},M_k),M_{m_D})$, where $m_B + m_D=m$.
Since $\pi_D$ is a direct sum of finitely many irreducible representations, it actually comes from $\Hom_1(\mathrmm C_0(Y_D,M_k),M_{m_D})$, for some compactly-contained open subset $Y_D$ of $D^n \setminus S^{n-1}$.
Let $Y_B$ be an open neighbourhood of $S^{n-1}$ in $D^n \setminus \overline{Y}$, such that $Y_B$ can be identified with $S^{n-1} \times [0,1)$, compatible with the identification of $S^{n-1}$ with $S^{n-1} \times \{0\}$.

By Lemma \ref{lem:HomXMkLocContr}, there exists a contractible neighbourhood $U_D$ of $\pi_D$ in 
\begin{equation} \Hom_1(\mathrmm C_0(Y_D,M_k),M_{m_D}); \end{equation}
let $\theta^D = (\theta^D_t)\colon [0,1] \times U_D \to U_D$ be a deformation retraction of $U_D$ onto $\{\pi_D\}$.
By the inductive hypothesis, $\pi_B$ has a neighbourhood in $\Hom(B,M_{m_B})$ which is contractible to $\pi_B$.
Combining the deformation retraction provided by Lemma \ref{lem:GluingDefRetract} with this contraction, we see that there exists a neighbourhood $U_B$ of $\pi_B \circ \lambda$ in 
\begin{equation} \Hom_1(\mathrmm C_0(Y_B,M_k) \oplus_{\mathrmm C(S^{n-1},M_k)} B,M_{m_B}) \end{equation}
and a deformation retraction $\theta^B = (\theta^B_t)\colon [0,1] \times U_B \to U_B$ of $U_B$ onto $\{\pi_B \circ \lambda\}$.

Define $U$ to be the set of ${}^*$-homomorphisms $A \to M_m$ of the form
\begin{equation} \Ad(u) \circ \diag(\pi_B', \pi_D' \circ \rho), \end{equation}
where $\pi_B' \in U_B$, $\pi_D' \in U_D$, and $u \in U_m$ such that $\|u-v\|<\e$ for some $v \in U_m$ such that
\begin{equation} \Ad(v) \circ \pi = \pi. \end{equation}

It is not hard to see that $U$ is an open neighbourhood of $\pi$, and by making $U_D, Y_B,$ and $\e$ sufficiently small, we turn $U$ into an arbitrarily small neighbourhood of $\sigma$.
In particular, by making $U$ small enough, we ensure that if
\begin{equation} \Ad(u) \circ \diag(\pi_B', \pi_D' \circ \rho) = \Ad(v) \circ \diag(\pi'_B, \pi_D \circ \rho) \in U \end{equation}
where $\pi_B' \in U_B$ and $\pi_D' \in U_D$ then
\begin{equation} \Ad(u) \circ \pi = \Ad(v) \circ \pi. \end{equation}

To show that $U$ is a contractible neighbourhood of $\pi$, we imitate the steps of the proof of Lemma \ref{lem:HomXMkLocContr}: first, we produce a deformation retraction of $U$ onto the set $W$ of all ${}^*$-homomorphisms $A \to M_m$ of the form
\begin{equation} \Ad(u) \circ \diag(\pi_B \circ \lambda, \sigma_D \circ \rho), \end{equation}
where $u \in U_m$ is such that $\|u-v\|<\e$ for some $v \in U_m$ such that
\begin{equation} \Ad(v) \circ \sigma = \sigma, \end{equation}
and then we contract $W$ onto $\sigma$.
\end{proof}

\subsection{Restricted semiprojectivity of non-commutative cell complexes}

We are now prepared to show the following, which says that NC cell complexes are semiprojective with respect to the class of $\mathrmm C^*$-algebras of the form $\mathrmm C(X,M_m)$, for any fixed $m$.
(Note, of course, that most NC cell complexes do not belong to this class.)

\begin{thm}
\label{thm:CellRSHSemiproj}
Let $A$ be an NC cell complex.
Let $X$ be a compact Hausdorff space, and $Y$ a closed subspace.
Let $\phi\colon A \to \mathrmm C(Y,M_m)$ be a unital ${}^*$-homomorphism.
Then there exists a closed neighbourhood $Z$ of $Y$ in $X$ such that $\phi$ lifts to a $^*$-homomorphism $\tilde\phi\colon A \to \mathrmm C(Z,M_m)$, where by ``lifts,'' it is meant that $\phi = \tilde\phi(\cdot)|_Y$.
\end{thm}

\begin{remark*}
Of course, the conclusion of this theorem, in the case that $X$ is metrizable, is the same as saying that if $B:=\mathrmm C(X,M_m)$, and $I_n$ is an increasing sequence of ideals of $B$, then any unital ${}^*$-homomorphism $A \to B/\overline{\bigcup I_n}$ lifts to a $^*$-homomorphism $A \to B/I_n$ for some $n$.
(It is stronger in the non-metrizable case.)
\end{remark*}

\begin{proof}[Proof of Theorem \ref{thm:CellRSHSemiproj}]
Recall the definition of $\hat\phi \in \mathrmm C(Y,\Hom(A,M_m))$ from Section \ref{sec:Notation}, and that the existence of the lift $\tilde\phi$ corresponds to extending $\hat\phi$ to a continuous map $\tilde{\hat\phi}\colon Z \to \Hom(A,M_m)$.
To prove the theorem, we must therefore show that $\Hom(A,M_m)$ is an absolute neighbourhood extensor with respect to the class of compact Hausdorff spaces (as defined in \cite[II.2]{Hu:RetractBook}), equivalently, an absolute neighbourhood retract \cite[Lemma 5.1]{Hu:RetractBook}.

Since $\Hom(A,M_m)$ is finite-dimensional (Corollary \ref{cor:HomndDim}) and metrizable, it embeds into $\R^n$ for some $n$ \cite[Theorem 1.11.4]{Engelking:DimTheory}.
It is, in addition, compact, since it can be identified with a weak$^*$-closed, bounded subset of $(A \otimes M_m)^*$ (by identifying a ${}^*$-homomorphism $\phi\colon A \to M_m$ with the functional that maps $a \otimes e_{i,j}$ to the $(i,j)$-th entry of the matrix $\phi(a)$).
Also, it is locally contractible by Proposition \ref{prop:NCCellLocContr}.
Therefore, by \cite[Theorem A.7]{Hatcher}, it is an absolute neighbourhood retract. 
\end{proof}

The following corollary says that the relations defining an NC cell complex $A$ are stable in the class of $\mathrmm C^*$-algebras of the form $\mathrmm C(X,M_m)$ (for fixed $m$).

\begin{cor}
\label{cor:CellRSHStableRelations}
Given an NC cell complex $A$, a finite set $\mathcal F \subset A$, a tolerance $\e > 0$ and an integer $m \in \N$, there exists a finite set $\mathcal G \subset A$ and $\dl > 0$ such that the following holds: if $X$ is a compact Hausdorff space and $\phi\colon A \to \mathrmm C(X,M_m)$ is a ${}^*$-linear map that is $(\mathcal G,\dl)$-multiplicative, then there is a ${}^*$-homomorphism $\bar\phi\colon A \to \mathrmm C(X,M_m)$ such that
\begin{equation} \phi(a) \approx_\e \bar\phi(a) \end{equation}
for all $a \in \mathcal F$.
\end{cor}

\begin{proof}
This is a standard argument, involving lifting a $^*$-homomorphism
\begin{equation} A \to \prod_n \mathrmm C(X_n,M_m)/\bigoplus_n \mathrmm C(X_n,M_m). \end{equation}
See \cite[Theorem 4.1.4]{Loring:book}.
\end{proof}

\begin{cor}
\label{cor:CellRSHCloseRepsHomotopic}
Let $A$ be an NC cell complex and let $m\in\N$.
For any finite subset $\mathcal F \subset A$ and $\e > 0$ there exists a finite subset $\mathcal G \subset A$ and $\dl > 0$ such that, if $X$ is a compact Hausdorff space and 
\begin{equation} \phi_0,\phi_1\colon A \to \mathrmm C(X,M_m) \end{equation}
are ${}^*$-homomorphisms such that
\begin{equation} \phi_0(a) \approx_\dl \phi_1(a) \end{equation}
for all $a \in \mathcal G$, then there exists a homotopy of ${}*$-homomorphisms, $\theta = (\theta_t)\colon A \to \mathrmm C([0,1]) \otimes \mathrmm C(X,M_m)$, such that $\theta_i = \phi_i$ for $i=0,1$ and
\begin{equation} \theta_t(a) \approx_\e \phi_0(a) \end{equation}
for all $a \in \mathcal F$ and $t\in[0,1]$.
\end{cor}

\begin{proof}
This argument is probably well-known, although we were unable to find a reference.
It is a proof by contradiction.
Suppose that the result is false for $\mathcal F$ and $\e >0$.
Let $(a_i)$ be a dense sequence in $A$.

Then for every $n$ we may find a compact space $X_n$ and ${}^*$-homomorphisms
\begin{equation} \phi_0^{(n)},\phi_1^{(n)}\colon A \to \mathrmm C(X_n,M_k) \end{equation}
such that
\begin{equation}
\label{eq:CellRSHCloseRepsHomotopicPhiClose}
\phi_0^{(n)}(a_i) \approx_{\frac1n} \phi_1^{(n)}(a_i), \quad i=1,\dots,n \end{equation}
yet there is no $\theta=(\theta_t)\colon A \to \mathrmm C([0,1]) \otimes \mathrmm C(X_n,M_k)$ such that $\theta_i=\phi_i^{(n)}$ for $i=0,1$ and
\begin{equation} \theta_t(a) \approx_\e \phi_0^{(n)}(a) \end{equation}
for all $x \in \mathcal F$ and $t\in[0,1]$.

Define a c.p.c.\ map $\Phi^{(n)}\colon A \to \mathrmm C([0,1]) \otimes \mathrmm C(X_n,M_k)$ by
\begin{equation}
\Phi^{(n)} := \id_{[0,1]} \otimes \phi_1^{(n)} + (1-\id_{[0,1]}) \otimes \phi_0^{(n)}.
\end{equation}
This is obviously not multiplicative, but since the $\phi_i^{(n)}$ are multiplicative and by \eqref{eq:CellRSHCloseRepsHomotopicPhiClose}, it induces a ${}^*$-homomorphism
\begin{equation} \Phi\colon A \to \prod_n \mathrmm C([0,1]) \otimes \mathrmm C(X_n,M_k) / \bigoplus_n \mathrmm C_0((0,1)) \otimes \mathrmm C(X_n,M_k) \end{equation}
This $^*$-homomorphism lifts to a ${}^*$-homomorphism 
\[ \tilde\Phi\colon A \to \prod_{n \geq n_0} \mathrmm C([0,1]) \otimes \mathrmm C(X_n,M_k), \]
for some $n_0 \in \mathbb N$, giving a contradiction.
\end{proof}

\subsection{Approximating arbitrary subhomogeneous algebras}

In this section, we prove the following approximation result.

\begin{thm}[Theorem \ref{thm:ApproxCellA}]
\label{thm:ApproxCell}
Let $A$ be a unital, separable subhomogeneous algebra of topological dimension at most $n$.
Then $A$ is locally approximated by NC cell complexes of topological dimension at most $n$.
\end{thm}

\begin{remark*}
Note that NC cell complexes are, by definition, unital, and so we could not expect the result above to generalize to non-unital subhomogeneous algebras.
The result can, nonetheless, be applied to studying these non-unital algebras, by applying it to their unitizations.
\end{remark*}

We begin with a few technical lemmas.

\begin{lemma}
\label{lem:ApproxCellB}
Let $A$ be given by the pull-back diagram
\begin{equation}
\xymatrix{
A \ar[d]_\lambda \ar[r]^-{\rho} & \mathrmm C(X,M_k) \ar[d]^{f \mapsto f|_{X^{(0)}}} \\
B \ar[r]_-{\phi} & \mathrmm C(X^{(0)},M_k)
}
\end{equation}
where $X$ is a compact Hausdorff space and $X^{(0)}$ is a closed subspace.
Suppose that $B$ is locally approximated by subalgebras $B_n$.
Define $A_n$ by the pull-back diagram
\begin{equation}
\xymatrix{
A_n \ar[d] \ar[r] & \mathrmm C(X,M_k) \ar[d]^{f \mapsto f|_{X^{(0)}}} \\
B_n \ar[r]^-{\phi|_{B_n}} & \mathrmm C(X^{(0)},M_k).
}
\end{equation}
Then $A$ is locally approximated by the $A_n$.
\end{lemma}

\begin{proof}
Let $\mathcal F \subset A$ be a finite subset and let $\e >0$.
There exists $n \in \N$ such that, for each $a=(b,f) \in \mathcal F$, there exists $b_a \in B_n$ such that
\begin{equation} b \approx_\e b_a. \end{equation}
Then, since $\phi(b_a) \approx_\e f|_{X^{(0)}}$, using the definition of the quotient norm, there exists $g_a \in \mathrmm C(X,M_k)$ such that 
\begin{equation} g_a|_{X^{(0)}} = \phi(b_a) \quad \text{and}\quad g_a \approx_\e f. \end{equation}
Then, $(b_a,g_a) \in A_n$ and $a \approx_\e (b_a,g_a)$.
\end{proof}

\begin{lemma}
\label{lem:CellRSHChar}
Let $A$ be given by the pull-back diagram
\begin{equation}
\xymatrix{
A \ar[d]_\lambda \ar[r]^-{\rho} & \mathrmm C(X,M_k) \ar[d]^{f \mapsto f|_{X^{(0)}}} \\
B \ar[r]_-{\phi} & \mathrmm C(X^{(0)},M_k),
}
\end{equation}
where $B$ is an NC cell complex of topological dimension at most $n$, $X$ is a CW complex of dimension at most $n$, and $X^{(0)}$ is a subcomplex.
Then $A$ is an NC cell complex of topological dimension at most $n$.
\end{lemma}

\begin{proof}
We prove this by induction on the dimension, $m$, of $X$.

Removing the interior of $X^{(0)}$ from $X$ (and from $X^{(0)}$) does not change $A$ at all, and allows $X^{(0)}$ to be replaced by its boundary (in $X$).
Thus, we may assume that $X^{(0)}$ is an $(m-1)$-dimensional subcomplex.
Let $Y \subset X$ be the $(m-1)$-skeleton of $X$, so that $X^{(0)}$ is a subcomplex of $Y$.

Define $B'$ and $\rho'$ by the following pull-back diagram:
\begin{equation}
\xymatrix{
B' \ar[d] \ar[r]^-{\rho'} & \mathrmm C(Y,M_k) \ar[d]^{f \mapsto f|_{X^{(0)}}} \\
B \ar[r]_-{\phi} & \mathrmm C(X^{(0)},M_k).
}
\end{equation}
By the inductive hypothesis, $B'$ is an NC cell complex of topological dimension at most $n$.
Moreover, it is not hard to see that the diagram
\begin{equation}
\xymatrix{
A \ar[d] \ar[r]^-{\rho} & \mathrmm C(X,M_k) \ar[d]^{f \mapsto f|_{Y}} \\
B' \ar[r]_-{\rho'} & \mathrmm C(Y,M_k)
}
\end{equation}
commutes and is a pull-back diagram.

This reduces the problem to the case that $X^{(0)}$ is the $m$-skeleton of $X$.
To deal with this case, we use induction on the number of $m$-cells in $X$.
In the base case of this induction, $X$ has no $m$-cells, and then $X=X^{(0)}$ and so $A=B$ and we are finished.

Otherwise, by singling out an $m$-cell of $X$, we obtain a pull-back diagram
\begin{equation}
\xymatrix{
\mathrmm C(X) \ar[d] \ar[r]^-{\rho} & \mathrmm C(D^m) \ar[d]^{f \mapsto f|_{S^{m-1}}} \\
\mathrmm C(X') \ar[r]_-{\alpha} & \mathrmm C(S^{m-1}),
}
\end{equation}
where $X'$ is a subcomplex of $X$ given by deleting the interior of one $m$-cell, and $\alpha$ is a ${}^*$-homomorphism.
This means that $X'$ contains $X^{(0)}$ and has one fewer $m$-cell than $X$ does.

Define $B'$ and $\rho'$ by the following pull-back diagram:
\begin{equation}
\xymatrix{
B' \ar[d] \ar[r]^-{\rho'} & \mathrmm C(X',M_k) \ar[d]^{f \mapsto f|_{X^{(0)}}} \\
B \ar[r]_-{\phi} & \mathrmm C(X^{(0)},M_k).
}
\end{equation}
By the inductive hypothesis, $B'$ is an NC cell complex of topological dimension at most $n$.
Moreover, it is not hard to see that the diagram
\begin{equation}
\xymatrix{
A \ar[d] \ar[rr] && \mathrmm C(D^n,M_k) \ar[d]^{f \mapsto f|_{S^{n-1}}} \\
B' \ar[rr]_-{(\alpha\otimes \mathrm{id}_{M_k}) \circ \rho'} && \mathrmm C(S^{n-1},M_k)
}
\end{equation}
is a pull-back diagram, and, therefore, $A$ is an NC cell complex of topological dimension at most $n$.
\end{proof}

\begin{lemma}
\label{lem:ApproxCellBaby}
Let $A$ be given by the pull-back diagram
\begin{equation}
\xymatrix{
A \ar[d]_\lambda \ar[r]^-{\rho} & \mathrmm C(X,M_k) \ar[d]^{f \mapsto f|_{X^{(0)}}} \\
B \ar[r]_-{\phi} & \mathrmm C(X^{(0)},M_k),
}
\end{equation}
where $B$ is an NC cell complex of topological dimension at most $n$, and $X$ is a simplicial complex of dimension at most $n$.
Then $A$ is an inductive limit of NC cell complexes of topological dimension at most $n$, where the maps of the inductive system are injective.
\end{lemma}

\begin{proof}
$X^{(0)}$ is a closed subset of a simplicial complex, and therefore it is the intersection of a decreasing sequence $(Y_n)_{n=1}^\infty$ of closed subsets, such that for each $n$, by changing the simplicial complex structure on $X$, $Y_n$ can be viewed as a subcomplex.
By Theorem \ref{thm:CellRSHSemiproj}, $\phi$ lifts to a $^*$-homomorphism $\tilde\phi\colon B \to \mathrmm C(Y_{n_0},M_k)$ for some $n_0$.

It follows that $A$ is an inductive limit (with injective connecting maps) of algebras $A_n$, defined by the pull-back diagrams
\begin{equation} 
\xymatrix{
A_n \ar[d]_\lambda \ar[r]^-{\rho} & \mathrmm C(X,M_k) \ar[d]^{f \mapsto f|_{Y_n}} \\
B \ar[r]_-{\tilde\phi(\cdot)|_{Y_n}} & \mathrmm C(Y_n,M_k),
}
\end{equation}
$n=1,2,\dots$.
Indeed, as subalgebras of $B \oplus C(X,M_k)$, we have 
\begin{equation}
A_1 \subseteq A_2 \subseteq \cdots \quad \text{and} \quad A = \overline{\bigcup_i A_i}.
\end{equation}
Moreover, by Lemma \ref{lem:CellRSHChar}, $A_n$ is an NC cell complex of topological dimension at most $n$.
\end{proof}

\begin{proof}[Proof of Theorem \ref{thm:ApproxCell}]
By \cite[Theorem 2.16]{Phillips:RSH}, $A$ is a recursive subhomogeneous algebra.
By induction on the length of the recursive subhomogeneous decomposition (see \cite[Definition 1.2]{Phillips:RSH}), it suffices to show that if $A$ is given by a pull-back diagram
\begin{equation}
\xymatrix{
A \ar[d] \ar[r] & \mathrmm C(X,M_k) \ar[d]^{f \mapsto f|_{X^{(0)}}} \\
B \ar[r]_-{\phi} & \mathrmm C(X^{(0)},M_k),
}
\end{equation}
where $B$ is locally approximated by NC cell complexes of topological dimension at most $n$ and $\dim X \leq n$, then $A$ is locally approximated by images of NC cell complexes of topological dimension at most $n$.
Using Lemma \ref{lem:ApproxCellB}, we may assume that $B$ is a single NC cell complex of topological dimension at most $n$. 

Let $\mathcal F \subset A$ be a finite subset and let $\e > 0$.
Without loss of generality, since $B$ is finitely generated, the finite set $\mathcal G := \{b \in B \mid (b,f) \in \mathcal F\}$ includes a set of generators of $B$.
Let $\dl \in (0,\e/3)$ be as provided by Corollary \ref{cor:CellRSHCloseRepsHomotopic}, with $m=k$ and $\e/3$ in place of $\e$.
By Corollary \ref{cor:CellRSHStableRelations}, there exists $\eta \in (0,\dl)$ such that, if $\alpha\colon B \to C$ is a ${}^*$-homomorphism, $D$ is a subalgebra of $C$, $\alpha(\mathcal G) \subset_\eta D$, and $D$ is isomorphic to $\mathrmm C(W,M_k)$ for some space $W$, then there is a ${}^*$-homomorphism $\hat\alpha\colon B \to D$ such that $\alpha(b) \approx_{\dl} \hat\alpha(b)$ for all $b\in\mathcal G$.

By Theorem \ref{thm:CellRSHSemiproj}, there exists a closed neighbourhood $Y$ of $X^{(0)}$ such that $\phi$ lifts to a $^*$-homomorphism $\tilde\phi\colon B \to \mathrmm C(Y,M_k)$.
Replace $Y$ by a smaller neighbourhood, if necessary, so that 
\begin{equation}
\label{eq:ApproxCellYDef}
\tilde\phi(b) \approx_\eta f|_Y, \quad a=(b,f) \in \mathcal F.
\end{equation}
For each $a=(b,f) \in \mathcal F$, choose $f_a \in \mathrmm C(X,M_k)$ such that 
\begin{equation}
\label{eq:ApproxCellfaDef}
f_a|_Y = \tilde\phi(b)\quad \text{and}\quad f_a \approx_\eta f.
\end{equation}
Since $X \setminus Y^{\circ}$ has dimension at most $n$, by \cite[Theorem 1.10.16]{Engelking:DimTheory} there exists an $n$-dimensional simplical complex $\Gamma$ and an injective ${}^*$-homomorphism $\sigma\colon \mathrmm C(\Gamma,M_k) \to \mathrmm C(X \setminus Y^{\circ},M_k)$ such that, for each $a=(b,f) \in \mathcal F$,
\begin{equation}
\label{eq:ApproxCellgaDef}
 f_a|_{X \setminus Y^{\circ}} \approx_\eta \sigma(g_a)
\end{equation}
for some $g_a \in \mathrmm C(\Gamma,M_k)$.
Let $\bar g_a \in \mathrmm C(X,M_k)$ be such that
\begin{equation}
\label{eq:ApproxCellgbarDef}
f_a \approx_\eta \bar g_a, \quad \bar g_a|_{X^{(0)}} = \phi(b), \quad\text{and}\quad \bar g_a|_{X \setminus Y^{\circ}} = \sigma(g_a).
\end{equation}
Note then that for $a=(b,f) \in \mathcal F$,
\begin{equation} \tilde\phi(b)|_{\partial Y} = f_a|_{\partial Y} \approx_\eta \bar g_a|_{\partial Y} \in \sigma(\mathrmm C(\Gamma,M_k))|_{\partial Y} =: D. \end{equation}
Let $\Gamma^{(0)}$ be the closed subset of $\Gamma$ such that $C_0(\Gamma \setminus \Gamma^{(0)})$ is the kernel of the $^*$-homomorphism $\sigma(\cdot)|_{\partial Y}$.
Then since $D$ is isomorphic to $\mathrmm C(\Gamma^{(0)},M_k)$, our choice of $\eta$ ensures that there is a ${}^*$-homomorphism $\theta_1\colon B \to D$
such that
\begin{equation}
\label{eq:ApproxCelltheta1Def}
\theta_1(b) \approx_{\dl} \bar g_a|_{\partial Y}
\end{equation}
for all $a=(b,f) \in \mathcal F$.
By Theorem \ref{thm:CellRSHSemiproj}, there is a closed neighbourhood $W$ of $\partial Y$ in $Y$ such that $\theta_1:B \to D \subset C(\partial Y,M_k)$ extends to a $^*$-homomorphism $B \to C(W,M_k)$, that we continue to denote $\theta_1$.
Replace $W$ by a smaller neighbourhood, if necessary, so that $\theta_1(b) \approx_\dl \tilde\phi(b)|_W$.

Corollary \ref{cor:CellRSHCloseRepsHomotopic} now provides us with a homotopy $\theta=(\theta_t)\colon B \to \mathrmm C([0,1]) \otimes \mathrmm C(W,M_k)$, such that $\theta_0=\tilde\phi(\cdot)|_{W}$, $\theta_1$ agrees with the existing definition, and 
\begin{equation} \theta_t(b) \approx_{\e/3} \tilde\phi(b)|_{W},\quad b \in \mathcal G. \end{equation}

Note that $\partial W$ (taken in $X$) is a disjoint union of $\partial Y$ and the closed set $\partial W \setminus \partial Y$.
Therefore, by Urysohn's Lemma, there exists a continuous function $f\colon W \to [0,1]$ satisfying
\begin{equation} f|_{\partial W \setminus \partial Y} \equiv 0 \quad\text{and} \quad f|_{\partial Y} \equiv 1. \end{equation}

Define $\psi\colon B \to \mathrmm C(Y,M_k)$ by
\begin{equation} \psi(b)(y) = \begin{cases}
\tilde\phi(b)(y),\quad&\text{if }y \in Y \setminus W, \\
\theta_{f(y)}(b)(y),\quad&\text{if }y \in W.
\end{cases} \end{equation}
Since $\tilde\phi(b)|_{\partial W \setminus \partial Y} = \theta_0(b)|_{\partial W \setminus \partial Y}$, the function $\psi(b)$ is continuous.
Also,
\begin{equation}
\label{eq:ApproxCellpsiApprox}
\psi(b) \approx_{\e/3} \tilde\phi(b),\quad b \in \mathcal G \quad \text{and} \quad \psi(\cdot)|_{X^{(0)}} = \phi.
\end{equation}

Since $\theta_1(B) \subseteq \sigma(\mathrmm C(\Gamma,M_k))$, $\psi(\cdot)|_{\partial Y}$ is a $^*$-homomorphism from $B$ to $\sigma(C(\Gamma,M_k))|_{\partial Y} \cong C(\Gamma^{(0)})$. 
Define $A'$ by the pull-back diagram
\begin{equation}
\xymatrix{
A' \ar[d] \ar[r] & \mathrmm C(\Gamma,M_k)  \ar[d]^{f \mapsto f|_{\Gamma^{(0)}}} \\
B \ar[r]_-{\psi(\cdot)|_{\partial Y}} & \mathrmm C(\Gamma^{(0)},M_k).
}
\end{equation}
By Lemma \ref{lem:ApproxCellBaby}, $A'$ is locally approximated by NC cell complexes of topological dimension at most $n$.

Let us now describe a $^*$-homomorphism $\pi\colon A' \to A$.
For $(b,g) \in A' \subseteq B \oplus \mathrmm C(\Gamma,M_k)$, define $f \in \mathrmm C(X,M_k)$ by
\begin{equation}
\label{eq:ApproxCellpiDef}
f|_{X \setminus Y^{\circ}} := \sigma(g)|_{X \setminus Y^{\circ}}
\quad \text{and} \quad 
f|_{Y^{\circ}} := \psi(b).
\end{equation}
The definition of $B$ ensures that $\sigma(b)|_{\partial Y} = \psi(g)|_{\partial Y}$, so that $f$ is indeed continuous.
Note that
\begin{equation} f|_{X^{(0)}} = \psi(b)|_{X^{(0)}} \stackrel{\eqref{eq:ApproxCellpsiApprox}}= \phi(b), \end{equation}
so $(b,f) \in A$; we define $\pi(b,g) := (b,f)$.
Since $\sigma$ is injective, $f$ is only $0$ if $g$ is zero; consequently, $\pi$ is injective.

Now, let us show that $\mathcal F \subset_\e \pi(A')$.
Let $a =(b,f) \in \mathcal F$.
By \eqref{eq:ApproxCelltheta1Def} and \eqref{eq:ApproxCellgbarDef}, 
\begin{equation} g_a|_{\Gamma_0} \approx_{\e/3} \theta_1(b)|_{\partial Y} = \psi(b)|_{\partial Y}, \end{equation}
 so let $g \in \mathrmm C(\Gamma,M_k)$ be such that 
\begin{equation}
\label{eq:ApproxCellgDef}
g|_{\Gamma_0} = \psi(b)|_{\partial Y}\quad \text{and} \quad g \approx_{\e/3} g_a.
\end{equation}
Then $(b,g) \in A'$; let us show that $\pi(b,g) \approx_\e a$.

Let $\pi(b,g) = (b,f')$.
Then
\begin{eqnarray}
\notag
f'|_Y &\stackrel{\eqref{eq:ApproxCellpiDef}}=& \psi(b) \\
\notag
&\stackrel{\eqref{eq:ApproxCellpsiApprox}}{\approx_{\e/3}}& \tilde\phi(b) \\
&\stackrel{\eqref{eq:ApproxCellYDef}}{\approx_{\e/3}}& f,
\end{eqnarray}
and
\begin{eqnarray}
\notag
f'|_{X \setminus Y} &\stackrel{\eqref{eq:ApproxCellpiDef}}=& \sigma(g)|_{X \setminus Z} \\
\notag
&\stackrel{\eqref{eq:ApproxCellgDef}}{\approx_{\e/3}}& \sigma(g_a)|_{X \setminus Z} \\
\notag
&\stackrel{\eqref{eq:ApproxCellgaDef}}{\approx_{\e/3}}& f_a|_{X \setminus Z} \\
&\stackrel{\eqref{eq:ApproxCellfaDef}}{\approx_{\e/3}}& f|_{X \setminus Z}.
\end{eqnarray}
Thus, $\pi(b,g) =(b,f') \approx_\e (b,f) = a$ as required.
\end{proof}

\begin{cor}
Let $A$ be a $\mathrmm C^*$-algebra.
Then $A$ is locally subhomogeneous if and only if it is locally approximated by NC cell complexes.
\end{cor}

\begin{proof}
Since NC cell complexes are subhomogenous, the reverse direction is obvious.
In the other direction, since a $\mathrmm C^*$-subalgebra of a subhomogeneous algebra is itself subhomogeneous, every locally subhomogeneous algebra is locally approximated by finitely generated, subhomogeneous algebras.
Moreover, \cite[Theorem 1.5]{NgWinter:ash} shows that every finitely generated, subhomogeneous algebra satisfies the hypothesis of Theorem \ref{thm:ApproxCell}, and is therefore itself locally approximated by NC cell complexes.
\end{proof}

\subsection{$\mathcal W$-stable approximately subhomogeneous algebras}

We include another result here, which is proven using an adaptation of the proof of Theorem \ref{thm:ApproxCell}.
Recall the stably projectionless, monotracial, $K$-theoretically trivial algebra $\mathcal W$ studied in \cite{Jacelon:R}.

\begin{thm}
\label{thm:WASH}
Let $A$ be a separable locally subhomogeneous algebra.
Then $A \otimes \mathcal W$ is an inductive limit of one-dimensional NC cell complexes (i.e., point--line algebras; see Remark \ref{rmk:NCcellDef} (iii)).
\end{thm}

We need some technical lemmas to prove this.

\begin{lemma}
\label{lem:drSHext}
Suppose that 
\begin{equation} 0 \to I \to A \to B \to 0 \end{equation}
is an extension of $\mathrmm C^*$-algebras, such that $I$ and $B$ are subhomogeneous.
Then the topological dimension of $A$ is the maximum of the topological dimension of $I$ and of $B$.
\end{lemma}

\begin{proof}
$A$ is subhomogeneous, and for each $k$, $\mathrm{Prim}_k(A)$ can be decomposed into an open set homeomorphic to $\mathrm{Prim}_k(I)$ and a closed set homeomorphic to $\mathrm{Prim}_k(B)$.
Therefore,
\begin{equation} \dim \mathrm{Prim}_k(A) = \max\{\dim \mathrm{Prim}_k(I), \dim \mathrm{Prim}_k(B)\}, \end{equation}
and the result follows by the definition of topological dimension (Definition \ref{def:TopDim}).
\end{proof}

\begin{prop}
\label{prop:WASHtechnical}
Let $C$ be a $\mathrmm C(X)$-algebra, let $X^{(0)} \subseteq X$ be closed and let $A$ be given by the pull-back diagram
\begin{equation}
\label{eq:WASHtechnicalPB}
\xymatrix{
A \ar[d]_\lambda \ar[r]^-{\rho} & C  \ar[d]^{f \mapsto f|_{X^{(0)}}} \\
B \ar[r]_-{\phi} & C|_{X^{(0)}},
}
\end{equation}
where $\phi$ is a ${}^*$-homomorphism.
If $B$ is approximated by subhomogeneous algebras of topological dimension at most $1$, and $C$ is locally subhomogeneous, then $A$ is locally subhomogeneous.
Moreover, if $C$ is approximated by subhomogeneous $\mathrmm C^*$-algebras of topological dimension at most $n \geq 1$, then so also is $A$.
\end{prop}

\begin{proof}
This proof is along the same lines as the proof of Theorem \ref{thm:ApproxCell}.
Let $\mathcal F \subset A$ be a finite subset that we wish to approximate, up to a tolerance $\e > 0$.
Without loss of generality, by Theorem \ref{thm:ApproxCell} and Lemma \ref{lem:ApproxCellB}, we may suppose that $B$ is a one-dimensional NC cell complex.
Choose a finite generating subset $\mathcal G$ of $B$ containing $\lambda(\mathcal F)$.

Since $B$ is (weakly) semiprojective \cite{EilersLoringPedersen:NCCW} (see also \cite[Remark 3.5]{Enders:SemiprojPB}), there exists $\eta \in (0,\e/3)$ such that, if $\alpha\colon B \to E$ is a ${}^*$-homomorphism and $E'$ is a subalgebra of $E$ such that $\alpha(\mathcal G) \subset_\eta E'$ then there is a ${}^*$-homomorphism $\hat\alpha\colon B \to E'$ such that $\alpha(b) \approx_{\e/3} \hat\alpha(b)$ for all $b\in\mathcal G$.
Also since $B$ is semiprojective, there exists a closed neighbourhood $Y$ of $X^{(0)}$ such that $\phi$ lifts to a $^*$-homomorphism $\tilde\phi\colon B \to C|_Y$.
Replace $Y$ by a smaller neighbourhood if necessary, so that $\tilde\phi\circ \lambda(a) \approx_\eta \rho(a)|_Y$ for $a \in \mathcal F$.

For each $a =(b,f) \in \mathcal F$, choose $f_a \in C$ such that
\begin{equation}
\label{eq:1NCCWfaDef}
f_a|_Y = \tilde\phi(b) \quad\text{and}\quad f_a \approx_\eta f.
\end{equation}

Since $C$ is locally subhomogeneous, there exists a subhomogeneous subalgebra $D$ of $C$ such that, for each $a \in \mathcal G$,
\begin{equation} f_a \approx_\eta g_a \end{equation}
for some $g_a \in D$.
Let $Z$ be a closed neighbourhood of $X \setminus Y^\circ$, disjoint from $X^{(0)}$.
Choose, for each $a=(b,f) \in \mathcal F$, an element $\bar g_a \in C$ such that
\begin{equation}
\label{eq:1NCCWgbarDef}
f_a \approx_\eta \bar g_a, \quad \bar g_a|_{X^{(0)}} = \phi(b)|_{X^{(0)}}, \quad\text{and}\quad \bar g_a|_Z = g_a|_Z \in D|_Z.
\end{equation}

Consider the quotient of $A$ by $C|_{X \setminus Y}$, which we denote by $B \oplus_{C|_{X^{(0)}}} C|_Y$.
The pair $(\id_B,\tilde\phi)$ provides a ${}^*$-homomorphism $B \to B \oplus_{C|_{X^{(0)}}} C|_Y$, and the image of $\mathcal G$ under this ${}^*$-homomorphism is contained, up to $\eta$, in the subalgebra 
\begin{equation} \mathrmm C^*(\{(b,\bar g_a|_Y) \mid a=(b,f) \in \mathcal F\}). \end{equation}
By the choice of $\eta$, it follows that there is a ${}^*$-homomorphism
\begin{align}
\psi=(\psi_1,\psi_2)\colon B \to\, & \mathrmm C^*(\{(b,\bar g_a|_Y) \mid a=(b,f) \in \mathcal F\})
\end{align}
such that
\begin{equation}
\label{eq:1NCCWpsiDef}
\psi\circ \lambda(a) \approx_{\e/3} (b,\bar g_a|_Y)
\end{equation}
for all $a=(b,f) \in \mathcal F$.

Note that, for $b\in \mathcal G$, there exists $a \in \mathcal F$ such that $\lambda(a)=b$; thus $\psi_2(b)|_{Y \cap Z} \in \mathrmm C^*(\{\bar g_a|_{Y \cap Z}\}) \subseteq D|_{Y \cap Z}$.
Since $\mathcal G$ generates $B$, it follows that $\psi_2(B) \subseteq D|_{Y \cap Z}$; therefore, we may define $E$ by the pull-back diagram
\begin{equation}
\xymatrix{
E \ar[r] \ar[d] & D|_Z \ar[d]^{f \mapsto f|_{Y \cap Z}} \\
B \ar[r]_-{\psi_2(\cdot)|_{Y \cap Z}} & D|_{Y \cap Z},
}
\end{equation}
which is clearly subhomogeneous since $D$ and $B$ are.
Moreover, $E$ can be identified with a subalgebra of $A$ by
\begin{equation} (b,d) \mapsto (\psi_1(b), c) \in B \oplus C \end{equation}
where $c|_Y := \psi_2(b)$ and $c|_Z := d$.
This subalgebra approximates $\mathcal F$ up to $\e$.

Applying Lemma \ref{lem:drSHext} to $E$ yields the final statement, about topological dimension.
\end{proof}

\begin{proof}[Proof of Theorem \ref{thm:WASH}]
Since one-dimensional subhomogeneous algebras are approximated by one-dimensional NC cell complexes (Theorem \ref{thm:ApproxCell}), and one-dimensional NC cell complexes are semiprojective \cite{EilersLoringPedersen:NCCW} (see also \cite[Remark 3.5]{Enders:SemiprojPB}), it suffices to show that $A \otimes \mathcal W$ is locally approximated by one-dimensional subhomogeneous algebras.
For this, it suffices to assume that $A$ is recursive subhomogeneous (in fact, by Theorem \ref{thm:ApproxCell}, we could assume it is an NC cell complex---but this isn't needed).

By induction on its decomposition, we need to show that if an algebra $A$ is defined by a pull-back diagram as in \eqref{eq:WASHtechnicalPB} (with $X^{(0)}$ possibly empty), where $C:=\mathrmm C(X,\mathcal W)$, and $B$ is approximated by one-dimensional subhomogeneous algebras, then $A$ is approximated by one-dimensional subhomogeneous algebras.
In this case, $C$ is approximated by one-dimensional subhomogeneous algebras by \cite[Theorem 1.4]{Santiago:DimRed}, and hence Proposition \ref{prop:WASHtechnical} implies that the same holds for $A$, as required.
\end{proof}

\section{Decomposition rank}

Recall the definition of the decomposition rank of a $^*$-homomorphism from \cite[Definition 2.2]{TW:Zdr}.
If $A \subseteq B$ then we write $\dr(A \subseteq B)$ to denote the decomposition rank of the inclusion map; when $A=B$, this is the same as the decomposition rank of $B$, whereas in the case of a first-factor embedding $A \otimes 1_D \subseteq A \otimes D$, where $D$ is strongly self-absorbing, this decomposition rank value relates to the decomposition rank of $A \otimes D$, see \cite[Proposition 2.6]{TW:Zdr}.

In this section, we prove the following result:

\begin{thm}
\label{thm:drBound}
Let $A$ be an NC cell complex.
Then 
for any unital $\mathrmm C^*$-algebra $D$,
\begin{equation}
\dr (A \otimes 1_D \subseteq A \otimes D) \leq \sup \dr (\mathrmm C(Z) \otimes 1_D \subseteq \mathrmm C(Z) \otimes D),
\end{equation}
where the maximum is taken over all compact metrizable spaces $Z$ of the same topological dimension as that of $A$.
\end{thm}

In combination with the main result of \cite{TW:Zdr} (and Theorem \ref{thm:ApproxCell} to handle general subhomogeneous $\mathrm C^*$-algebras), Theorem \ref{thm:MainThmA} is a consequence (see Corollary \ref{cor:MainThmA}).
It also provides an alternative proof of the main result of \cite{Winter:drSH}, see Corollary \ref{cor:drSH} (although both proofs are fairly technical, they are fundamentally different; A.T.\ and Wilhelm Winter found out the hard way that the argument in \cite{Winter:drSH} probably cannot be adapted to prove Theorem \ref{thm:drBound}).

\subsection{A special case}
We include the construction of c.p.c.\ approximations for special one-step NC cell complexes, to give an idea of how the general case works.
The general case involves a significant amount of recursion and induction which can be avoided in this case.

Consider a separable $\mathrmm C^*$-algebra $A$ which is given by a pull-back diagram
\begin{equation}
\label{eq:BabySetup}
\xymatrix{
A \ar[r]^-{\sigma} \ar[d]_{\lambda} & \mathrmm C(X_1 \times [0,1],M_{m_1}) \ar[d]^-{ f \mapsto f|_{X_1 \times \{0\}}} \\
\mathrmm C(X_0,M_{m_0}) \ar[r]_-{\theta} & \mathrmm C(X_1,M_{m_1}),
}
\end{equation}
where $\theta$ is a unital ${}^*$-homomorphism.
This implies that $m_1 = \el m_0$ for some $\el \in \N$.

Recalling that $\hat\theta(x)$ means $\ev_x \circ \theta$, define
\begin{equation} S\colon \mathrm{Hom}(\mathrmm C(X_0,M_{m_0}),M_{m_1}) \to \{\text{finite subsets of }X_0\} \end{equation}
by
\begin{equation} S(\Ad(u) \circ \diag(\ev_{z_1},\dots, \ev_{z_\el})) := \{z_1,\dots,z_\el\} \end{equation}
(note that we are taking a set, not a multiset, here).
Of course, every unital ${}^*$-homomorphism $\alpha\colon \mathrmm C(X_0,M_{m_0}) \to M_{m_1}$ can be expressed as $\alpha=\Ad(u) \circ \diag(\ev_{z_1},\dots, \ev_{z_\el})$, and although this expression is not unique, the set $S(\alpha)=\{z_1,\dots,z_\el\}$ is; in fact,
\begin{equation} S(\alpha) = X_0 \setminus U \end{equation}
where $\ker(\alpha) = C_0(U,M_{m_0})$.

Define
\begin{equation}
\label{eq:BabyYDef}
Y := \{(x_1,x_0) \in X_1 \times X_0 \mid x_1 \in X_1\text{ and }x_0 \in S(\hat\theta(x_1))\};
\end{equation}
since $S$ is continuous (in, say, the Hausdorff metric on finite subsets of $X_0$), it follows that $Y$ is a closed, and therefore compact, subset of $X_1 \times X_0$.

Define $\bar\xi\colon \mathrmm C(Y,M_{m_0}) \to \mathrmm C(X_1,M_{m_1})$ as follows.
For $x_1 \in X_1$ and $f \in \mathrmm C(Y,M_{m_0})$, let us express $\hat\theta(x_1)$ as
\begin{equation} \hat\theta(x_1) = \,\Ad(u) \circ \diag(\ev_{z_1},\dots,\ev_{z_\el}), \end{equation}
and then define
\begin{equation}
\label{eq:barXiDef}
 \bar\xi(f)(x_1) := \,\Ad(u)(\diag(f(x_1,z_1),\dots,f(x_1,z_\el))).
\end{equation}
(One may check that this is well defined---it does not depend on the chosen representation of $\hat\theta(x)$ and $\bar\xi(f)$ is indeed continuous.)
A critical property of $\bar\xi$ is that $\theta$ factors as
\begin{equation} 
\label{eq:barXiFactors}
\mathrmm C(X_0,M_{m_0}) \to \mathrmm C(Y,M_{m_0}) \labelledrightarrow{\bar\xi} \mathrmm C(X_1,M_{m_1}), \end{equation}
where the first map is induced by the coordinate projection $Y \to X_0$.

Fix $\eta \in (0,1)$ and define
\begin{equation}
\label{eq:BabyZDef}
Z_\eta := (X_0 \times \{0\}) \amalg (Y \times [0,\eta]) \amalg (X_1 \times [\eta,1])/\sim, \end{equation}
where $\sim$ is the equivalence relation generated by the following:
\begin{enumerate}
\item for $(x_1,x_0) \in Y$,
\begin{equation} (x_0,0) \sim ((x_1,x_0),0); \text{ and} \end{equation}
\item for $(x_1,x_0) \in Y$,
\begin{equation} ((x_1,x_0),\eta) \sim (x_1,\eta). \end{equation}
\end{enumerate}
(It is easy to see that $\sim$ is closed, and therefore $Z_\eta$ is Hausdorff.)

These constructions will be used in the proof of Proposition \ref{prop:BabyVersion} (a baby version of Theorem \ref{thm:drBound}) below.

\begin{example}
This simple example illustrates the constructions.
Let $X_0:=\{a,b\}$, $X_1:=\{c,d\}$, $m_0=1, m_1=2$, and define $\theta\colon C(X_0) \to C(X_1,M_2)$ by
\begin{equation}
\theta(f)(c) := \diag(f(a), f(b)), \quad \theta(f)(d) := \diag(f(a),f(a)).
\end{equation}
In this case, we have $Y=\{(c,a),\, (c,b),\, (d,a)\}$ and $\bar\xi\colon C(Y) \to C(X_1,M_2)$ is given by 
\begin{equation}
\bar\xi(f)(c) = \diag(f(c,a), f(c,b)), \quad \bar\xi(f)(d) = \diag(f(d,a), f(d,a)).
\end{equation}
The space $Z_\eta$ looks like the following:

\centerline{\includegraphics[height=1in]{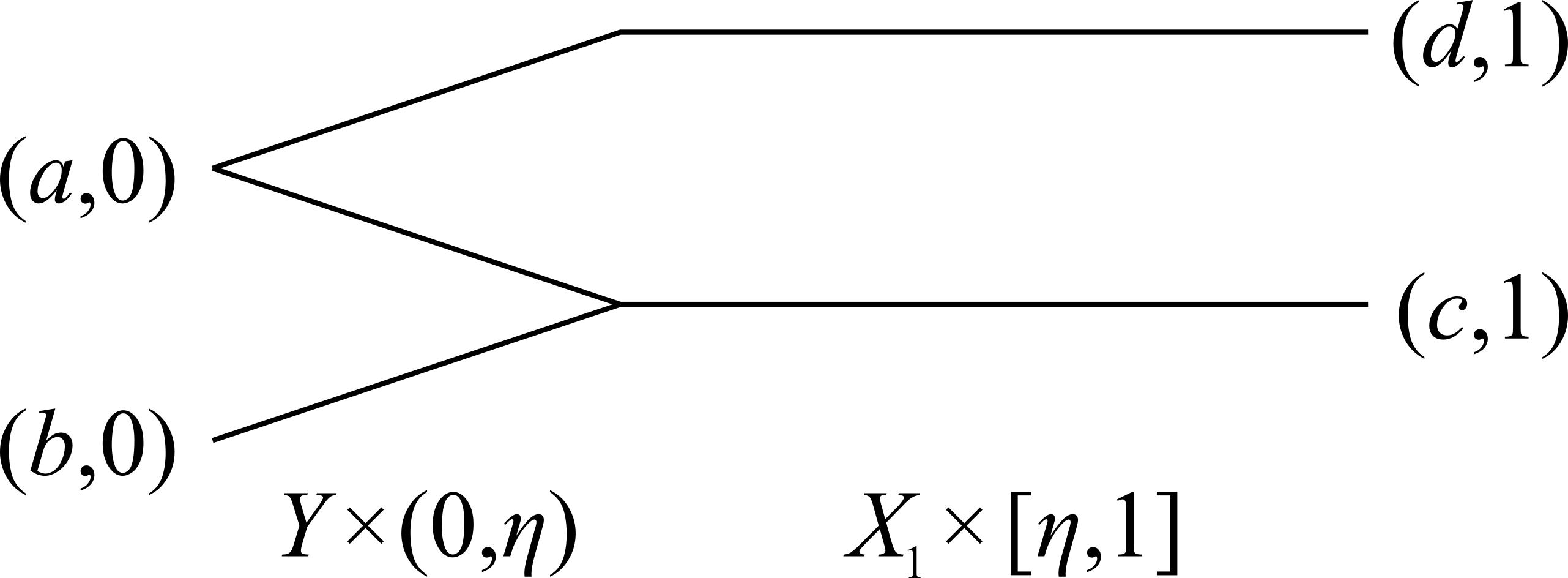}}
\end{example}

\begin{defn}
\label{def:Colourable}
Let $F,A$ be $\mathrmm C^*$-algebras, with $F$ finite dimensional, and let $\phi\colon F \to A$ be a c.p.\ map.
The map $\phi$ is \textbf{$(n+1)$-colourable} if $F$ decomposes as a direct sum of subalgebras, $F=F_0 \oplus \cdots \oplus F_n$, such that $\phi|_{F_i}$ is orthogonality preserving (i.e.\ order zero), $i=0,\dots,n$.
Likewise, a finite family $(e_j)_{j\in I}$ of positive contractions is \textbf{$(n+1)$-colourable} means that we can decompose $I=I_0\amalg \cdots \amalg I_n$ such that $(e_j)_{j\in I_i}$ is a pairwise orthogonal family, $i=0,\dots,n$.
\end{defn}

\begin{prop}
\label{prop:BabyVersion}
Consider the $\mathrmm C^*$-algebra $A$ given by the pull-back diagram \eqref{eq:BabySetup} and let $D$ be an infinite dimensional UHF algebra.
Then $\dr (A \otimes D) \leq 2$.
\end{prop}

\begin{proof}
Let $\mathcal F \subset A$ be a finite subset and let $\e > 0$; we must show that we can $(\mathcal F,\e)$-approximately factorize the identity map on $A$ in a $3$-colourable way.
We may assume that $\mathcal F$ consists of positive elements.
Moreover, by approximating $\mathcal F$, we may assume that there exists $\eta > 0$ such that, for every $a = (f,g) \in \mathcal F \subset \mathrmm C(X_0,M_{m_0}) \oplus \mathrmm C(X_1 \times [0,1],M_{m_1})$, $x \in X_1$, and $t \in [0,2\eta]$,
\begin{equation}
\label{eq:BabyVersionF}
 g(x,t) = g(x,0). \end{equation}

Set $Z:=Z_\eta$ as defined in \eqref{eq:BabyZDef}.
Let us identify the spaces $X_0$, $Y \times (0,\eta)$, and $X_1 \times [\eta,1]$ with subspaces of $Z$.

Let $\mathcal U_0$ be an open cover of $X_0$ such that for every $(f,g) \in \mathcal F \subset \mathrmm C(X_0,M_{m_0}) \oplus \mathrmm C(X_1\times [0,1],M_{m_1})$ and every $U \in \mathcal U_0$, the restriction $f|_U$ is approximately constant (to within $\e$).
Choose an open cover $\mathcal U_1$ of $Z$ such that for $U \in \mathcal U_1$:

(a) If $U \cap X_0 \neq \emptyset$ then $U \cap (X_1 \times [\eta,1]) = \emptyset$, and there exists $V \in \mathcal U_0$ such that $U \cap X_0 \subseteq V$ and for every $((x,y),t) \in (Y \times (0,\eta)) \cap U$, it is the case that $y \in V$;

(b) If $U \cap (Y \times (0,\eta)) \neq \emptyset$ then $U \cap (X_1 \times [2\eta,1]) = \emptyset$, and there exist disjoint open sets $V_1,\dots,V_q \in \mathcal U_0$ such that for any $(x,t) \in U \cap (X_1 \times [\eta,2\eta))$ or $((x,y),t) \in U \cap (Y \times (0,\eta))$,
\begin{equation} S(\hat\theta(x)) \subset V_1 \cup \cdots \cup V_q; \text{and} \end{equation}

(c) For every $(f,g) \in \mathcal F \subset \mathrmm C(X_0,M_{m_0}) \oplus \mathrmm C(X_1 \times [0,1],M_{m_1})$, the restriction $g|_{U \cap (X_1 \times [\eta,1])}$ is approximately constant (to within $\e$).

Since the embedding $\mathrmm C(Z)\otimes 1_D \subset \mathrmm C(Z) \otimes D$ has decomposition rank at most $2$ \cite[Theorem 4.1]{TW:Zdr}, we may find a $3$-colourable approximate (to within $\e$) finite partition of unity $(a_i)_{i=1}^r$ in $\mathrmm C(Z,D)_+$ subordinate to $\mathcal U_1$ (by \cite[Proposition 3.2]{TW:Zdr}, where this concept is defined; see also Definition \ref{def:Colourable}).
By possibly rescaling, assume that $\sum a_i \leq 1_D$.

Let's now define our c.p.c.\ approximation
\begin{equation} A \labelledrightarrow{\psi} F_1 \oplus \cdots \oplus F_r \labelledrightarrow{\phi} D \otimes A, \end{equation}
by defining finite dimensional (in fact matrix) algebras $F_i$ together with the parts of $\psi$ and $\phi$ corresponding to each $F_i$.
This is done in cases.

\textit{Case 1}: The support of $a_i$ intersects $X_0$.
Let $b_i \in \mathrmm C(Y \times [0,\eta],D)$ be given by composing $a_i$ with the natural map $Y \times [0,\eta] \to Z$ (from the definition of $Z$).
Using $\bar\xi\colon C(Y,M_{m_0})\to C(X_1,M_{m_1})$ as defined in \eqref{eq:barXiDef}, identify the domain and codomain appropriately so as to view $\bar\xi \otimes \id_{\mathrmm C([0,\eta]) \otimes D}$ as a map from $\mathrmm C(Y \times [0,\eta],D) \otimes M_{m_0}$ to $D \otimes \mathrmm C(X_1 \times [0,\eta], M_{m_1})$.
Then, note that by condition (a) on the open cover $\mathcal U_1$, $\bar\xi(b_i \otimes 1_{m_0})$ vanishes on $X_1 \times \{\eta\}$, i.e.,
\begin{align}
(\bar\xi \otimes \id_{\mathrmm C([0,\eta]) \otimes D})(b_i \otimes 1_{m_0}) &\in D \otimes C_0(X_1 \times [0,\eta), M_{m_1})
\end{align}
and
\begin{eqnarray}
\notag
 (\bar\xi \otimes \id_{\mathrmm C([0,\eta]) \otimes D})(b_i \otimes 1_{m_0})|_{X_1 \times \{0\}} &=& (\bar\xi \otimes \id_{D})(b_i|_{Y \times \{0\}} \otimes 1_{m_0}) \\
&\stackrel{\eqref{eq:barXiFactors}}=& (\theta \otimes \id_{D})(a_i|_{X_0} \otimes 1_{m_0}),
\end{eqnarray}
where, again, we make appropriate identifications of domains and codomains of the maps involved.
This shows that
\begin{equation} e_i := (a_i|_{X_0} \otimes 1_{m_0}, (\bar\xi \otimes \id_{\mathrmm C([0,\eta]) \otimes D})(b_i \otimes 1_{m_0})) \in (D \otimes A)_+. \end{equation}

Set
\begin{equation} F_i := M_{m_0} \end{equation}
and define the c.p.c.\ order zero map $\phi_i\colon M_{m_0} \to D \otimes A$ by
\begin{equation} \phi_i(\kappa) := (a_i|_{X_0} \otimes \kappa, (\bar\xi \otimes \id_{\mathrmm C([0,\eta]) \otimes D})(b_i \otimes \kappa)); \end{equation}
(the range of this map is contained in $D \otimes A$, for similar reasons as for $e_i$).
Fix a choice of $x_i \in X_0 \cap \mathrm{supp}\, a_i$, and use this to define $\psi_i\colon A \to M_k$ by
\begin{equation} \psi_i(f,g) = f(x_i). \end{equation}

\textit{Case 2}: The support of $a_i$ intersects $Y \times (0,\eta)$, but not $X_0$.
Let $U \in \mathcal U_1$ contain the support of $a_i$ and for this set, let $V_1,\dots,V_q \in \mathcal U_0$ be disjoint sets as in (b), which we relabel $V_{i,1},\dots,V_{i,q(i)}$.
We define $F_i:=F_{i,1} \oplus \cdots \oplus F_{i,q(i)}$ where we define $F_{i,j}$ (and the parts of $\phi,\psi$ relating to $F_{i,j}$) now.
Let us therefore fix $j=1,\dots,q(i)$. 

Define $b_{i,j}\colon Y \times (0,2\eta) \to D$ by
\begin{equation} b_{i,j}((x,y),t) := \begin{cases}
a_i((x,y),t),\quad &\text{if }t \in (0,\eta)\text{ and }y\in V_{i,j}; \\
a_i(x,t),\quad &\text{if }t\in [\eta,2\eta) \text{ and }S(\hat\theta)(x) \cap V_{i,j} \neq \emptyset; \\
0, \quad &\text{otherwise.}
\end{cases}
\end{equation}
One may check (using (b)) that $b_{i,j}$ is continuous, and vanishes on $Y \times \{0,2\eta\}$; therefore, we may view it as an element of
\begin{equation} C_0(Y \times (0,2\eta),D) \subset C_0(Y \times (0,1],D). \end{equation}
We note, for future use, that for $((x,y),t) \in Y \times (0,2\eta)$,
\begin{equation}
\label{eq:BabybFact}
\sum_{j=1}^{q(i)} b_{i,j}((x,y),t) = \begin{cases} a_i((x,y),t), \quad &\text{if }t\in (0,\eta); \\ a_i(x,t),\quad &\text{if }t \in [\eta,2\eta). \end{cases}
\end{equation}

Define
\begin{equation} e_{i,j}:= (0_{X_0}, (\bar\xi \otimes \id_{\mathrmm C([0,1]) \otimes D})(b_{i,j} \otimes 1_{m_0})) \in A_+. \end{equation}
Define the summand
\begin{equation} F_{i,j} := M_{m_0} \end{equation}
and define $\phi_{i,j}\colon M_{m_0} \to D \otimes A$ by
\begin{equation} \phi_{i,j}(\kappa):= (0_Y, (\bar\xi \otimes \id_{\mathrmm C([0,1]) \otimes D})(b_{i,j} \otimes \kappa)). \end{equation}
Fix a choice of $y_{i,j} \in V_{i,j}$ and use this to define $\psi_{i,j}\colon A \to M_{m_0}$ by
\begin{equation} \psi_{i,j}(f,g) = f(y_{i,j}). \end{equation}

\textit{Case 3}: The support of $a_i$ is contained in $X_1 \times (\eta,1]$.
In this case, $a_i|_{X_1 \times (\eta,1]} \in C_0(X_1 \times (\eta,1],D)_+ \subset \mathrmm C(X_1 \times [0,1],D)_+$, so
\begin{equation} e_i := (0_{X_0},a_i|_{X_1 \times [\eta,1]} \otimes 1_{m_1}) \in (D \otimes A)_+. \end{equation}
Define the summand
\begin{equation} F_i := M_{m_1} \end{equation}
and define $\phi_i\colon M_{m_1} \to D \otimes A$ by
\begin{equation} \phi_i(\kappa) := (0_{X_0},a_i|_{X_1 \times [\eta,1]} \otimes \kappa). \end{equation}
Choose $z_i \in \mathrm{supp}\, a_i \subseteq X_1 \times (\eta,1]$ and use this to define $\psi_i\colon A \to M_k$ by
\begin{equation} \psi_i(f,g) := g(z_i). \end{equation}

Set $G_1,G_2$ and $G_3$ equal to the set of all $i$ from Cases 1, 2 and 3 respectively.
Thus, we have
\begin{equation} F=\bigoplus_{i \in G_1} F_i \oplus \bigoplus_{i \in G_2} \bigoplus_{j=1}^{q(i)} F_{i,j} \oplus \bigoplus_{i \in G_3} F_i. \end{equation}
The maps $\psi$ and $\phi$ are given as (direct) sums in the same way.

It is easy to see that whenever $a_ia_{i'}=0$, we also have $e_ie_{i'}=0$ (or $e_ie_{i',j'}=0$, etc.), and that for $i \in G_2$, the elements $e_{i,1},\dots,e_{i,q(i)}$ are mutually orthogonal.
Therefore, $\phi$ is $3$-colourable.
Since $\sum a_i \leq 1_Z$, we also have
\begin{equation} \phi(1_F) = \sum_{i \in G_1} e_i + \sum_{i \in G_2} \sum_{j=1}^{q(i)} e_{i,j} + \sum_{i\in G_3} e_i \leq 1_{D \otimes A}, \end{equation}
and so $\phi$ is contractive.

Next, we must check that $\phi\psi(a) \approx 1_D \otimes a$ for $a = (f,g) \in \mathcal F$.
We see that $\phi\psi(a) = (f',g')$ where
\begin{equation}
\label{eq:Babyf'Def}
f' := \sum_{i \in G_1} (a_i|_{X_0} \otimes f(x_i))
\end{equation}
and
\begin{align}
\notag
g' &:= \sum_{i\in G_1}(\bar\xi \otimes \id_{\mathrmm C([0,\eta]) \otimes D})(b_i \otimes f(x_i)) + \\
\notag
&\quad \sum_{i \in G_2} \sum_{j=1}^{q(i)} (\bar\xi \otimes \id_{\mathrmm C([0,\eta]) \otimes D})(b_{i,j} \otimes f(y_{i,j})) + \\
&\quad \sum_{i \in G_3} a_i|_{X_1 \times [\eta,1]} \otimes g(z_i).
\end{align}
For $x \in X_0$, we have that $a_i(x) \neq 0$ only if $f(x_i) \approx_\e f(x)$.
Also, by the definition of $G_1$, $\sum_{i\in G_1} a_i(x) = \sum_i a_i(x) \approx_\e 1_D$.
Therefore, it follows from \eqref{eq:Babyf'Def} that $f'(x) \approx_{2\e} 1_D \otimes f(x)$.
Hence $\|f'-f\| \leq 2\e$.

Next, let $(x,t) \in X_1 \times [0,\eta)$.
Express $\hat\theta(x)$ as 
\begin{equation}
\label{eq:BabyHatThetaExp}
\hat\theta(x) = \,\Ad(u) \circ \diag(\ev_{z_1},\dots,\ev_{z_\el}).
\end{equation}
Then, for $i \in G_1$,
\begin{eqnarray} 
&&\quad (\bar\xi \otimes \id_{\mathrmm C([0,\eta]) \otimes D})(b_i \otimes f(x_i))(x,t) \\
\notag
&\stackrel{\eqref{eq:barXiDef}}=& \Ad(1_D \otimes u)(\diag(b_i((x,z_1),t) \otimes f(x_i),\dots,b_i((x,z_\el),t) \otimes f(x_i))) \\
\notag
&=& \Ad(1_D \otimes u)(\diag(a_i((x,z_1),t) \otimes f(x_i),\dots,a_i((x,z_\el),t) \otimes f(x_i))).
\end{eqnarray}
using the definition of $b_i$.
Likewise, for $i \in G_2$ and $j=1,\dots,q(i)$,
\begin{eqnarray}
\notag
&&
(\bar\xi \otimes \id_{\mathrmm C([0,\eta]) \otimes D})(b_{i,j} \otimes f(y_{i,j}))(x,t) \\
\notag
&\stackrel{\eqref{eq:barXiDef}}=& \,\Ad(1_D \otimes u)(\diag(b_{i,j}((x,z_1),t) \otimes f(y_{i,j}),\dots, \\
\label{eq:BabyComputation1}
&&\quad b_{i,j}((x,z_\el),t) \otimes f(y_{i,j}))).
\end{eqnarray}
Combining these, we find that
\begin{align}
\notag
g'(x,t) &= \sum_{i\in G_1}(\bar\xi \otimes \id_{\mathrmm C([0,\eta]) \otimes D})(b_i \otimes f(x_i))(x,t) + \\
\notag
&\quad \sum_{i \in G_2} \sum_{j=1}^{q(i)} (\bar\xi \otimes \id_{\mathrmm C([0,\eta]) \otimes D})(b_{i,j} \otimes f(y_{i,j}))(x,t) \\
\label{eq:BabyComputation2}
&= \Ad(1_D \otimes u)(\diag(c_1,\dots,c_l))
\end{align}
where
\begin{align}
\notag
c_k &:= \sum_{i\in G_1} a_i(x,z_k) \otimes f(x_i) + \\
\notag
&\qquad \sum_{i \in G_2} \sum_{j=1}^{q(i)} b_{i,j}((x,z_k),t) \otimes f(y_{i,j}) \\
\notag
&\approx_\e \sum_{i\in G_1} a_i(x,z_k) \otimes f(z_k) + \\
\notag
&\qquad \sum_{i \in G_2} \sum_{j=1}^{q(i)} b_{i,j}((x,z_k),t) \otimes f(z_k) \\
&\approx_\e 1_D \otimes f(z_k),
\label{eq:BabycComputation1}
\end{align}
using the fact that $f$ is approximately constant on each set in $\mathcal U_0$ for the first approximation, and using \eqref{eq:BabybFact} and that $\sum_{i\in G_1 \cup G_2} a_i(x,z_k) \approx_\e 1$ for the second approximation.
Putting \eqref{eq:BabycComputation1} into \eqref{eq:BabyComputation2} we obtain
\begin{eqnarray}
\notag
g'(x,t) &\approx_{2\e}& \Ad(1_D \otimes u)(\diag(1_D \otimes f(z_1), \dots, 1_D \otimes f(z_k)) \\
&\stackrel{\eqref{eq:barXiFactors}}=& 1_D \otimes \theta(f)(x) \\
&\stackrel{\eqref{eq:BabyVersionF}}=& 1_D \otimes g(x,t).
\end{eqnarray}
Next, let $(x,t) \in X_1 \times [\eta,2\eta)$.
Express $\hat\theta(x)$ once again as in \eqref{eq:BabyHatThetaExp}.
For $i\in G_2$ and $j=1,\dots,q(i)$, the computation \eqref{eq:BabyComputation1} remains valid, and so
\begin{equation}
\sum_{i\in G_2} \sum_{j=1}^{q(i)} (\bar\xi \otimes \id_{\mathrmm C([0,\eta]) \otimes D})(b_{i,j} \otimes f(y_{i,j}))(x,t) =
\Ad(1_D \otimes u)(\diag(c_1,\dots,c_l)
\end{equation}
where
\begin{eqnarray}
\notag
c_k &:=& \sum_{i\in G_2} \sum_{j=1}^{q(i)} b_{i,j}((x,z_k),t) \otimes f(y_{i,j}) \\
\notag
&\approx_\e& \sum_{i\in G_2} \sum_{j=1}^{q(i)} b_{i,j}((x,z_k),t) \otimes f(z_k) \\
&\stackrel{\eqref{eq:BabybFact}}=& \sum_{i\in G_2} a_i(x,t) \otimes f(z_k),
\end{eqnarray}
using the fact that $f$ is approximately constant on each set in $\mathcal U_0$, in the approximation.
Thus we have
\begin{eqnarray}
\notag
&& \sum_{i\in G_2} \sum_{j=1}^{q(i)} (\bar\xi \otimes \id_{\mathrmm C([0,\eta]) \otimes D})(b_{i,j} \otimes f(y_{i,j}))(x,t) \\
\notag
&\approx_\e&
\sum_{i\in G_2} a_i(x,t) \otimes \Ad(u)(\diag(f(z_1),\dots,f(z_k)) \\
\notag
&=& \sum_{i \in G_2} a_i(x,t) \otimes \theta(f)(x) \\
&\stackrel{\eqref{eq:barXiFactors}}=& \sum_{i \in G_2} a_i(x,t) \otimes g(x,t).
\end{eqnarray}

Using this and the fact that $\sum_{i\in G_2 \cup G_3} a_i(x,t) \approx_\e 1_D$, we obtain
\begin{align}
\notag
g'(x,t) &= \sum_{i \in G_2} \sum_{j=1}^{q(i)} (\bar\xi \otimes \id_{\mathrmm C([0,\eta]) \otimes D})(b_{i,j} \otimes f(y_{i,j}))(x,t) \\
\notag
&\quad + \sum_{i \in G_3} a_i(x,t) \otimes g(z_i) \\
\notag
&\approx_\e \sum_{i \in G_2 \cup G_3} a_i(x,t) \otimes g(x,t) \\
&\approx_\e 1_D \otimes g(x,t).
\end{align}

Finally, for $(x,t) \in X_1 \times [2\eta,1]$, we have
\begin{equation}
g'(x,t) = \sum_{i \in G_3} a_i(x,t) \otimes g(z_i) \approx_{2\e} 1_D \otimes g(x,t)
\end{equation}
This concludes the verification that $1_D \otimes (f,g) \approx_{3\e} \phi\psi(f,g)$.

This shows that $\dr(A \otimes 1_D \subset A \otimes D) \leq 2$.
Since $D$ is UHF, it follows that $\dr(A \otimes D) \leq 2$.
\end{proof}

\subsection{Set-up}
\label{sec:Setup}

Let us fix an NC cell complex $A$; below we shall fix a decomposition of it.
In subsequent sections, we shall produce a construction based on this fixed data.

For a topological space $X$, let $CX$ denote the cone over $X$, defined by
\begin{equation} X \times [0,1]/\sim \end{equation}
where $(x,1) \sim (x',1)$ for all $x,x' \in X$.
Note that this is the commutative space corresponding to the unitization of the standard $\mathrmm C^*$-algebraic cone over $\mathrmm C(X)$.
We view $CX$ as the (non-topological) disjoint union of $X \times [0,1)$ and $\{1\}$.

Let $A$ have an NC cell decomposition, in which it is the $k_0$-th stage, as follows.
The first stage is $A_0 = \mathrmm C(CX_0,M_{m_0})$ and, for $k \geq 1$, the $k$-th stage is $A_k$ which satisfies the pull-back diagram
\begin{equation}
\xymatrix{
A_k \ar[d]_{\lambda_k} \ar[r]^-{\sigma_k} & \mathrmm C(CX_k,M_{m_k}) \ar[d]^-{f \mapsto f|_{X_k \times \{0\}}} \\
A_{k-1} \ar[r]_-{\theta_k} & \mathrmm C(X_k,M_{m_k}),
}
\end{equation}
where each $X_k$ is a sphere of some dimension.
Thus, $A=A_{k_0}$.
Occasionally, it will be convenient to use $\sigma_0\colon A_0 \to \mathrmm C(CX_0,M_{m_0})$ to denote the identity map on $\mathrmm C(CX_0,M_{m_0})$.

We will frequently use the notation 
\begin{equation}
\lambda_k^{k'}:= \lambda_{k'+1} \circ \cdots \circ \lambda_k\colon A_k \to A_{k'}
\end{equation}
for $k' \leq k$.

Presently we will initiate a construction that will lead to c.p.c.\ approximations for the inclusion $A_{k_0} \to A_{k_0} \otimes D$.
These will involve tuples of paramaters $(\eta_1,\dots,\eta_{k_0}) \in (0,1)^{k_0}$.
We will frequently abbreviate this tuple as $\bar\eta$, and also use $\bar\eta$ to denote $(\eta_1,\dots,\eta_k)$ for values $k<k_0$ where appropriate.

Throughout this section, we will use the notation $x_k$ (or $x_k'$, etc.) to denote a point in $X_k$, $v_k$ to denote a point in $CX_k$, and similar notation for points in sets $Z_{k,\bar\eta}$ and $Y_{k,\bar\eta}$ to be introduced in Section \ref{sec:ZY}.
This is not to be confused with the $k^{\text{th}}$ entry from a tuple or sequence of points.
When we wish to refer to an element of a tuple or sequence of points in $X_k$, we use the notation $x_k^{(i)}$ (where the tuple or sequence is indexed over $i$).

\begin{remark}
In the construction to follow, it is not important that the spectrum of $A_0$ is a cone; nor is it important that $X_k$ is a sphere.
It only matters that at each stage (after $A_0$), the space is a cone where the gluing occurs at the ``wide end'' of the cone.
(In fact, what is pertinent is that the gluing space is a neighbourhood retract.)
However, no stronger theorem (than Theorem \ref{thm:MainThmA}) is obtained by replacing $CX_0$ and $X_k$ by more general spaces.
\end{remark}

\subsection{Approximating subalgebras}
\label{sec:ApproxSubalgs}
Let us now define approximating subalgebras $A_{k,(\eta_1,\dots,\eta_k)}$ of $A_k$, parametrized by $\bar\eta=(\eta_1,\dots,\eta_k) \in (0,1)^k$.
Set
\begin{equation} A_{0,()} := A_0. \end{equation}
Having defined $A_{k-1,\bar\eta}$, define $A_{k,\bar\eta}$ to consist of $a \in A_k$ such that
\begin{align}
\notag
\lambda_k(a) &\in A_{k-1,\bar\eta},\text{ and} \\
\sigma_k(a)(x,t) &= \sigma_k(a)(x,0),\text{ for all }x \in X, t \in [0,\eta_k].
\label{eq:ApproxSubalgs}
\end{align}
Evidently, $A_k$ is locally approximated by the subalgebras $A_{k,\bar\eta}$.

\subsection{Construction of $Z_{k,\bar\eta}$}
\label{sec:ZY}
Let us now construct Hausdorff spaces $Y_{k,(\eta_1,\dots,\eta_{k-1})}$ (for $k \geq 1$ and $\eta_1,\dots,\eta_{k-1} \in (0,1)$) and $Z_{k,(\eta_1,\dots,\eta_k)}$ (for $k \geq 0$ and $\eta_1,\dots,\eta_{k-1} \in (0,1)$) recursively, together with a map
\begin{equation} S_{k,(\eta_1,\dots,\eta_k),m}\colon \mathrm{Hom}(A_k,M_m) \to \{\text{finite subsets of } Z_{k,\bar\eta}\}; \end{equation}
we emphasize that the values of this map are sets, not multisets.
Define
\begin{equation} Z_{0,()} := CX_0, \end{equation}
and for $\alpha \in \mathrm{Hom}(A_0,M_m)$, define $S_{0,(),m}(\alpha)$ to consist of all $x_0 \in CX_0$ for which $\ev_{x_0}$ is equivalent to a subrepresentation of $\alpha$.

Now, having defined $Z_{k-1,\bar\eta}$ and $S_{k-1,\bar\eta,m}$, set
\begin{equation}
Y_{k,\bar\eta} := \{(x_k,z_{k-1}) \in X_k \times Z_{k-1,\bar\eta} \mid z_{k-1} \in S_{k-1,\bar\eta,m_k}(\hat \theta_k(x_k))\},
\end{equation}
where we recall that $\hat \theta_k\colon X_k \to \mathrm{Hom}(A_{k-1},M_{m_k})$ is the continuous map induced by $\theta_k\colon A_{k-1} \to \mathrmm C(X_k,M_{m_k})$.
Next, for $\eta_k > 0$, set
\begin{align}
\notag
Z_{k,\bar\eta} := &\big((Z_{k-1,\bar\eta} \times \{0\})\, \amalg\, (Y_{k,\bar\eta} \times [0,\eta_k])\\
&\qquad \amalg\, (X_k \times [\eta_k,1])\big)/\sim,
\end{align}
where $\sim$ is the equivalence relation generated by the following:
\begin{enumerate}
\item
For every $z_{k-1} \in Z_{k-1,\bar\eta}$ and $(x_k,z_{k-1}) \in Y_{k,\bar\eta}$,
\begin{equation} (z_{k-1},0) \sim ((x_k,z_{k-1}),0); \end{equation}
\item
For every $x_k \in X_k$ and $(x_k,z_{k-1}) \in Y_{k,\bar\eta}$,
\begin{equation} ((x_k,z_{k-1}),\eta_k) \sim (x_k,\eta_k); \quad\text{and} \end{equation}
\item
For every $x_k,x_k' \in X_k$,
\begin{equation} (x_k,1) \sim (x_k',1). \end{equation}
\end{enumerate}
As this equivalence relation is closed, $Z_{k,\bar\eta}$ is a Hausdorff space.
Identify $Z_{k-1,\bar\eta} \times \{0\}$, $Y_{k,\bar\eta} \times (0,\eta_k)$, $X_k \times [\eta_k,1)$, and $\{1\}$ in the natural way with subsets of $Z_{k,\bar\eta}$; note that $Z_{k,\bar\eta}$ is in fact the (non-topological) disjoint union of these four subsets.

Now let us define maps
\begin{equation} S_{k,\bar\eta,m}\colon \mathrm{Hom}(A_k,M_m) \to \{\text{finite subsets of } Z_{k,\bar\eta}\}. \end{equation}
For $\alpha \in \mathrm{Hom}(A_k,M_m)$, let $S_{k,\bar\eta,m}(\alpha)$ be the set consisting of all $z_k \in Z_{k,\bar\eta}$ for which one of the following properties obtains:

(i) $z_k = 1$ and $\ev_1 \circ \sigma_k$ is equivalent to a subrepresentation of $\alpha$;

(i)$'$ $z_k = (x_k,t) \in X_k \times [\eta_k,1)$ and $\ev_{(x_k,t)} \circ \sigma_k$ is equivalent to a subrepresentation of $\alpha$;

(ii) $z_k = ((x_k,z_{k-1}),t) \in Y_{k,\bar\eta} \times (0,\eta_k)$ and $\ev_{(x_k,t)} \circ \sigma_k$ is equivalent to a subrepresentation of $\alpha$; or

(iii) $z_k = (z_{k-1},0) \in Z_{k,\bar\eta}$ and $\alpha$ contains a subrepresentation that factors as
\begin{equation} A_k \labelledrightarrow{\lambda_k} A_{k-1} \labelledrightarrow{\alpha'} M_{m'}, \end{equation}
such that $z_{k-1} \in S_{k-1,\bar\eta,m'}(\alpha')$.

This concludes the inductive definition of the spaces $Y$, $Z$, and the maps $S$.

We may regard $Y_{k,\bar\eta}$ as an ``$X$-fibred'' space, with its fibres being finite subsets of $Z_{k-1,\bar\eta}$.
For a subset $A$ of $Y_{k,\bar\eta}$ and for $x_k \in X_k$, let us write $A(x_k)$ to denote the set
\begin{equation} \{z_{k-1}\in Z_{k-1,\bar\eta} \mid (x_k,z_{k-1}) \in A\}, \end{equation}
and for $B \subseteq X_k$, let us also consider the subset 
\begin{equation}
A|_B := \{(z_{k-1},x_k) \mid x_k \in B\text{ and }z_{k-1} \in A(x_k)\} \subseteq A.
\end{equation}
(Note that, while $A(x_k)$ can be canonically identified with $A|_{\{x_k\}}$, they are technically different.)
Let us also regard $Z_{k,\bar\eta}$ and $CX_k$ as $[0,1]$-fibred spaces, and use similar notation: for $A \subseteq Z_{k,\bar\eta}$, 
\begin{align}
\notag
A(0) &:= \{z_{k-1} \mid (z_{k-1},0) \in A\} \subseteq Z_{k-1,\bar\eta}, \\
\notag
A(t) &:= \{(x_k,z_{k-1}) \mid ((x_k,z_{k-1}),t) \in A\} \subseteq Y_{k,\bar\eta},\quad t \in (0,\eta), \\
\notag
A(t) &:= \{x_k \mid (x_k,t) \in A\} \subseteq X_k, \quad t \in [\eta,1), \text{ and}\\
A(1) &:= \{*\};
\end{align}
and for $B \subseteq [0,1]$,
\begin{equation}
A|_B := \{(w,t) \mid t \in [0,1]\text{ and }w \in A(t)\} \subseteq A.
\end{equation}
It is clear how to define similar notation for subsets of $CX_k$.

\begin{remark}
\label{rmk:ZAltDesc}
A more explicit description of the points in $Z_{k,\bar\eta}$ is as follows.
The points of $Z_{k,\bar\eta}$ are tuples $(v_{k_1},\dots,v_{k_p})$ where
\begin{enumerate}
\item $v_{k_i} \in CX_{k_i}$ with $0 \leq k_1 < \cdots < k_p \leq k$;
\item $v_{k_1} \in CX_{k_1}|_{[\eta_{k_1},1]}$ (or $CX_0$ in case $k_1=0$);
\item $v_{k_i}=(x_{k_i},t_i) \in CX_{k_i}|_{(0,\eta_{k_i})}$ for $i=1,\dots,p$; and
\item $\ev_{v_{k_i}} \circ \sigma_{k_i} \circ \lambda_{k_{i+1}}^{k_i}\colon A_{k_{i+1}-1} \to M_{k_i}$ is a subrepresentation of $\theta_{k_{i+1}}(x_{k_{i+1}})$ for $i=0,\dots,p-1$.
\end{enumerate}
We will refer to this description at times, although we do not use it exclusively because the topology is more accessible using the other description.

Under this description, for $\alpha \in \Hom(A_k,M_m)$, the set $S_{k,\bar\eta,m}(\alpha)$ consists of $z_k=(v_{k_1},\dots,v_{k_p})$ for which $\ev_{v_{k_p}} \circ \sigma_{k_p} \circ \lambda_k^{k_p}$ is equivalent to a subrepresentation of $\alpha$.
\end{remark}

\begin{lemma}
\label{lem:SContinuous}
Each map 
\begin{equation}
S_{k,\bar\eta,m}\colon \mathrm{Hom}(A_k,M_m) \to \{\mathrm{finite\ subsets\ of}\ Z_{k,\bar\eta}\}
\end{equation}
is continuous, where we use the Hausdorff metric on the set of finite subsets of $Z_{k,\bar\eta}$.
\end{lemma}

\begin{proof}
Implicitly, we are fixing a metric $d$ on $Z_{k,\bar\eta}$, giving rise to a Hausdorff metric.
(By compactness, the topology does not depend on this choice of metric).

\textbf{Step 1.}
Let $\alpha \in \mathrm{Hom}(A_k,M_m)$ and let $\e >0$; in this step, we show that for each $z_k \in S_{k,\bar\eta,m}(\alpha)$, there is a neighbourhood $U$ of $\alpha$ such that, for all $\alpha' \in U$, there exists $z_k' \in S_{k,\bar\eta,m}(\alpha')$ with $d(z_k,z_k')<\e$.
We prove this inductively on $k$, thereby allowing us to assume it is true for $k-1$.
(The base case, where $k=0$, is straightforward.)

Consider cases:

(i) If $z_k=1$ or $z_k=(x_k,t)$ for $t \in (\eta_k,1)$, then this means that $\ev_{z_k} \circ \sigma_k$ is equivalent to a subrepresentation of $\alpha$.
We may then choose $U$ to be a neighbourhood of $\alpha$ such that every $\alpha' \in U$ contains some subrepresentation $\ev_{z_k'} \circ \sigma_k$, where $z_k' \in CX_k|_{(\eta_k,1]}$ and $d(z_k,z_k') < \e$ (inside $Z_{k,\bar\eta}$).
Thus, for such $\alpha'$, $z_k' \in S_{k,\bar\eta,m}(\alpha')$.

(ii) If $z_k=(x_k,\eta_k)$, then this means that $\ev_{z_k} \circ \sigma_k$ is equivalent to a subrepresentation of $\alpha$.
This is essentially the same as Case (i), but we need to be a bit more careful.
Choose $U_x$ and $U_t$ to be neighbourhoods of $x_k$ and $t$ inside $X_k$ and $(0,1)$ respectively, such that the set
\begin{equation} \left(Y_{k,\bar\eta}|_{U_x} \times (U_t \cap (0,\eta_k))\right) \cup \left(U_x \times (U_t \cap [\eta_k,1)\right)\end{equation}
is contained in the ball of radius $\e$ about $z_k$.
We may then choose $U$ to be a neighbourhood of $\alpha$ such that every $\alpha' \in U$ contains some subrepresentation $\ev_{(x_k',t')} \circ \sigma_k$, where $x_k' \in U_x$ and $t' \in U_t$.
Thus, for such $\alpha'$, if $t' \geq \eta_k$ then $z_k':=(x_k',t') \in S_{k,\bar\eta,m}(\alpha')$ and $d(z_k,z_k') < \e$.
Otherwise, $t' < \eta_k$, and for any $z_{k-1} \in S_{k-1,\bar\eta,m_k}(\hat\theta_k(x_k))$, we have
\begin{equation} z_k' := ((x_k',z_{k-1}),t') \in S_{k,\bar\eta,m}(\alpha') \end{equation}
and $d(z_k,z_k') < \e$.

(iii) If $z_k=((x_k,z_{k-1}),t) \in Y_{k,\bar\eta} \times (0,\eta_k)$, then this means that $\ev_{(x_k,t)} \circ \sigma_k$ is equivalent to a subrepresentation of $\alpha$ and that $z_{k-1} \in S_{k-1,\bar\eta,m_k}(\hat\theta_k(x_k))$.
Let $U_x,U_z$, and $U_t$ be neighbourhoods of $x_k,z_{k-1}$, and $t$ inside $X_k, Z_{k-1,\bar\eta}$, and $(0,\eta_k)$ respectively, such that
\begin{equation} \left((U_x \times U_z) \cap Y_{k,\bar\eta}\right) \times U_t \end{equation}
is contained in the ball of radius $\e$ about $z_k$.

By induction and by continuity of $\hat\theta_k$, there is a neighbourhood $U_x'$ of $x_k$ inside $X_k$ such that, for $x_k' \in U_x'$, there exists $z_{k-1}' \in S_{k-1,\bar\eta,m_k}(\hat\theta_k(x_k))$ such that $z_{k-1}' \in U_z$.
We may suppose that $U_x' \subseteq U_x$.
We may then choose $U$ to be a neighbourhood of $\alpha$ such that every $\alpha' \in U$ contains some subrepresentation $\ev_{(x_k',t')} \circ \sigma_k$, where $x_k' \in U'_x$ and $t' \in U_t$.
For such $\alpha'$, it follows that there exists $z_k':=((x_k',z_{k-1}'),t') \in S_{k,\bar\eta,m}(\alpha)$ with $d(z_k',z_k)<\e$.

(iv) If $z_k=(z_{k-1},0) \in Z_{k-1,\bar\eta} \times \{0\}$, then this means that $\alpha$ contains a subrepresentation that factors as
\begin{equation} A_k \labelledrightarrow{\lambda_k} A_{k-1} \labelledrightarrow{\beta} M_{n}, \end{equation}
in such a way that $z_{k-1} \in S_{k-1,\bar\eta,n}(\beta)$.

By induction, we may pick a neighbourhood $V$ of $\beta$ in $\Hom(A_{k-1},M_n)$ such that for every $\beta' \in V$, there exists $z_{k-1}' \in S_{k-1,\bar\eta,n}(\beta')$ with 
\begin{equation} d((z_{k-1},0),(z_{k-1}',0))<\e/2 \end{equation}
(inside $Z_{k,\bar\eta}$).

If $\beta$ is not contained as a subrepresentation of $\hat\theta_k(x_k)$ for any $x_k$, then we may choose $U$ to be a neighbourhood of $\alpha$ such that every $\alpha' \in U$ contains a subrepresentation that factors as
\begin{equation} A_k \labelledrightarrow{\lambda_k} A_{k-1} \labelledrightarrow{\beta'} M_{n}, \end{equation}
with $\beta' \in V$.

Otherwise, $\beta$ is contained as a subrepresentation of $\hat\theta_k(x_k)$ for some $x_k$ (possibly more than one), and we therefore have a gluing
\begin{equation} (z_{k-1},0) \sim ((x_k,z_{k-1}),0) \end{equation}
for any such $x_k$.
We may then choose $U$ to be a neighbourhood of $\alpha$ such that for every $\alpha' \in U$, one of the following holds:

(a) $\alpha'$ contains a subrepresentation that factors as
\begin{equation} A_k \labelledrightarrow{\lambda_k} A_{k-1} \labelledrightarrow{\beta'} M_{n}, \end{equation}
with $\beta' \in V$.
In this case, there exists $z_{k-1}' \in S_{k-1,\bar\eta,n}(\beta')$ with 
\begin{equation} d((z_{k-1},0),(z_{k-1}',0))<\e/2. \end{equation}
By construction, $(z_{k-1}',0) \in S_{k,\bar\eta,m}(\alpha')$.

(b) $\alpha'$ contains a subrepresentation of the form $\ev_{(x_k',t)} \circ \sigma_k$ where $t \in (0,\eta_k)$ and $\hat\theta_k(x_k')$ contains a subrepresentation that factors as
\notag
\begin{equation} A_k \labelledrightarrow{\lambda_k} A_{k-1} \labelledrightarrow{\beta'} M_{n}, \end{equation}
with 
$\beta' \in V$, and therefore there exists $z_{k-1}' \in S_{k-1,\bar\eta,n}(\beta')$ with 
\begin{equation}
d((z_{k-1},0),(z_{k-1}',0))<\e/2.
\end{equation}
We may also ask (by making $U$ small enough) that in this case, $d(((x_k',z_{k-1}'),t),(z_{k-1}',0))<\e/2$.
By construction, $z_k':=((x_k',z_{k-1}'),t) \in S_{k,\bar\eta,m}(\alpha')$ and
\begin{equation} d(z_k',z_k) \leq d(z_k',(z_{k-1}',0)) + d((z_{k-1}',0)),z_k) < \e/2+\e/2=\e.\end{equation}

This concludes Step 1.

\textbf{Step 2.}
Let $\alpha \in \mathrm{Hom}(A_k,M_m)$.
To show continuity at $\alpha$, we must show that for $\e > 0$, there is a neighbourhood $U$ of $\alpha$ such that, for all $\alpha' \in U$, the Hausdorff distance from $S_{k,\bar\eta,m}(\alpha)$ to $S_{k,\bar\eta,m}(\alpha')$ is at most $\e$.
This can be done using Step 1 and a pigeonhole principle argument.

Namely, we may associate a non-zero dimension to each $z_k \in S_{k,\bar\eta,m}(\alpha)$ such that the sum of these dimensions (counted with multiplicity) is exactly $m$.
To be precise, if $z_k = (v_{k_1},\dots,v_{k_p})$ as in Remark \ref{rmk:ZAltDesc}, then the ``dimension'' of $z_k$ is $D(z_k):=m_{k_1}$.
Then, with a bit of care, in Step 1, we achieve, for each $z_k \in S_{k,\bar\eta,m}(\alpha)$, a neighbourhood $U_{z_k}$ of $\alpha$ such that for $\alpha' \in U_{z_k}$, there exist $z_k^{(1)},\dots,z_k^{(q)} \in S_{k,\bar\eta,m}(\alpha')$ for some $q$ (counting with multiplicity) with $d(z_k,z_k^{(i)}) < \e$ for each $i$, and such that $D(z_k^{(1)}) + \cdots + D(z_k^{(q)}) = D(z_k)$.
Set
\begin{equation} \textstyle{U := \bigcap_{z_k \in S_{k,\bar\eta,m}(\alpha)} U_{z_k}}. \end{equation}
Then the pigeonhole principle implies that for $\alpha' \in U$, every $z_k' \in S_{k,\bar\eta,m}(\alpha')$ is of the form $z_k^{(i)}$ as above for some $z_k \in S_{k,\bar\eta,m}(\alpha)$, and so $d(z_k,z_k') < \e$.
\end{proof}

\subsection{Pointwise finite dimensional approximations}
\label{sec:PointwiseApprox}

For $z_k =(v_{k_1},\dots,v_{k_p}) \in Z_{k,\bar\eta}$ (using the notation of Remark \ref{rmk:ZAltDesc}), define $F_{k,\bar\eta,z_k}:=M_{m_{k_1}}$ and define the ${}^*$-homomorphism $\nu_{k,\bar\eta,z_k}:A_k \to F_{k,\bar\eta,z_k}$ by
\begin{equation}
\label{eq:nuDef1}
\nu_{k,\bar\eta,z_k}(a) := \big(\sigma_{k_1} \circ \lambda_k^{k_1}(a)\big)(v_{k_1}).
\end{equation}

Next, we construct, for each $\alpha \in \mathrm{Hom}(A_k,M_m)$ and $z_k \in S_{k,\bar\eta,m}(\alpha)$, a ${}^*$-homomorphism $\mu_{k,\bar\eta,\alpha,z_k}\colon F_{k,\bar\eta,z_k} \to M_m$.
This gives us two ${}^*$-ho\-mo\-mor\-phisms
\begin{equation}
\label{eq:munuDef1}
 A_k
\labelledrightarrow
{\nu_{k,\bar\eta,z_k}}
F_{k,\bar\eta,z_k} 
\labelledrightarrow
{\mu_{k,\bar\eta,\alpha,z_k}}
M_m;
\end{equation}
ultimately, our c.p.c.\ approximations used to prove Theorem \ref{thm:drBound} will be built out of these two maps; in fact, the c.p.c.\ approximations will look like
\begin{equation}
\xymatrix{
A_{k_0} \ar[dr]_-{\psi} \ar[rr]^-{\id_{A_{k_0}} \otimes 1_D} && A_{k_0} \otimes D, \\
& F \ar[ur]_-{\phi} &
}
\end{equation}
where $F$ will be a direct sum of certain $F_{k_0,\bar\eta,z_{k_0}}$ and $\psi$ will be a corresponding direct sum of the $\nu_{k_0,\bar\eta,z_{k_0}}$.
The components of the map $\phi$ will be built out of $\sigma_k$ and certain $\mu_{k-1,\bar\eta,\hat\theta_k(x_k),z_{k-1}}$.

The construction of $\mu_{k,\bar\eta,\alpha,z_k}$ is recursive in $k$.

For $k=0$, $A_0 = \mathrmm C(CX_0,M_{m_0})$ and if $z_0 \in S_{0,(),m}(\alpha)$ then this means that $z_0 \in CX_0$ and
$\ev_{z_0}$ is equivalent to a subrepresentation of $\alpha$.
We have
$F_{0,(),z_0} = M_{m_0}$ and $\nu_{0,(),z_0} = \ev_{z_0}$.
Denote by $\alpha'$ the largest subrepresentation of $\alpha$ that is equivalent to a multiple of $\ev_{z_0}$, which means that there exists a unique $\mu_{0,(),\alpha,z_0}$ such that \eqref{eq:munuDef1} is a factorization of $\alpha'$.

For $k >0$, the definition breaks into cases.

(i) $z_k \in Z_{k,\bar\eta}|_{[\eta_k,1]} = CX_k|_{[\eta_k,1]}$:
This means that $\ev_{z_k} \circ \sigma_k$ is equivalent to a subrepresentation of $\alpha$.
We handle this in exactly the same way as for $k=0$.
We have $F_{k,\bar\eta,z_k} = M_{m_k}$ and $\nu_{k,\bar\eta,z_k} = \ev_{z_k} \circ \sigma_k$.
Let $\alpha'$ denote the largest subrepresentation of $\alpha$ that is equivalent to a multiple of $\ev_{z_k} \circ \sigma_k$, and we define $\mu_{k,\bar\eta,\alpha,z_k}$ to be the unique ${}^*$-homomorphism for which \eqref{eq:munuDef1} is a factorization of $\alpha'$; thus we have a commuting triangle
\begin{equation}
\label{eq:muDef1}
\xymatrix{
A_k \ar[rr]^-{\alpha'} \ar[rd]_-{\nu_{k,\bar\eta,z_k}} && M_m. \\
& F_{k,\bar\eta,z_k} \ar[ru]_-{\mu_{k,\bar\eta,\alpha,z_k}} &
}
\end{equation}

(ii) $z_k = ((x_k,z_{k-1}),t) \in Y_{k,\bar\eta} \times (0,\eta_k)$:
In this case, $F_{k,\bar\eta,z_k} = F_{k-1,\bar\eta,z_{k-1}}$.
Also, $z_{k-1} \in S(\hat \theta_k(x_k))$, so by recursion, we have a map $\mu_{k-1,\bar\eta,\hat \theta_k(x_k), z_{k-1}}\colon F_{k,\bar\eta,z_k} \to M_{m_k}$.
We also know that $\ev_{(x_k,t)}$ is equivalent to a subrepresentation $\alpha$.
Let $\alpha'$ denote the largest subrepresentation that is equivalent to a multiple of $\ev_{(x_k,t)}$, and factorize it as
\begin{equation}
\label{eq:muDefmu'Def}
 A_k \labelledrightarrow{ev_{(x_k,t)}}  M_{m_k} \labelledrightarrow{\mu'}  M_m.
\end{equation}
Now, set
\begin{align}
\label{eq:muDef2}
\mu_{k,\bar\eta,\alpha,z_k} &:= \mu' \circ \mu_{k-1,\bar\eta,\hat \theta_k(x_k),z_{k-1}}\colon F_{k,\bar\eta,z_k} \to M_m.
\end{align}

Diagrammatically (and leaving out some subscripts), we have
\begin{equation}
\label{eq:munuDefCommute2}
\xymatrix{
A_k \ar[r]^-{\nu_k} \ar[d]_{\lambda_k} & F_k=F_{k-1} \ar[r]^-{\mu_k} \ar[dr]_{\mu_{k-1}} & M_m \\
A_{k-1} \ar[ur]_{\nu_{k-1}} \ar[rr]_{\alpha'} && M_{m_k}, \ar[u]^-{\mu'}
}
\end{equation}
where the bottom triangle approximately commutes (it exactly commutes when restricted to $A_{k-1,\bar\eta}$, as defined in Section \ref{sec:ApproxSubalgs}).

(iii) $z_k = (z_{k-1},0) \in Z_{k,\bar\eta}$.
In this case, $F_{k,\bar\eta,z_k} = F_{k-1,\bar\eta,z_{k-1}}$ and $\nu_{k,\bar\eta,z_k} = \nu_{k-1,\bar\eta,z_{k-1}} \circ \lambda_k$.
Then, $\alpha$ contains a subrepresentation that factors as
\begin{equation} A_k \labelledrightarrow{\lambda_k} A_{k-1} \labelledrightarrow{\alpha'} M_{m'} \labelledrightarrow{\mu'} M_m, \end{equation}
such that $\alpha'$ is unital, $\mu'$ is injective (usually non-unital), in such a way that $z_{k-1} \in S_{k-1,\bar\eta,m'}(\alpha')$.
Take the largest subrepresentation of this form, i.e., let $m'$ be maximal satisfying these conditions (note that this uniquely determines $\alpha'$, up to unitary equivalence).

By recursion we have a map $\mu_{k-1,\bar\eta,\alpha',z_{k-1}}\colon F_{k,\bar\eta,z_k} \to M_{m'}$.
Set
\begin{align}
\label{eq:muDef3}
\mu_{k,\bar\eta,\alpha,z_k} &:= \mu' \circ \mu_{k-1,\bar\eta,\alpha',z_{k-1}}\colon F_{k,\bar\eta,z_k} \to M_m.
\end{align}
Diagrammatically, we have the commuting diagram \eqref{eq:munuDefCommute2} (with $m'$ in place of $m_k$), except that in this case, all triangles commute exactly.

\begin{lemma}
\label{lem:PointwiseApprox}
For $\alpha \in \mathrm{Hom}(A_k,M_m)$, and $a \in A_{k,\bar\eta}$ (as defined in Section \ref{sec:ApproxSubalgs}),
\begin{equation}
\alpha(a) = \sum_{z_k \in S_{k,\bar\eta,m}(\alpha)} \mu_{k,\bar\eta,\alpha,z_k} \circ \nu_{k,\bar\eta,z_k}(a).
\end{equation}
In particular,
\begin{equation}
\label{eq:PointwiseApproxMult}
\mu_{k,\bar\eta,\alpha,z_k}(1_{F_{k,\bar\eta,z_k}})\cdot\alpha(a) = \mu_{k,\bar\eta,\alpha,z_k} \circ \nu_{k,\bar\eta,z_k}(a)
\end{equation}
and
\begin{equation}
\label{eq:muOrthog}
(\mu_{k,\bar\eta,\alpha,z_k}(1_{F_{k,\bar\eta,z_k}}))_{z_k \in S_{k,\bar\eta,m}(\alpha)}
\end{equation}
is a partition of unity in $M_m$ (consisting of orthogonal projections).
\end{lemma}

\begin{proof}
This is proven by induction on $k$.
In the case $k=0$, it is trivial.

Suppose that the statement holds for $k-1$.
Let $\alpha \in \mathrm{Hom}(A_k,M_m)$.
The map $\alpha$ decomposes as follows:
\begin{equation} \alpha = \Ad(u) \circ
\diag(\ev_{v_1} \circ \sigma_k, \ev_{v_p} \circ \sigma_k, \alpha' \circ \lambda_k),
\end{equation}
where $v_1,\dots,v_p \in CX_k|_{(0,1]}$ and $\alpha' \in \mathrm{Hom}(A_{k-1},M_{m'})$ for some $m'$.

By possibly reordering, arrange that 
\begin{equation} v_1,\dots,v_{p'} \in CX_k|_{[\eta_k,1]}\quad\text{and}\quad v_{p'+1},\dots,v_p \in CX_k|_{(0,\eta_k)}. \end{equation}
Let us assume that $v_1,\dots,v_p$ are distinct; this is only to simplify notation---if the $v_i$ are not distinct, the following argument still works provided we keep track of multiplicities.

For $i=1,\dots,p'$, we have $v_i \in S_{k,\bar\eta,m}(\alpha)$ and 
\begin{align}
\notag
F_{k,\bar\eta,v_i} &= M_{m_k}, \\
\notag
\nu_{k,\bar\eta,v_i} &= \ev_{v_i} \circ \sigma_k, \\
\notag
\mu_{k,\bar\eta,\alpha,v_i} &= \Ad(u) \circ \\
&\qquad \diag(0_{m_k},\dots,0_{m_k},\id_{M_{m_k}},0_{m_k},\dots,0_{m_k},0_{m'}),
\end{align}
by \eqref{eq:muDef1}.
Thus, for $a \in A_k$,
\begin{equation}
\label{eq:PointwiseApprox1}
\mu_{k,\bar\eta,\alpha,v_i} \circ \nu_{k,\bar\eta,v_i}(a) = e_i \alpha(a),
\end{equation}
where
\begin{equation}
\label{eq:PointwiseApproxeiDef}
 e_i := \Ad(u)(\diag(0_{m_k},\dots,0_{m_k},1_{m_k},0_{m_k},\dots,0_{m_k},0_{m'})),
\end{equation}
with the $1_{m_k}$ appearing in the $i$th position.

For $i=p'+1,\dots,p$, write $v_i=(x^{(i)}_k,t_i)$.
Then for every $z_{k-1}\in S_{k-1,\bar\eta,m_k}(\hat\theta_k(x^{(i)}_k))$,
\begin{equation}
 ((x^{(i)}_k,z_{k-1}),t_i) \in S_{k,\bar\eta,m}(\alpha). \end{equation}
By induction, for $a \in A_{k,\bar\eta}$, since $\lambda_k(a) \in A_{k-1,\bar\eta}$,
\begin{align}
\label{eq:PointwiseApproxthetaDecomp}
\notag
&\hspace*{-10mm}\hat\theta_k(x^{(i)}_k)(\lambda_k(a)) \\
&\ \ = \sum_{z_{k-1} \in S_{k-1,\bar\eta,m_k}(\hat\theta_k(x^{(i)}_k))} \mu_{k-1,\bar\eta,\hat\theta_k(x^{(i)}_k),z_{k-1}} \circ \nu_{k-1,\bar\eta,z_{k-1}}(\lambda_k(a)).
\end{align}
By \eqref{eq:muDef2}, 
for $z_{k-1} \in S_{k-1,\bar\eta,m_k}(\hat\theta_k(x^{(i)}_k))$,
\begin{align}
\label{eq:PointwiseApproxmunuIndDef1}
\notag
\nu_{k,\bar\eta,((x^{(i)}_k,z_{k-1}),t_i)} &= \nu_{k-1,z_{k-1}} \circ \lambda_k, \text{ and} \\
\notag
\mu_{k,\bar\eta,\alpha,((x^{(i)}_k,z_{k-1}),t_i)} &= \Ad(u) \circ \diag(0_{m_k},\dots,0_{m_k},\mu_{k-1,\hat\theta_k(x^{(i)}_k),z_{k-1}}, \\
&\qquad\qquad 0_{m_k},\dots,0_{m_k},0_{m'}).
\end{align}
Thus, defining $e_i$ again as in \eqref{eq:PointwiseApproxeiDef}, for $a \in A_{k,\bar\eta}$, we have
\begin{eqnarray}
\notag
e_i\alpha(a) &=& \Ad(u) \circ \diag(0,\dots,0,\ev_{(x^{(i)}_k,t)} \circ \sigma_k(a),0,\dots,0,0) \\
\notag
&\stackrel{\eqref{eq:ApproxSubalgs}}=& \Ad(u) \circ \diag(0,\dots,0,\ev_{x_k^{(i)}} \circ \theta_k\circ\lambda_k(a),0,\dots,0,0) \\
\notag
&\stackrel{\eqref{eq:PointwiseApproxthetaDecomp}}=&
\Ad(u) \circ \diag(0,\dots,0,\\
\notag
&&\sum_{z_{k-1} \in S_{k-1,\bar\eta,m_k}(\hat\theta_k(x^{(i)}_k))} \mu_{k-1,\bar\eta,\hat\theta_k(x^{(i)}_k),z_{k-1}} \circ \nu_{k-1,\bar\eta,z_{k-1}}(\lambda_k(a)), \\
\notag
&&\qquad\qquad 0,\dots,0,0) \\
\label{eq:PointwiseApprox2}
&\stackrel{\eqref{eq:PointwiseApproxmunuIndDef1}}=&
\sum_{z_{k-1} \in S_{k-1,\bar\eta,m_k}(\hat\theta_k(x_k))} \mu_{k,\bar\eta,\alpha,((x_k,z_{k-1}),t)} \circ \nu_{k,\bar\eta,((x_k,z_{k-1}),t)}(a).
\end{eqnarray}

Next let us consider $\alpha'$.
$S_{k,\bar\eta,m}(\alpha)$ contains $S_{k-1,\bar\eta,m'}(\alpha') \times \{0\}$, and by induction, for $a \in A_{k,\bar\eta}$, since $\lambda_k(a) \in A_{k-1,\bar\eta}$,
\begin{equation}
\label{eq:PointwiseApproxalpha'Decomp}
\alpha'(\lambda_k(a)) = \sum_{z_{k-1} \in S_{k-1,\bar\eta,m'}(\alpha')} \mu_{k-1,\bar\eta,\alpha',z_{k-1}} \circ \nu_{k-1,\bar\eta,z_{k-1}} \circ \lambda_k(a).
\end{equation}
By \eqref{eq:muDef3}, 
for $z_{k-1} \in S_{k-1,\bar\eta,m'}(\alpha')$,
\begin{align}
\label{eq:PointwiseApproxmunuIndDef2}
\notag
\nu_{k,\bar\eta,(z_{k-1},0)} &= \nu_{k-1,\bar\eta,z_{k-1}} \circ \lambda_k, \text{ and} \\
\mu_{k,\bar\eta,\alpha,(z_{k-1},0)} &= \Ad(u) \circ \diag(0_{m_k},\dots,0_{m_k},\mu_{k-1,\bar\eta,\alpha',z_{k-1}},).
\end{align}
Define
\begin{equation}
\label{eq:PointwiseApproxe'Def}
e' := \Ad(u) \circ \diag(0_{m_k},\dots,0_{m_k},1_{m'}).
\end{equation}
Then, for $a \in A_{k,\bar\eta}$,
\begin{eqnarray}
\notag
e\alpha(a) &=& \Ad(u) \circ \diag(0,\dots,0,\alpha' \circ \lambda_k(a)) \\
\notag
&\stackrel{\eqref{eq:PointwiseApproxalpha'Decomp}}=&
\Ad(u) \circ \diag(0,\dots,0, \\
\notag
&&\qquad \sum_{z_{k-1} \in S_{k-1,\bar\eta,m'}(\alpha')} \mu_{k-1,\bar\eta,\alpha',z_{k-1}} \circ \nu_{k-1,\bar\eta,z_{k-1}} \circ \lambda_k(a)) \\
\label{eq:PointwiseApprox3}
&\stackrel{\eqref{eq:PointwiseApproxmunuIndDef2}}=&
\sum_{z_{k-1} \in S_{k-1,\bar\eta,m'}(\alpha'} \mu_{k,\bar\eta,\alpha,(z_{k-1},0)} \circ \nu_{k,\bar\eta,(z_{k-1},0)}(a).
\end{eqnarray}

Putting these together, for $a \in A_{k,\bar\eta}$ we have
\begin{eqnarray}
\notag
&&\hspace*{-20mm}\sum_{z_k \in S_{k,\bar\eta,m}(\alpha)} \mu_{k,\bar\eta,\alpha,z_k} \circ \nu_{k,\bar\eta,z_k}(a) \\
\notag
&=& \sum_{i=1}^{p'} \mu_{k,\bar\eta,\alpha,v_i} \circ \nu_{k,\bar\eta,v_i}(a) + \sum_{i=p'+1}^p  \\
\notag
&&\sum_{z_{k-1} \in S_{k-1,\bar\eta,m_k}(\hat\theta_k(x^{(i)}_k))} \mu_{k,\bar\eta,\alpha,((x_k^{(i)},z_{k-1}),t_i)} \circ \nu_{k,\bar\eta,((x_k^{(i)},z_{k-1}),t_i)}(a) \\
\notag
&&\quad +\sum_{z_{k-1} \in S_{k-1,\bar\eta,m'}(\alpha')} \mu_{k,\bar\eta,\alpha,(z_{k-1},0)} \circ \nu_{k,\bar\eta,(z_{k-1},0)}(a) \\
\notag
&\stackrel{\begin{tabular}{@{}c}{\scriptsize \eqref{eq:PointwiseApprox1},\eqref{eq:PointwiseApprox2},} \\ {\scriptsize \eqref{eq:PointwiseApprox3}} \end{tabular}}=&
\sum_{i=1}^{p'} e_i\alpha(a) + \sum_{i=p'+1}^p e_i\alpha(a) + e'\alpha(a) \\
&\stackrel{\begin{tabular}{@{}c} {\scriptsize \eqref{eq:PointwiseApproxeiDef},} \\ {\scriptsize \eqref{eq:PointwiseApproxe'Def}}\end{tabular}}=& \alpha(a).
\end{eqnarray}

Since $\mu_{k,\bar\eta,\alpha,z_k}$ is a ${}^*$-homomorphism, \eqref{eq:muOrthog} is a family of projections.
\end{proof}

\subsection{A partition of $Z_{k,\bar\eta}$}

We now partition $Z_{k,\bar\eta}$ into sets
\begin{equation} Z_{k,\bar\eta,0}^*,\dots,Z^*_{k,\bar\eta,k}, \end{equation}
where we define $Z^*_{k,\bar\eta,m,k'}$ by fixing $k'$ and using recursion in $k$, from $k'$ to $k_0$.

For $k=k'=0$, set
\begin{equation} Z^*_{0,(),0} := Z_{0,()}. \end{equation}
For $k=k'>0$, set
\begin{equation} Z^*_{k,\bar\eta,k} := Z_{k,\bar\eta}|_{[\eta_k,1]}. \end{equation}

For $k>k'$, define $Z^*_{k,\bar\eta,k'}$ to consist of:
\begin{enumerate}
\item $((x_k,z_{k-1}),t) \in Y_{k,\bar\eta} \times (0,\eta_k)$ such that 
\begin{equation} z_{k-1} \in Z^*_{k-1,\bar\eta,k'}; \text{ and} \end{equation}
\item $(z_{k-1},0) \in Z_{k-1,\bar\eta} \times \{0\}$ such that 
\begin{equation} z_{k-1} \in Z^*_{k-1,\bar\eta,k'}. \end{equation}
\end{enumerate}

Likewise, let us define
\begin{align}
\label{eq:Z**Def}
\notag
Z_{0,(),0}^{**} &:= Z_{0,()} \\
\notag
Z_{k,\bar\eta,k}^{**} &:= Z_{k,\bar\eta}|_{(\eta_k,1]},\quad k>0 \quad \text{(compare, $Z_{k,\bar\eta,k}^* = Z_{k,\bar\eta}|_{[\eta_k,1]}$)}, \\
\notag
Z_{k,\bar\eta,k'}^{**} &:= \{((x_k,z_{k-1}),t) \in Y_{k,\bar\eta} \times (0,\eta_k) \mid z_{k-1} \in Z_{k-1,\bar\eta,k'}^{**}\}\\
&\qquad \cup\, \{(z_{k-1},0) \in Z_{k-1,\bar\eta} \times \{0\} \mid z_{k-1} \in Z_{k-1,\bar\eta,k'}^{**}\}.
\end{align}
Observe that $Z_{k,\bar\eta,k'}^{**} \subseteq Z_{k,\bar\eta,k'}^*$ and $Z_{k,\bar\eta,k'}^{**}$ is open in $Z_{k,\bar\eta}$.

\begin{lemma}
\label{lem:ZkPartition}
(i)
In the notation of Remark \ref{rmk:ZAltDesc},
$Z^*_{k,\bar\eta,k'}$ consists of points $z_k=(v_{k_1},\dots,v_{k_p}) \in Z_{k,\bar\eta}$ for which $k_1=k'$; for $k'>0$, the points in $Z^{**}_{k,\bar\eta,k'}$ are the ones that additionally satisfy $v_{k_1} \in CX_{k_1}|_{(\eta_{k_1},1]}$.

(ii)
\begin{equation}
\label{eq:ZkPartition}
 Z_{k,\bar\eta} = \coprod_{k'=0}^k Z_{k,\bar\eta,k'}^*
\end{equation}
\end{lemma}

\begin{proof}
(i) is easily proven by induction on $k$, and (ii) follows immediately.
\end{proof}

The map $(\alpha,z_k) \mapsto \mu_{k,\bar\eta,\alpha,z_k}$ has a certain continuity property, which is expressed in the following lemma.

\begin{lemma}
\label{lem:muContinuous}
For $\alpha \in \mathrm{Hom}(A_k,M_m)$ and
\begin{equation} z_k \in S_{k,\bar\eta,m}(\alpha) \cap Z_{k,\bar\eta,k'}^{**}, \end{equation}
and for $\e > 0$, there exist neighbourhoods $U_{z_k}$ (respectively $U_\alpha$) of $z_k$ ($\alpha$) in $Z_{k,\bar\eta,k'}^{**}$ ($\mathrm{Hom}(A_k,M_m)$) such that the following approximation holds:
For every $\alpha' \in U_\alpha$ and every contraction $\kappa \in M_{m_{k'}}$,
\begin{equation}
\label{eq:muContinuous}
 \mu_{k,\bar\eta,\alpha,z_k}(\kappa) \approx_\e \sum_{z_k' \in S_{k,\bar\eta,m}(\alpha') \cap U_{z_k}} \mu_{k,\bar\eta,\alpha',z_k'}(\kappa).
\end{equation}
\end{lemma}

\begin{proof}
Decompose $\alpha$:
\begin{equation}
\label{eq:muContinuousalphaDecomp}
 \alpha = \,\Ad(u) \circ
\diag(
\ev_{v_1} \circ \sigma_k,\dots,
\ev_{v_p} \circ \sigma_k,
\beta \circ \lambda_k), 
\end{equation}
where $v_1,\dots,v_p \in CX_k|_{(0,1]}$, and $\beta \in \mathrm{Hom}(A_{k-1},M_{m'})$ for $m'=m-pm_k$.

Let us prove the statement inductively on $k \geq k'$.
For $k=k'=0$, it is quite easy---and uses the same idea as the case $k=k'>0$ which is done in detail next.

For $k=k'>0$, assume that the $v_i$ are ordered so that $v_1=\cdots=v_{p'}=z_k$ and $v_i \neq z_k$ for $i=p'+1,\dots,p$ (thus $p'\geq 1$).

Let $U_{z_k}$ be a neighbourhood of $z_k$ in $CX_k|_{(\eta_k,1]}$ such that $v_i \not\in \overline V$ for $i=p'+1,\dots,p$ (we identify this with a subset of $Z_{k,\bar\eta,k}^{**}$).
Then define $U_\alpha$ to consist of all $\alpha' \in \mathrm{Hom}(A_k,M_m)$ that are of the form
\begin{equation}
\label{eq:muContinuousalpha'Decomp}
\alpha' = \Ad(u') \circ \diag(\ev_{v'_1}\circ \sigma_k,\dots,\ev_{v'_{p'}} \circ \sigma_k,\beta) \end{equation}
where $u'$ is a unitary near to $u$, $v'_1,\dots,v'_{p'} \in U_{z_k}$, and $\beta \in \mathrm{Hom}(A_k,M_{m-p'm_k})$ does not contain $\ev_v \circ \sigma_k$ as a subrepresentation  for any $v \in \overline{V}$.
Then $U_\alpha$ is a neighbourhood of $\alpha$.

For $\alpha' \in U_\alpha$, let us write $\alpha'$ as in \eqref{eq:muContinuousalpha'Decomp}, so
\begin{equation} S_{k,\bar\eta,m}(\alpha') \cap U_{z_k} = \{v_1',\dots,v_{p'}'\}. \end{equation}
Thus,
\begin{align}
\notag
& \hspace*{-10mm} \sum_{z_k' \in S_{k,\bar\eta,m}(\alpha') \cap U_{z_k}} \mu_{k,\bar\eta,\alpha',z_k'}(\kappa) \\
\notag
&= \sum_{i=1}^{p'} \mu_{k,\bar\eta,\alpha',v_i'}(\kappa) \\
\notag
&= \Ad(u')(\diag(\kappa,\dots,\kappa,0_{m-p'm_k}) \\
\notag
&\approx_\e \Ad(u)(\diag(\kappa,\dots,\kappa,0_{m-p'm_k}) \\
&= \mu_{k,\bar\eta,z_k}(\kappa),
\end{align}
thus establishing \eqref{eq:muContinuous}.

For $k>k'$, let us consider two cases.

(i) $z_k=((x_k,z_{k-1}),t) \in Y_{k,\bar\eta} \times (0,\eta_k)$:
In this case, $z_{k-1} \in Z_{k-1,\bar\eta,k'}^{**}$.
Referring to the decomposition in \eqref{eq:muContinuousalphaDecomp}, we may assume that the $v_i$ are ordered so that $v_1=\cdots=v_{p'}=(x_k,t)$ and $v_i \neq (x_k,t)$ for $i=p'+1,\dots,p$.

By induction and by continuity of $\hat\theta_k$, there exist neighbourhoods $U_{z_{k-1}}$, $U_{x_k}$ of $z_{k-1}$, $x_k$ respectively (the neighbourhood $U_{x_k}$ arising from a neighbourhood of $\hat\theta_k(x_k)$), such that, for every $x_k' \in U_{x_k}$ and every contraction $\kappa \in M_{m_{k'}}$,
\begin{equation}
\label{eq:muContinuousCase2Ind}
 \mu_{k-1,\bar\eta,\hat\theta_k(x_k),z_{k-1}}(\kappa) \approx_{\e/p} \sum_{z_{k-1}' \in S_{k,\bar\eta,m_k}(\hat\theta_k(x_k')) \cap U_{z_{k-1}}} \mu_{k-1,\bar\eta,\hat\theta_k(x_k'),z_{k-1}'}(\kappa).
\end{equation}
Choose a neighbourhood $V \subset U_{x_k} \times (0,\eta_k)$ of $(x_k,t)$ such that $v_i \not\in \overline{V}$ for $i=p'+1,\dots,p$.
Define
\begin{equation} U_{z_k} := \{((x_k',z_{k-1}'),t') \in Y_{k,\bar\eta} \times (0,\eta_k) \mid (x_k',t') \in V, z_{k-1}' \in U_{z_{k-1}}\}. \end{equation}
Then define $U_\alpha$ to consist of all $\alpha' \in \mathrm{Hom}(A_k,M_m)$ with a decomposition
\begin{equation}
\label{eq:muContinuousalpha'Decomp2}
\alpha' = \Ad(u') \circ \diag(\ev_{v'_1}\circ \sigma_k,\dots,v'_{p'} \circ \sigma_k,\beta) \end{equation}
where $u'$ is a unitary near to $u$, $v'_1,\dots,v'_{p'} \in V$, and $\beta \in \mathrm{Hom}(A,M_{m-p'm_k})$ does not contain $\ev_v \circ \sigma_k$ as a subrepresentation for any $v \in \overline{V}$.
Our set-up ensures that $U_\alpha$ is a neighbourhood of $\alpha$.

To see \eqref{eq:muContinuous}, let $\alpha' \in U_\alpha$ be as in \eqref{eq:muContinuousalpha'Decomp2} and let $v_i'=(x^{(i)}_k,t_i)$ for each $i$, so that
\begin{align}
S_{k,\bar\eta,m}(\alpha') \cap U_{z_k} = \bigcup_{i=1}^n \{(x_k^{(i)},z_{k-1}'),t_i) \mid z_{k-1}' \in R_i\},
\end{align}
where $R_i:=\{z_{k-1}' \in U_{z_{k-1}} \mid (x_k^{(i)},z_{k-1}') \in Y_{k,\bar\eta}\}$.
For $z_{k-1}' \in R_i$, by \eqref{eq:muDef2} we have
\begin{align}
\notag
&\mu_{k,\bar\eta,\alpha',((x_k^{(i)},z_{k-1}'),t_i)}(\kappa) =  \\
&\qquad \Ad(u')\circ\diag(0,\dots,0,\mu_{k-1,\bar\eta,\hat\theta_k(x_k^{(i)}),z_{k-1}'}(\kappa),0,\dots,0,0_{m-p'm_k}),
\end{align}
where the non-zero entry is in the $i^\text{th}$ position.
By \eqref{eq:muContinuousCase2Ind}, it follows that
\begin{align}
\notag
& \sum_{z_{k-1}'\in R_i} \mu_{k,\bar\eta,\alpha',((x_k^{(i)},z_{k-1}'),t_i)}(\kappa) \\
&\quad \approx_{\e}
\Ad(u')(\diag(0,\dots,0,\mu_{k-1,\bar\eta,\theta_k(x_k^{(i)}),z_{k-1}}(\kappa),0,\dots,0,0_{m-p'm_k})).
\end{align}
Thus,
\begin{eqnarray}
\notag
&& \hspace*{-10mm} \sum_{z_k' \in S_{k,\bar\eta,m}(\alpha') \cap U_{z_k}} \mu_{k,\bar\eta,\alpha',z_k'}(\kappa) \\
\notag
&=& \sum_{i=1}^{p'} \sum_{z_{k-1}' \in R_i} \mu_{k,\bar\eta,\alpha',((x_k^{(i)},z_{k-1}'),t_i)}(\kappa) \\
\notag
&\approx_\e& \Ad(u')(\diag(1_{p'} \otimes \mu_{k-1,\bar\eta,\theta_k(x_k^{(i)}),z_{k-1}}(\kappa),0_{m-p'm_k})) \\
&\stackrel{\eqref{eq:muDef2}}=& \mu_{k,\bar\eta,\alpha}(\kappa)
\end{eqnarray}
(the approximation is within $\e$ because the errors, each of norm at most $\e$, are orthogonal), which establishes \eqref{eq:muContinuous}.

(ii) $z_k=(z_{k-1},0) \in Y_{k,\bar\eta} \times (0,\eta_k)$:
Then $z_{k-1} \in Z_{k-1,\bar\eta,k'}^{**} \cap S_{k-1,\bar\eta,m'}(\beta)$ (referring to the decomposition \eqref{eq:muContinuousalphaDecomp}).
By induction, there exist neighbourhoods $U_{z_{k-1}}$ of $z_{k-1}$ in $Z_{k-1,\bar\eta,k'}^{**}$ and $U_{\beta}$ of $\beta$ in $\Hom(A_{k-1},M_{m'})$ such that for every $\beta' \in U_{\beta}$ and every contraction $\kappa \in M_{m_{k'}}$,
\begin{align}
\label{eq:muContinuousCase3Ind}
\notag
&\hspace*{-10mm}\mu_{k-1,\bar\eta,\beta',z_{k-1}}(\kappa) \\
&\quad \approx_{\e/3} \sum_{z_{k-1}' \in S_{k,\bar\eta,m'}(\beta') \cap U_{z_{k-1}}} \mu_{k-1,\bar\eta,\beta',z_{k-1}'}(\kappa).
\end{align}

Here we need to be particularly careful about our choice of neighbourhoods, to take into account the possibility that $\beta \circ \lambda_k$ can be decomposed differently, involving point evaluations on $CX_k|_{\{0\}}$.
For each $q\in \mathbb N$, consider the set
\begin{align}
\notag
 V_q &:= \{(v,x_k^{(1)},\dots,x_k^{(q)},\gamma) \in U(M_{m'}) \times X_k^q \times \mathrm{Hom}(A_{k-1},M_{m'-qm_k}) \mid \\
&\qquad \Ad(v) \circ \diag(\ev_{(x_k^{(1)},0)} \circ \sigma_k,\dots,\ev_{(x_k^{(q)},0)} \circ \sigma_k,\gamma \circ \lambda_k) \in U_\beta \circ \lambda_k\}. 
\end{align}
Since
\begin{align}
\notag
(v,x_k^{(1)},\dots,x_k^{(q)},\gamma) \mapsto\, &\Ad(v) \circ \diag(\ev_{(x_k^{(1)},0)} \circ \sigma_k,\dots,\\
&\qquad \ev_{(x_k^{(q)},0)} \circ \sigma_k,\gamma \circ \lambda_k)
\end{align}
is a continuous map into $\mathrm{Hom}(A_{k-1},M_{m'}) \circ \lambda_k$, the set $V_q$ is open.
Let $\eta' < \eta_k$ be such that $v_1,\dots,v_p \in CX_k|_{(\eta',1]}$.
Define
\begin{align}
\notag
 W := \bigcup_{q\in\mathbb N} &\{\Ad(v) \circ \diag(\ev_{(x_k^{(1)},t_1)} \circ \sigma_k,\dots,\ev_{(x_k^{(q)},t_q)} \circ \sigma_k,\gamma \circ \lambda_k) \mid \\
&(v,x_k^{(1)},\dots,x_k^{(q)},\gamma) \in V_q\text{ and }t_1,\dots,t_q \in [0,\eta')\},
\end{align}
so that $W$ is a neighbourhood of $\beta \circ \lambda_k$ in $\mathrm{Hom}(A_k,M_{m'})$.

Set
\begin{align}
\notag
U_\alpha &:= \{\Ad(v) \circ \diag(\ev_{w_1} \circ \sigma_k,\dots,\ev_{w_p} \circ \sigma_k,\check\beta) \mid \\
&\qquad w_1,\dots,w_p \in CX_k|_{(\eta',1]}, \check\beta \in W, v \in U(M_m), v \approx_{\e/3} u\},
\end{align}
which is a neighbourhood of $\alpha$, and set
\begin{align}
\notag
 U_{z_k} :=\, &U_{z_{k-1}} \times \{0\} \cup \\
&\quad \{((x_k,z_{k-1}'),t) \in Y_{k,\bar\eta} \times (0,\eta') \mid z_{k-1}' \in U_{z_{k-1}}\},
\end{align}
which is a neighbourhood of $z_k$.

Let $\alpha' \in U_\alpha$ so that
\begin{equation} \alpha' = \Ad(v) \circ \diag(\ev_{w_1} \circ \sigma_k,\dots,\ev_{w_p} \circ \sigma_k,\check\beta), \end{equation}
for some 
$w_1,\dots,w_p \in CX_k|_{(\eta',1]}$, some $\check\beta \in W$, and some $v \in U(M_m)$ with 
\begin{equation}
\label{eq:muContinuousvApproxu}
v \approx_{\e/3} u
\end{equation}
Thus,
\begin{equation}
\check\beta = \Ad(w) \circ \diag(\ev_{(x_k^{(1)},t_1)} \circ \sigma_k,\dots,\ev_{(x_k^{(q)},t_q)} \circ \sigma_k, \gamma \circ \lambda_k)
\end{equation}
where $q \in \mathbb N$, $(x_k^{(i)},t_i) \in CX_k|_{[0,\eta')}$, $\gamma \in \Hom(A_{k-1},M_{m'-qm_k})$, and these satisfy
\begin{equation}
\label{eq:muContinuousCheckBeta'Def}
\check\beta' := \Ad(w) \circ \diag(\ev_{(x_k^{(1)},0)} \circ \sigma_k,\dots,\ev_{(x_k^{(q)},0)} \circ \sigma_k, \gamma \circ \lambda_k) \in U_\beta \circ \lambda_k.
\end{equation}
Set 
\begin{equation}
\label{eq:muContinuousw'Def}
w' := \diag(1_{pm_k}, w) \in U(M_m),
\end{equation}
so that $\check\beta$ corresponds to the subrepresentation
\begin{equation}
\Ad(vw') \circ \diag(0_{pm_k}, \ev_{(x_k^{(1)},0)} \circ \sigma_k,\dots,\ev_{(x_k^{(q)},0)} \circ \sigma_k, \gamma \circ \lambda_k)
\end{equation}
of $\alpha$.
Note that
\begin{equation}
\check\beta' = \Ad(w) \circ \diag(\hat\theta_k(x_k^{(1)}),\dots,\hat\theta_k(x_k^{(q)}),\gamma) \circ \lambda_k,
\end{equation}
and since $\lambda_k$ is surjective, it follows from \eqref{eq:muContinuousCheckBeta'Def} that
\begin{equation}
\beta':= \Ad(w) \circ \diag(\hat\theta_k(x_k^{(1)}),\dots,\hat\theta_k(x_k^{(q)}),\gamma) \in U_\beta.
\end{equation}
We compute
\begin{align}
\notag
S_{k,\bar\eta,m}(\alpha') \cap U_{z_k} &= \bigcup_{i=1}^q R_i \cup R_\gamma
\end{align}
where for $i=1,\dots,q$,
\begin{equation}
R_i := \begin{cases}
\{((x_k^{(i)},z_{k-1}'),t_i) \mid z_{k-1}' \in S_{k-1,\bar\eta,m_k}(\hat\theta_k(x_k^{(i)})) \cap U_{z_{k-1}}\}, \quad &\text{if } t_i>0; \\
\{(z_{k-1}',0) \mid z_{k-1}' \in S_{k-1,\bar\eta,m_k}(\hat\theta_k(x_k^{(i)})) \cap U_{z_{k-1}}\}, \quad &\text{if }t_i=0,
\end{cases}
\end{equation}
and
\begin{equation}
R_\gamma:= \{(z_{k-1}',0) \mid z_{k-1}' \in S_{k-1,\bar\eta,m'-qm_k}(\gamma) \cap U_{z_{k-1}}\}.
\end{equation}
For $z_k' \in R_i$, let $z_k'=(x_k,z_{k-1}',t_i)$ if $t_i>0$ or $z_k'=(z_{k-1}',0)$ if $t_i=0$; we see that $z_{k-1}' \in S_{k-1,\bar\eta,m'}(\beta') \cap U_{z_{k-1}}$.
Then by \eqref{eq:muDef2} (in the former case) or \eqref{eq:muDef3} (in the latter), and using $w'$ as defined in \eqref{eq:muContinuousw'Def}, we have
\begin{eqnarray}
\notag
&&\mu_{k,\bar\eta,\alpha',z_k'} \\
&=& \Ad(vw') \circ \diag(0_{(p+i-1)m_k},\mu_{k-1,\bar\eta,\hat\theta_k(x_k^{(i)}),z_{k-1}'},0_{m-(p-i)m_k})
\end{eqnarray}
For $z_k' = (z_{k-1}',0) \in R_\gamma$, by \eqref{eq:muDef3}, we have
\begin{equation}
\mu_{k,\bar\eta,\alpha,z_k'} = \Ad(vw') \circ \diag(0_{(p+q)m_k},\mu_{k-1,\bar\eta,\gamma,z_{k-1}'}) \\
\end{equation}
From these computations, we see that
\begin{eqnarray}
\notag
&&\sum_{z_k' \in S_{k,\bar\eta,m}(\alpha') \cap U_{z_k}} \mu_{k,\bar\eta,\alpha',z_k'}(\kappa) \\
\notag
&=& \sum_{z_{k-1}' \in S_{k-1,\bar\eta,m'}(\beta') \cap U_{z_{k-1}}} \Ad(v)(\diag(0_{pm_k},\mu_{k-1,\bar\eta,\beta',z_{k-1}'}(\kappa))) \\
\notag
&\stackrel{\eqref{eq:muContinuousCase3Ind}}{\approx_{\e/3}}& 
\Ad(v)(\diag(0_{pm_k},\mu_{k-1,\bar\eta,\beta,z_{k-1}}(\kappa))) \\
\notag
&\stackrel{\eqref{eq:muContinuousvApproxu}}{\approx_{2\e/3}}&
\Ad(u)(\diag(0_{pm_k},\mu_{k-1,\bar\eta,\beta,z_{k-1}}(\kappa))) \\
&\stackrel{\eqref{eq:muDef3}}=& \mu_{k,\bar\eta,\alpha,z_k}(\kappa).
\end{eqnarray}
\end{proof}

\subsection{Construction of $\xi_{k,\bar\eta,k'}$}

For $k' < k$ and $f \in C_0(Z^{**}_{k,\bar\eta,k'},M_{m_{k'}})$, define $\check\xi_{k,\bar\eta,k'}(f)\colon X_k \times [0,\eta_k) \to M_{m_k}$ by
\begin{equation}
\label{eq:checkXiDef}
\check\xi_{k,\bar\eta,k'}(f)(x_k,t) := \begin{cases}
\sum_{z_{k-1}} \mu_{k-1,\bar\eta,\hat\theta_k(x_k),z_{k-1}}(f((x_k,z_{k-1}),t)),\quad &\text{if }t>0; \\
\sum_{z_{k-1}} \mu_{k-1,\bar\eta,\hat\theta_k(x_k),z_{k-1}}(f(z_{k-1},0)),\quad &\text{if }t=0.
\end{cases}
\end{equation}
where the sum is taken over all 
\begin{equation}
 z_{k-1} \in S_{k-1,\bar\eta,m_k}(\hat\theta_k(x_k)) \cap Z^{**}_{k-1,\bar\eta,k'}.
\end{equation}
By Lemma \ref{lem:ZkPartition} (i), the domain of $\mu_{k-1,\bar\eta,\hat\theta_k(x_k),z_{k-1}}$ is indeed $M_{m_{k'}}$ for such $z_{k-1}$.

\begin{lemma}
\label{lem:checkXiDef}
For $k' < k$, $\check\xi_{k,\bar\eta,k'}$ defines a ${}^*$-homomorphism
\begin{equation}
C_0(Z^{**}_{k,\bar\eta,k'},M_{m_{k'}}) \to C_0(X_k \times [0,\eta_k),M_{m_k})
\end{equation}
\end{lemma}

\begin{proof}
That $\check\xi_{k,\bar\eta,k'}$ is a ${}^*$-homomorphism follows from the fact that in \eqref{eq:checkXiDef}, the right hand side is a sum of ${}^*$-homomorphisms (applied to $f$) with orthogonal ranges (by the last part of Lemma \ref{lem:PointwiseApprox}).

We need to show that $\check\xi_{k,\bar\eta,k'}(f)$ is continuous and vanishes at infinity, for every $f\in C_0(Z^{**}_{k,\bar\eta,k'},M_{m_{k'}})$.
For continuity, let $(x_k,t) \in X_k \times [0,\eta_k)$, $\e > 0$.
Let
\begin{equation}
\label{eq:xiContinuousSetup0}
S_{k-1,\bar\eta,m_k}(\hat\theta_k(x_k)) \cap Z^{**}_{k-1,\bar\eta,k'} = \{z_{k-1}^{(1)},\dots,z_{k-1}^{(p)}\}.
\end{equation}
By Lemma \ref{lem:muContinuous}, there exists a neighbourhood $V$ of $\hat\theta_k(x_k)$ in $\Hom(A_{k-1},M_{m_k})$ and, for each $i=1,\dots,p$, a neighbourhood  $U_i$ of $z_{k-1}^{(i)}$ in $Z_{k-1,\bar\eta,k'}^{**}$ such that for $\alpha \in V$ and each contraction $\kappa \in M_{m_{k'}}$,
\begin{equation}
\label{eq:xiContinuousSetup1}
\mu_{k-1,\bar\eta,\hat\theta_k(x_k),z_{k-1}^{(i)}}(\kappa) \approx_{\e/p} \sum_{z_{k-1}'} \mu_{k-1,\bar\eta,\alpha,z_{k-1}'}(\kappa),
\end{equation}
where the sum is over all $z_{k-1}' \in S_{k-1,\bar\eta,m_k}(\alpha) \cap U_i$.

By possibly shrinking the $U_i$, arrange that they are disjoint.
By Lemma \ref{lem:SContinuous}, if $\alpha$ is sufficiently close to $\hat\theta_k(x_k)$ then
\begin{equation}
\label{eq:xiContinuousSetup2}
S_{k-1,\bar\eta,m_k}(\alpha) \cap Z_{k-1,\bar\eta,k'}^{**} = \coprod_{i=1}^p S_{k-1,\bar\eta,m_k}(\alpha) \cap U_i.
\end{equation}
By possibly shrinking $V$, arrange that \eqref{eq:xiContinuousSetup2} holds for all $\alpha \in V$.

Let $W$ be a neighbourhood of $(x_k,t)$ in $X_k \times [0,\eta_k)$ such that, for $(x_k',t') \in W$, the following hold:

(i) $\hat\theta_k(x_k') \in V$;

(ii) If $t>0$ then $t'>0$ and for $((x_k',z_{k-1}'),t') \in Z_{k,\bar\eta,k'}^{**}$, there exists $i$ (necessarily unique) such that $z_{k-1}' \in U_i$ and
\begin{equation}
\label{eq:xiContinuousSetup3a}
f((x_k,z_{k-1}^i),t) \approx_{\e/p} f((x_k',z_{k-1}'),t').
\end{equation}

(iii) If $t=0$ and $t'>0$ then for $z_{k-1}' \in S_{k-1,\bar\eta,m_k}(\hat\theta_k(x_k'))$, there exists $i$ (necessarily unique) such that $z_{k-1}' \in U_i$ and
\begin{equation}
\label{eq:xiContinuousSetup3b}
f(z_{k-1}^i,0) \approx_{\e/p} f((x_k',z_{k-1}'),t').
\end{equation}

(iv) If $t=0$ and $t'=0$ then for $z_{k-1}' \in S_{k-1,\bar\eta,m_k}(\hat\theta_k(x_k'))$, there exists $i$ (necessarily unique) such that $z_{k-1}' \in U_i$ and
\begin{equation}
\label{eq:xiContinuousSetup3c}
f(z_{k-1}^i,0) \approx_{\e/p} f(z_{k-1}',0).
\end{equation}

Now, for $(x_k',t') \in W$, let us show that $\check\xi_{k,\bar\eta,k'}(f)(x_k,t) \approx_{3\e} \check\xi_{k,\bar\eta,k'}(f)(x_k',t')$.
Suppose first that $t,t'>0$.
Then for each $i=1,\dots,p$,
\begin{eqnarray}
\notag
&& \mu_{k-1,\bar\eta,\hat\theta_k(x_k),z_{k-1}^{(i)}}\big(f((x_k,z_{k-1}^{(i)}),t)\big) \\
\notag
&\stackrel{\eqref{eq:xiContinuousSetup1}}{\approx_{\e/p}}& \sum_{z_{k-1}' \in S_{k-1,\bar\eta,m_k}(\hat\theta_k(x_k')) \cap U_i} \\
\notag
&& \quad \mu_{k-1,\bar\eta,\hat\theta_k(x_k'),z_{k-1}'}\big(f((x_k,z_{k-1}^{(i)}),t)\big) \\
\notag
&\stackrel{\eqref{eq:xiContinuousSetup3a}}{\approx_{\e/p}}& \sum_{z_{k-1}' \in S_{k-1,\bar\eta,m_k}(\hat\theta_k(x_k')) \cap U_i} \\
&& \quad \mu_{k-1,\bar\eta,\hat\theta_k(x_k'),z_{k-1}'}\big(f((x_k',z_{k-1}'),t')\big).
\label{eq:xiContinuousAlmost1}
\end{eqnarray}
Thus,
\begin{eqnarray}
\notag
&& \check\xi_{k,\bar\eta,k'}(f)(x_k,t) \\
\notag
&\stackrel{\eqref{eq:checkXiDef},\eqref{eq:xiContinuousSetup0}}=& \sum_{i=1}^p \mu_{k-1,\bar\eta,\hat\theta_k(x_k),z_{k-1}^{(i)}}(f((x_k,z_{k-1}^{(i)}),t)) \\
\notag
&\stackrel{\eqref{eq:xiContinuousAlmost1}}{\approx_{2\e}}& \sum_{i=1}^p \textstyle{\sum_{z_{k-1}' \in S_{k-1,\bar\eta,m_k}(\hat\theta_k(x_k')) \cap U_i}} \\
\notag
&& \quad \mu_{k-1,\bar\eta,\hat\theta_k(x_k'),z_{k-1}'}\big(f((x_k',z_{k-1}'),t')\big) \\
\notag
&\stackrel{\eqref{eq:xiContinuousSetup2}}=&
\notag
\textstyle{\sum_{z_{k-1}' \in S_{k-1,\bar\eta,m_k}(\hat\theta_k(x_k')) \cap Z_{k-1,\bar\eta,k'}^{**}}} \\
\notag
&& \quad \mu_{k-1,\bar\eta,\hat\theta_k(x_k'),z_{k-1}'}\big(f((x_k',z_{k-1}'),t')\big) \\
&\stackrel{\eqref{eq:checkXiDef}}=& \check\xi_{k,\bar\eta,k'}(f)(x_k',t').
\end{eqnarray}

The argument is almost exactly the same in other cases (if $t=0$ and $t'>0$, or if $t=t'=0$), using \eqref{eq:xiContinuousSetup3b} or \eqref{eq:xiContinuousSetup3c} in place of \eqref{eq:xiContinuousSetup3a}.
This concludes the proof that $\check\xi_{k,\bar\eta,k'}(f)$ is continuous.

To see that $\check\xi_{k,\bar\eta,k'}(f)$ vanishes at infinity, note that
\begin{equation} \lim_{t\to \eta_k} \|f|_{Z^{**}_{k,\bar\eta,k'}|_{\{t\}}}\| = 0, \end{equation}
and so, by \eqref{eq:checkXiDef},
\begin{equation} \lim_{t\to \eta_k} \|\check\xi_{k,\bar\eta,k'}(f)|_{X_k \times \{t\}}\| = 0, \end{equation}
as required.
\end{proof}

Upon identifying $Z^{**}_{k,\bar\eta,k}$ with $CX_k|_{(\eta_k,1]}$ for $k >0$, we also define 
\begin{equation}
\label{eq:checkXiTrivDef}
\check\xi_{k,\bar\eta,k}\colon C_0(Z^{**}_{k,\bar\eta,k},M_{m_k}) \to C(CX_k,M_{m_k})
\end{equation}
to be the inclusion.
Likewise, $\check\xi_{0,(),0}\colon C_0(Z^{**}_{0,(),0},M_{m_0}) \to \mathrmm C(CX_0,M_{m_0})$ is defined to be the identity map.

Now, for $1 \leq k' \leq k$, we define
\begin{align}
\notag
 \xi_{k,\bar\eta,k'}\colon &C_0(Z^{**}_{k,\bar\eta,k'},M_{m_{k'}}) \to A_{k'-1} \oplus \\
&\qquad \mathrmm C(CX_{k'},M_{m_{k'}}) \oplus \cdots \oplus \mathrmm C(CX_k,M_{m_k})
\end{align}
by
\begin{align}
\notag
\xi_{k,\bar\eta,k'}(f) := &(0_{A_{k'-1}},\check\xi_{k',\bar\eta,k'}(f|_{Z^{**}_{k',\bar\eta,k'}}),\\
\label{eq:xiDef}
&\qquad \check\xi_{k'+1,\bar\eta,k'}(f|_{Z^{**}_{k'+1,\bar\eta,k'}}),\dots,\check\xi_{k,\bar\eta,k'}(f)),
\end{align}
where we are viewing $Z^{**}_{k'',\bar\eta,k'}$ as a subset of $Z^{**}_{k,\bar\eta,k'}$, for $k''=k',\dots,k-1$, by identifying
\begin{equation} Z^{**}_{k-1,\bar\eta,k'} \quad\text{with}\quad Z^{**}_{k,\bar\eta,k'}|_{\{0\}} \end{equation}
and so on.

For $k'=0$, we define $\xi_{k,\bar\eta,k'}$ by more or less the same formula, except that its range is
$\mathrmm C(CX_0,M_{m_0}) \oplus \cdots \oplus \mathrmm C(CX_k,M_{m_k})$.

\begin{lemma}
\label{lem:xiWellDef}
The image of $\xi_{k,\bar\eta,k'}$ is contained in $A_k$.
\end{lemma}

\begin{proof}
The proof consists in showing, by induction on $k$, simultaneously both the statement of the lemma and that, for $m\in \mathbb N$, $\alpha \in \mathrm{Hom}(A_k,M_m)$, and $f \in C_0(Z^{**}_{k,\bar\eta,k'},M_{m_{k'}})$,
\begin{equation}
\label{eq:xiWellDefExtra}
 \alpha(\xi_{k,\bar\eta,k'}(f)) = \sum_{z_k} \mu_{k,\bar\eta,\alpha,z_k}(f(z_k)),
\end{equation}
where the sum is taken over all
\begin{equation} z_k \in S_{k,\bar\eta,m}(\alpha) \cap Z^{**}_{k,\bar\eta,k'}. \end{equation}

The case $k=k'$ is trivial.

For the inductive step, note that
\begin{equation} \xi_{k,\bar\eta,k'}(f) = \xi_{k-1,\bar\eta,k'}(f|_{Z^{**}_{k-1,\bar\eta,k'}}) \oplus \check\xi_{k,\bar\eta,k'}(f), \end{equation}
so that, to show that the image of $\xi_{k,\bar\eta,k'}$ is contained in $A_k$, we must show that
\begin{equation}
\label{eq:xiWellDefToShow}
\theta_k(\xi_{k-1,\bar\eta,k'}(f|_{Z^{**}_{k-1,\bar\eta,k'}})) = \check\xi_{k,\bar\eta,k'}(f)|_{X_k \times \{0\}}.
\end{equation}

For $x_k \in X_k$,
\begin{eqnarray}
\notag
\check\xi_{k,\bar\eta,k'}(f)(x_k,0) &\stackrel{\eqref{eq:checkXiDef}}=&
\sum_{z_{k-1}} \mu_{k-1,\bar\eta,\hat\theta_k(x_k),z_{k-1}}(f(z_{k-1},0)) \\
\notag
&\stackrel{\eqref{eq:xiWellDefExtra}}=& \hat\theta_k(x_k)(\xi_{k-1,\bar\eta,k'}(f|_{Z^{**}_{k-1,\bar\eta,k'}})) \\
&=& \theta_k(\xi_{k-1,\bar\eta,k'}(f|_{Z^{**}_{k-1,\bar\eta,k'}}))(x_k),
\end{eqnarray}
where we have used induction for the second line.
This establishes \eqref{eq:xiWellDefToShow}.

Now, to show \eqref{eq:xiWellDefExtra}, let $\alpha \in \mathrm{Hom}(A_k,M_m)$, and decompose it as
\begin{equation}
 \alpha = \Ad(u) \circ 
\diag( \ev_{v_1} \circ \sigma_k,\dots, \ev_{v_p} \circ \sigma_k, \beta \circ \lambda_k),
\end{equation}
where $v_1,\dots,v_p \in CX_k|_{(0,1]}$, and $\beta \in \mathrm{Hom}(A_{k-1},M_{m'})$ for some $m'$.
Assume that $v_1,\dots,v_p$ are arranged so that $v_1,\dots,v_{p'} \in CX_k|_{[\eta_k,1]}$ and $v_{p'+1},\dots,v_p \in CX_k|_{(0,\eta_k)}$.

Fix $f \in C_0(Z^{**}_{k,\bar\eta,k'},M_{m_{k'}})$.
For $i=1,\dots,p'$,
\begin{equation} \ev_{v_i} \circ \sigma_k \circ \xi_{k,\bar\eta,k'}(f) = 0. \end{equation}

For $i=p'+1,\dots,p$, let $v_i=(x_k^{(i)},t_i) \in X_k \times (0,\eta_k)$.
Then
\begin{eqnarray}
\notag
&&\hspace*{-10mm} \ev_{v_i} \circ \sigma_k \circ \xi_{k,\bar\eta,k'}(f) \\
\notag
&\stackrel{\eqref{eq:xiDef}}=& \check\xi_{k,\bar\eta,k'}(f)(x_k^{(i)},t_i) \\
\notag
&\stackrel{\eqref{eq:checkXiDef}}=& \textstyle{\sum_{z_{k-1} \in S_{k-1,\bar\eta,m_k}(\hat\theta_k(x^{(i)}_k)) \cap Z^{**}_{k,\bar\eta,k'}}} \\
&&\quad \mu_{k-1,\bar\eta,\hat\theta_k(x_k),z_{k-1}}(f((x_k^{(i)},z_{k-1}),t)),
\end{eqnarray}
so that, by \eqref{eq:muDef2}, 
\begin{align}
\notag
&\diag(0_{m_k},\dots,0_{m_k},\ev_{v_i} \circ \sigma_k \circ \xi_{k,\bar\eta,k'}(f),0_{m_k},\dots,0_{m_k},0_{m'}) \\
&\quad = \textstyle{\sum_{z_{k-1} \in S_{k-1,\bar\eta,m_k}(\hat\theta_k(x^{(i)}_k)) \cap Z^{**}_{k-1,\bar\eta,k'}}}
\mu_{k,\bar\eta,\alpha,((x_k^{(i)},z_{k-1}),t_i)}(f(x_k,z_{k-1}),t).
\end{align}
Summing over $i$ gives
\begin{equation}
\label{eq:xiWellDefExtra1-e'}
(1-e')\alpha(\xi_{k,\bar\eta,k'}(f)) = \sum_{z_k} \mu_{k,\bar\eta,\alpha,z_k}(f(z_k)),
\end{equation}
where $e' := \Ad(u)(0_{m_k},\dots,0_{m_k},1_{m'})$ and the sum is taken over all
\begin{equation} z_k \in S_{k,\bar\eta,m}(\alpha) \cap Z^{**}_{k,\bar\eta,k'}|_{(0,1]}. \end{equation}

By induction,
\begin{equation} \beta(\xi_{k-1,\bar\eta,k'}(f|_{Z^{**}_{k-1,\bar\eta,k'}})) = \sum_{z_{k-1}} \mu_{k-1,\bar\eta,\beta,z_{k-1}}(f(z_{k-1},0)), \end{equation}
where the sum is taken over all
\begin{equation} z_{k-1} \in S_{k-1,\bar\eta,m'}(\beta) \cap Z^{**}_{k-1,\bar\eta,k'}. \end{equation}
By \eqref{eq:muDef3}, 
it follows that
\begin{align}
\label{eq:xiWellDefExtrae'}
\notag
e'\alpha(\xi_{k,\bar\eta,k'}(f)) &= (0_{m_k},\dots,0_{m_k},\beta \circ \lambda_k \circ \xi_{k,\bar\eta,k'}(f)) \\
&= \sum_{z_k} \mu_{k,\bar\eta,\alpha,(z_k)}(f(z_k)),
\end{align}
where the sum is taken over all
\begin{equation} z_k \in S_{k,\bar\eta,k'}(\alpha) \cap Z^{**}_{k,\bar\eta,k'}|_{\{0\}}. \end{equation}
Combining \eqref{eq:xiWellDefExtra1-e'} and \eqref{eq:xiWellDefExtrae'} yields \eqref{eq:xiWellDefExtra}.
\end{proof}

\subsection{Maps related to changing parameters $\bar\eta$}

Our construction will involve using different $Z_{k,\bar\eta}$ given by varying the parameters $\bar\eta$.
We will need to relate these spaces, and for this we define maps between them.

Let $\eta_1,\dots,\eta_k,\eta_1',\dots,\eta_k' \in (0,1)$ with $\eta_i \geq \eta_i'$ for $i=1,\dots,k$.
Let $\bar\eta$ and $\bar\eta'$ denote $(\eta_1,\dots,\eta_k)$ and $(\eta_1',\dots,\eta_k')$ respectively (or $(\eta_1,\dots,\eta_{k-1})$, $(\eta_1',\dots,\eta_{k-1}')$, as appropriate).
We shall now construct maps
\begin{equation} \rho_{k,\bar\eta,\bar\eta'}\colon Z_{k,\bar\eta} \to Z_{k,\bar\eta'}. \end{equation}

For $k=0$, set
\begin{equation} \rho_{0,(),()}:= \id_{X^{(0)}}. \end{equation}

Having defined $\rho_{k-1,\bar\eta,\bar\eta'}$, let us define $\rho_{k,\bar\eta,\bar\eta'}$.
We do this in cases:

(i) For $z_k \in Z_{k,\bar\eta}|_{[\eta_k,1)} = CX_k|_{[\eta_k,1]} = Z_{k,\bar\eta'}|_{[\eta_k,1]}$, simply set
\begin{equation} \rho_{k,\bar\eta,\bar\eta'}(z_k) := (z_k). \end{equation}

(ii) For $((x_k,z_{k-1}),t) \in Y_{k,\bar\eta} \times [\eta_k',\eta_k)$, set
\begin{equation} \rho_{k,\bar\eta,\bar\eta'}((x_k,z_{k-1}),t) := (x_k,t). \end{equation}

(ii)$'$ For $((x_k,z_{k-1}),t) \in Y_{k,\bar\eta} \times (0,\eta_k')$, set
\begin{equation}
\label{eq:rhoDefii'}
\rho_{k,\bar\eta,\bar\eta'}((x_k,z_{k-1}),t) := ((x_k,\rho_{k-1,\bar\eta,\bar\eta'}(z_{k-1})),t). \end{equation}
By \eqref{eq:rhoCommutes} (with $k-1$ in place of $k$, i.e., using induction on $k$) we see that
\begin{equation} \rho_{k,\bar\eta,\bar\eta'}(x_k,\rho_{k-1,\bar\eta,\bar\eta'}(z_{k-1})) \in Y_{k,\bar\eta'}, \end{equation}
so that the right-hand side of \eqref{eq:rhoDefii'} is in $Z_{k,\bar\eta'}$.

(iii) For $(z_{k-1},0) \in Z_{k-1,\bar\eta} \times \{0\}$, set
\begin{equation}
\label{eq:rhoDefiv}
 \rho_{k,\bar\eta,\bar\eta'}(z_{k-1},0) := (\rho_{k-1,\bar\eta,\bar\eta'}(z_{k-1}),0). \end{equation}

Here are some easy facts about the maps $\rho_{k,\bar\eta,\bar\eta'}$.

\begin{lemma}
\label{lem:rhoFacts}
Let $k,\bar\eta,\bar\eta'$ be given as above.

(i) $\rho_{k,\bar\eta,\bar\eta'}\colon Z_{k,\bar\eta} \to Z_{k,\bar\eta'}$ is continuous.

(ii) If $\eta_1'',\dots,\eta_k'' \in (0,1)$ are such that $\eta_i'' \leq \eta_i'$ for each $i$ then
\begin{equation}
\label{eq:rhoFacts2}
 \rho_{k,\bar\eta,\bar\eta''} = \rho_{k,\bar\eta',\bar\eta''} \circ \rho_{k,\bar\eta,\bar\eta'}. \end{equation}

(iii) For each $m$ and $\alpha \in \mathrm{Hom}(A_k,M_m)$,
\begin{align}
\notag
\rho_{k,\bar\eta,\bar\eta'}(S_{k,\bar\eta,m}(\alpha)) &= S_{k,\bar\eta',m}(\alpha) \quad \text{and} \\
\label{eq:rhoCommutes}
S_{k,\bar\eta,m}(\alpha) &= \rho_{k,\bar\eta,\bar\eta'}^{-1}(S_{k,\bar\eta',m}(\alpha)).
\end{align}
\end{lemma}

\begin{proof}
(i), (ii): These are easy inductive arguments.

(iii)
By induction on $k$; the case $k=0$ is easy.

For the inductive step, let $\alpha \in \mathrm{Hom}(A_k,M_m)$.
By induction, Cases (ii)$'$ and (iii) of the definition of $\rho_{k,\bar\eta,\bar\eta'}$, and Cases (ii) and (iii) of the definition of $S_{k,\bar\eta,m}(\alpha)$, we see that \eqref{eq:rhoCommutes} holds when both sides are restricted to $[0,\eta_k')$.
It is also clear, from Case (i) of the definition of $\rho_{k,\bar\eta,\bar\eta'}$ that \eqref{eq:rhoCommutes} holds when both sides are restricted to $[\eta_k,1]$.

Finally, Case (ii) of the definition of $\rho_{k,\bar\eta,\bar\eta'}$, together with Cases (i)$'$ and (ii) of the definition of $S_{k,\bar\eta,m}(\alpha)$ shows that \eqref{eq:rhoCommutes} holds when both sides are restricted to $[\eta_k',\eta_k)$.
\end{proof}

\begin{lemma}
\label{lem:omegaMuCompat}
Let $\bar\eta,\bar\eta'$ be as above.
Let $\alpha \in \mathrm{Hom}(A_k,M_m)$ and $z_k' \in S_{k,\bar\eta',m}(\alpha)$.
Then
\begin{equation}
\mu_{k,\bar\eta',\alpha,z_k'}(1) =
\sum_{z_k \in \rho^{-1}(\{z_k'\})} \mu_{k,\bar\eta,\alpha,z_k}(1).
\end{equation}
\end{lemma}

\begin{proof}
We prove this by induction on $k$.
In the base case, $k=0$, it is trivial since $\bar\eta=()=\bar\eta'$.

For $k>0$, consider cases.

(i) If $z_k' \in Z_{k,\bar\eta'}|_{[\eta_k,1]}=Z_{k,\bar\eta}|_{[\eta_k,1]}$ (see Case (i) of the definition of $\rho_{k,\bar\eta,\bar\eta'}$) then $\rho_{k,\bar\eta,\bar\eta'}^{-1}(\{z_k'\}) = \{z_k\}$ and $\mu_{k,\bar\eta',\alpha,z_k'}=\mu_{k,\bar\eta,\alpha,z_k}$, so the result holds in this case.

(ii) If $z_k'=(x_k,t) \in Z_{k,\bar\eta'}|_{[\eta_k',\eta_k)}$, then by Case (ii) of the definition of $\rho_{k,\bar\eta,\bar\eta'}$,
\begin{equation}
\label{eq:muRhoSumSetup1}
\rho_{k,\bar\eta,\bar\eta'}^{-1}(\{z_k'\}) = \{((x_k,z_{k-1}),t) \mid z_{k-1} \in S_{k,\bar\eta,m_k}(\hat\theta_k(x_k))\}.
\end{equation}
Let $\alpha'\colon A_k \to M_m$ be the maximal subrepresentation of $\alpha$ that factors through $\ev_{(x_k,t)} \circ \sigma_k$.
Then by \eqref{eq:muDef1},
\begin{equation}
\alpha' = \mu_{k,\bar\eta',\alpha,z_k'} \circ \ev_{(x_k,t)} \circ \sigma_k,
\end{equation}
and so, by \eqref{eq:muDef2}, for $z_{k-1} \in S_{k,\bar\eta,m_k}(\hat\theta_k(x_k))$,
\begin{equation}
\label{eq:muRhoSumSetup3}
\mu_{k,\bar\eta,\alpha,((x_k,z_{k-1},t)} = \mu_{k,\bar\eta',\alpha,z_k'} \circ \mu_{k-1,\bar\eta,\hat\theta_k(x_k),z_{k-1}}.
\end{equation}
By the final statement of Lemma \ref{lem:PointwiseApprox},
\begin{equation}
1_{m_k} = \sum_{z_{k-1} \in S_{k,\bar\eta,m_k}(\hat\theta_k(x_k))} \mu_{k-1,\bar\eta,\hat\theta_k(x_k),z_{k-1}}(1),
\end{equation}
and thus,
\begin{eqnarray}
\notag
&& \mu_{k,\bar\eta',\alpha,z_k}(1_{m_k}) \\
\notag
&=& \sum_{z_{k-1} \in S_{k,\bar\eta,m_k}(\hat\theta_k(x_k))} \\
\notag
&&\qquad \mu_{k,\bar\eta',\alpha,z_k} \circ \mu_{k-1,\bar\eta,\hat\theta_k(x_k),z_{k-1}}(1) \\
\notag
&\stackrel{\eqref{eq:muRhoSumSetup3}}=& 
\sum_{z_{k-1}} \mu_{k,\bar\eta,\alpha,((x_k,z_{k-1}),t}(1) \\
&\stackrel{\eqref{eq:muRhoSumSetup1}}=&
\sum_{z_k \in \rho_{k,\bar\eta,\bar\eta'}^{-1}(\{z_k'\})} \mu_{k,\bar\eta,\alpha,z_k,t}(1),
\end{eqnarray}
as required.

(iii)
If $z_k'=((x_k,z_{k-1}'),t) \in Y_{k,\bar\eta'} \times (0,\eta_k')$ then, defining $\mu'$ as in \eqref{eq:muDefmu'Def}, we have by \eqref{eq:muDef2}
\begin{equation} \mu_{k,\bar\eta',z_k'} = \mu' \circ \mu_{k-1,\bar\eta',\hat\theta_k(x_k),z_{k-1}'}. \end{equation}
Every $z_k \in \rho_{k,\bar\eta,\bar\eta'}^{-1}(z_k')$ is equal to $((x_k,z_{k-1}),t)$ for some $z_{k-1} \in \rho_{k-1,\bar\eta,\bar\eta'}^{-1}(z_{k-1}')$, and for such $z_k$ we likewise have
\begin{equation} \mu_{k,\bar\eta,z_k} := \mu' \circ \mu_{k-1,\bar\eta,\hat\theta_k(x_k),z_{k-1}}. \end{equation}
Therefore, in this case, the conclusion follows from the inductive hypothesis.

(iv) If $z_k'=(z_{k-1}',0) \in Z_{k-1,\bar\eta'} \times \{0\}$ then defining $\alpha'$ and $\mu'$ as in Case (iii) of the definition of $\mu_{k,\bar\eta,\alpha,z_k}$, we have by \eqref{eq:muDef3},
\begin{equation} \mu_{k,\bar\eta',z_k'} := \mu' \circ \mu_{k-1,\bar\eta',\alpha',z_{k-1}'}. \end{equation}
Every $z_k \in \rho_{k,\bar\eta,\bar\eta'}^{-1}(z_k')$ is of the form $(z_{k-1},0)$ for $z_{k-1} \in \rho_{k-1,\bar\eta,\bar\eta'}^{-1}(z_{k-1}')$, and for such $z_k$ we likewise have
\begin{equation} \mu_{k,\bar\eta,z_k} := \mu' \circ \mu_{k-1,\bar\eta,\alpha',z_{k-1}}. \end{equation}
So once again, in this case, the conclusion follows from the inductive hypothesis.
\end{proof}

\subsection{Open covers}
\label{sec:OpenCover}

Now fix $\eta \in (0,1/2)$ and \textbf{let us use $\bar\eta$ or $\overline\eta^k$ to denote $(\eta,\dots,\eta)$}.
Further, write 
\begin{equation} \overline{2\eta}^{k} := (2\eta,\dots,2\eta) \in (0,1)^k \end{equation}
and
\begin{equation} \overline\eta^{k'}\overline{2\eta}^{k''} := (\eta,\dots,\eta,2\eta,\dots,2\eta) \in (0,1)^{k'+k''}, \end{equation}
where $\eta$ appears $k'$ times and $2\eta$ appears $k''$ times.

We shall construct open covers $\mathcal U_{k}$ of $Z_{k,\bar\eta}$.
These will depend on a prescribed finite subset $\mathcal F$ of $A_{k_0,\bar\eta}$ and a tolerance $\dl>0$, although this dependence is suppressed in the notation.
For simplicity, all open covers will be hereditary, i.e., closed under taking open subsets.

First, for each $k$, choose a hereditary open cover $\mathcal V_k$ of $CX_k$ such that, for each $a \in \mathcal F$, the function $\sigma_k \circ \lambda_{k_0}^k(a)$ is approximately constant, to within $\dl$, on each set in $\mathcal V_k$, i.e., for $V \in \mathcal V_k$ and $v_k,v_k' \in V$,
\begin{equation}
\label{eq:VkApproxConst}
\sigma_k \circ \lambda_{k_0}^k(a)(v_k) \approx_\dl \sigma_k \circ \lambda_{k_0}^k(a)(v_k').
\end{equation}

Let us define $\mathcal U_k$ recursively over $k$.
For $k=0$, set $\mathcal U_0 := \mathcal V_0$.

For $k\geq 1$, set $\mathcal U^{(1)}_{k}$ to the set of all open sets $U \in \mathcal V_k$ which are subsets of $CX_k|_{(\eta,1]}$.
Identifying $CX_k|_{(\eta,1]}$ with $Z_{k,\bar\eta}|_{(\eta,1]}$, let us view $\mathcal U^{(1)}_{k}$ as a family of open subsets of $Z_{k,\bar\eta}$.

Set $\mathcal U^{(2)}_{k}$ equal to the collection of all open sets $U \subset Z_{k,\bar\eta}|_{[0,2\eta)}$ such that the following decomposition holds:
We may decompose
\begin{equation}
\label{eq:UkDecomp}
\rho_{k,(\eta,\dots,\eta,2\eta),\bar\eta}^{-1}(U)
=
V_{1} \amalg \cdots \amalg V_{p} \end{equation}
such that $V_1,\dots,V_p$ are open in $Z_{k,(\eta,\dots,\eta,2\eta)}$, and for each $i=1,\dots,p$,
\begin{equation}
\label{eq:UkDecompProperty}
V_{i}(0) \cup 
\{z_{k-1} \mid ((x_k,z_{k-1}),t) \in V_i, \text{some }x_k \in X_k, t \in (0,2\eta)\} \in \mathcal U_{k-1}.
\end{equation}

Now set
\begin{equation} \mathcal U_{k} := \mathcal U^{(1)}_{k} \cup \mathcal U^{(2)}_{k}. \end{equation}

\begin{lemma}
\label{lem:UkOpenCover}
For $k\geq 1$, $\mathcal U^{(1)}_k$ (respectively $\mathcal U^{(2)}_k$) is a hereditary open cover of $Z_{k,\bar\eta}|_{(\eta,1]}$ (respectively $Z_{k,\bar\eta}|_{[0,2\eta)}$).
Therefore, $\mathcal U_k$ is a hereditary open cover of $Z_{k,\bar\eta}$.
\end{lemma}

\begin{proof}
Using the fact that $\mathcal V_k$ is a hereditary open cover of $CX_k$, we see that $\mathcal U^{(1)}_k$ is an open cover of $Z_{k,\bar\eta}|_{(\eta,1]}$.
The rest of the lemma is proven by induction; at the outset, we know that $\mathcal U_0$ is an open cover of $Z_{0,()}$.

Now let us assume that $\mathcal U_{k-1}$ is an open cover of $Z_{k-1,\bar\eta}$, and prove that $\mathcal U^{(2)}_k$ is an open cover of $Z_{k,\bar\eta}|_{[0,2\eta)}$.
Fix $z_k \in Z_{k,\bar\eta}|_{[0,2\eta)}$, and let us argue that it is contained in a set in $\mathcal U^{(2)}_k$ in cases as follows:

(i)
If $z_k=(x_k,t) \in X_k \times [\eta,2\eta)$, write $S_{k,\bar\eta,m_k}(\hat \theta_k(x)) = \{z_{k-1}^{(1)},\dots,z_{k-1}^{(p)}\} \subset Z_{k-1,\bar\eta}$.
By the inductive hypothesis, choose $W_1,\dots,W_p \in \mathcal U_{k-1}$ that are disjoint, with $z_{k-1}^i \in W_i$ for each $i$.
By continuity of $S_{k,\bar\eta,m_k} \circ \hat \theta_k$, there is a neighbourhood $\hat U$ of $(x_k,t)$ in $CX_k$ such that
\begin{equation} \bigcup_{(x_k',t') \in \hat U} S_{k,\bar\eta,m_k}(\hat\theta_k(x_k',t')) \subseteq W_1\cup \cdots \cup W_p. \end{equation}
Set
\begin{align}
\notag
U:= &\{((x_k',z_{k-1}),t) \in Y_{k,\bar\eta} \times (0,\eta) \mid (x_k',t) \in \hat U\} \\
&\quad \cup (\hat U \cap X_k \times [\eta,2\eta)). \end{align}
Then, by construction, $U \in \mathcal U^{(2)}_k$ ($p_U = p$ and the sets $V_{U,1},\dots,V_{U,p_U}$ are preimages of $W_1,\dots,W_p$) and $(x_k,t) \in U$.

(ii)
If $z_k=((x_k,z_{k-1}),t) \in Y_{k,\bar\eta} \times (0,\eta)$, we do the exact same thing as in (i).

(iii)
For $(z_{k-1},0) \in Z_{k-1,\bar\eta}$, by the inductive hypothesis, there exists a neighbourhood $W$ of $z_{k-1}$ such that $W \in \mathcal U_{k-1}$.
Set
\begin{equation} U:= W \times \{0\} \cup \{((x_k,z'_{k-1}),t) \in Y_{k,\bar\eta} \times [0,\eta) \mid z'_{k-1} \in W\}. \end{equation}
By construction, $U \in \mathcal U^{(2)}_k$ ($p_U = 1$) and $(z_{k-1},0) \in U$.
\end{proof}

\begin{lemma}
\label{lem:OpenCoverFact}
Let $U \in \mathcal U_k$.
Then we have a decomposition 
\begin{equation} \rho_{k,\overline{2\eta}^k,\overline\eta^k}^{-1}(U) = W_{1} \amalg \cdots \amalg W_{q} \end{equation}
where, for each $i=1,\dots,q$, there is $k(i) \in \{0,\dots,k\}$ such that the set
\begin{equation} \hat W_{i} := \rho_{k,\overline{2\eta}^k,\overline \eta^{k(i)}\overline{2\eta}^{k-k(i)}}(W_{i}) \end{equation}
is open component of $\rho_{k, \overline\eta^{k(i)}\overline{2\eta}^{k-k(i)},\overline\eta^k}^{-1}(U)$, 
\begin{equation} \hat W_{i} \subseteq Z^{**}_{k,\overline\eta^{k(i)}\overline{2\eta}^{k-k(i)},k(i)}, \end{equation}
and for $a \in \mathcal F$ and $z_k,z_k' \in \hat W_i$
\begin{equation}
\label{eq:OpenCoverFactnuApprox}
\nu_{k,\overline\eta^{k(i)}\overline{2\eta}^{k-k(i)},z_k}(\lambda_{k_0}^k(a) \approx_\e 
\nu_{k,\overline\eta^{k(i)}\overline{2\eta}^{k-k(i)},z_k'}(\lambda_{k_0}^k(a))
\end{equation}
in $M_{m_{k(i)}}$.
\end{lemma}

\begin{proof}
Let us use induction on $k$.
For $k=0$, the conclusion is trivial (with $q=1$).
For $k\geq 1$, we consider cases as follows.

(i) If $U \in \mathcal U_k^{(1)}$ then by definition, $U \in Z^{**}_{k,\overline\eta^k,k}$.
We may therefore set 
\begin{equation} q=1,\quad W_1=U,\quad k(1)=k. \end{equation}

(ii) If $U \in \mathcal U_k^{(2)}$, then let $V_1,\dots,V_p$ be as in \eqref{eq:UkDecomp}.
Fix $i$ for the moment.
Set
\begin{equation} V_i' :=
V_{i}(0) \cup 
\{z_{k-1} \mid ((x_k,z_{k-1}),t) \in V_i, \text{some }x_k \in X_k, t \in (0,2\eta)\} \in \mathcal U_{k-1}.
\end{equation}
i.e., exactly the set appearing in \eqref{eq:UkDecompProperty}.
By induction, we have the decomposition
\begin{equation}
\rho_{k-1,\overline{2\eta}^{k-1},\overline\eta^{k-1}}^{-1}(V_i') = W_{i,1}' \amalg \cdots \amalg W'_{i,q_i},
\end{equation}
and we can find $k(i,j) \in \{0,\dots,k-1\}$ for $j=1,\dots,q_i$ such that the set
\begin{equation} \hat W'_{i,j} := \rho_{k-1, \overline{2\eta}^{k-1},\overline \eta^{k(i,j)}\overline{2\eta}^{k-1-k(i,j)}}(W'_{i,j}) \end{equation}
is an open component of $\rho_{k-1, \overline\eta^{k(i,j)}\overline{2\eta}^{k-1-k(i)},\overline\eta^{k-1}}^{-1}(V'_i)$, 
\begin{equation} \hat W'_{i,j} \subseteq Z^{**}_{k-1,\overline\eta_{k(i,j)}\overline{2\eta}^{k-1-k(i,j)},k(i,j)}, \end{equation}
and for $a \in \mathcal F$, $z_{k-1},z_{k-1}' \in \hat W'_{i,j}$,
\begin{equation}
\label{eq:OpenCoverFactnuApproxInd}
\nu_{k-1,\overline\eta^{k(i)}\overline{2\eta}^{k-1-k(i)},z_{k-1}}(\lambda_{k_0}^{k-1}(a)) \approx_\e 
\nu_{k-1,\overline\eta^{k(i)}\overline{2\eta}^{k-1-k(i)},z_{k-1}'}(\lambda_k(a)).
\end{equation}

Define $\hat W_{i,j} \subset Z_{k,\overline\eta^{k(i,j)}\overline{2\eta}^{k-k(i,j)}}$ to consist of all 
\begin{equation}
\label{eq:OpenCoverFactWijDef1}
(z_{k-1},0) \in \rho_{k,\overline\eta^{k(i,j)}\overline{2\eta}^{k-k(i,j)},(\eta,\dots,\eta,2\eta)}^{-1}(V_i)|_{\{0\}} \end{equation}
such that $z_{k-1} \in \hat W'_{i,j}$, together with all
\begin{equation} 
\label{eq:OpenCoverFactWijDef2}
((x_k,z_{k-1}),t) \in \rho_{k,\overline\eta^{k(i,j)}\overline{2\eta}^{k-k(i,j)},(\eta,\dots,\eta,2\eta)}^{-1}(V_i)|_{(0,\eta)} \end{equation}
such that $z_{k-1} \in \hat W'_{i,j}$.

Then $\hat W_{i,j}$ is the preimage of a component of
\begin{equation} \rho_{k-1,\overline\eta^{k(i,j)}\overline{2\eta}^{k-1-k(i,j)},\overline\eta^{k-1}}^{-1}(V_i') \end{equation}
under the continuous map $Z_{k,\overline\eta^{k(i,j)}\overline{2\eta}^{k-k(i,j)}}|_{[0,\eta)} \to Z_{k-1,\overline\eta^{k(i,j)}\overline{2\eta}^{k-k(i,j)-1}}$ defined by $(z_{k-1},0) \mapsto z_{k-1}$ and $((x_k,z_{k-1}),t) \mapsto z_{k-1}$.
The set $\hat W_{i,j}$ is therefore a component of
\begin{equation} \rho_{k,\overline\eta^{k(i,j)}\overline{2\eta}^{k-k(i,j)},(\eta,\dots,\eta,2\eta)}^{-1}(V_i). \end{equation}
Since $V_i$ is a component of $\rho_{k,\overline\eta^{k-1},(\eta,\dots,\eta,2\eta)}^{-1}(U)$, it follows that $\hat W_{i,j}$ is a component of
\begin{equation} \rho_{k,\overline\eta^{k(i,j)}\overline{2\eta}^{k-k(i,j)},(\eta,\dots,\eta,2\eta)}^{-1}(U). \end{equation}
Moreover, since $\hat W_{i,j} \subseteq Z_{k,(\eta,\dots,\eta,2\eta)}|_{[0,\eta)} = Z_{k,\overline{\eta}^k}|_{[0,\eta)}$, it follows that $\hat W_{i,j}$ is a component of
\begin{equation} \rho_{k,\overline\eta^{k(i,j)}\overline{2\eta}^{k-k(i,j)},\overline{\eta}^k}^{-1}(U). \end{equation}

By the recursive definition \eqref{eq:Z**Def} of $Z^{**}_{k,\overline\eta^{k(i,j)}\overline{2\eta}^{k-k(i,j)},k(i,j)}$, it follows that
\begin{equation} \hat W_{i,j} \subseteq Z^{**}_{k,\overline\eta^{k(i,j)}\overline{2\eta}^{k-k(i,j)},k(i,j)}. \end{equation}
For $z_k,z_k'\in \hat W_{i,j}$, these correspond to points $z_{k-1},z_{k-1}'$ according to the definition of $\hat W_{i,j}$ (\eqref{eq:OpenCoverFactWijDef1} and \eqref{eq:OpenCoverFactWijDef2}), so that by the definition \eqref{eq:nuDef1} of $\nu_{k,\bar\eta,z_k}$,
\begin{equation}
\nu_{k,\overline\eta^{k(i)}\overline{2\eta}^{k-k(i)},z_k} = \nu_{k-1,\overline\eta^{k(i)}\overline{2\eta}^{k-1-k(i)},z_{k-1}} \circ \lambda_k,
\end{equation}
and
\begin{equation}
\nu_{k,\overline\eta^{k(i)}\overline{2\eta}^{k-k(i)},z_k'} = \nu_{k-1,\overline\eta^{k(i)}\overline{2\eta}^{k-1-k(i)},z_{k-1}'} \circ \lambda_k.
\end{equation}
Using this, and \eqref{eq:OpenCoverFactnuApproxInd}, it follows that for $a \in \mathcal F$,
\begin{equation}
\nu_{k,\overline\eta^{k(i)}\overline{2\eta}^{k-k(i)},z_k}(\lambda_{k_0}^k(a)) \approx_\e 
\nu_{k,\overline\eta^{k(i)}\overline{2\eta}^{k-k(i)},z_k'}(\lambda_{k_0}^k(a)).
\end{equation}

Define
\begin{equation} W_{i,j} := \rho_{k,\overline{2\eta}^{k},\overline\eta^{k(i,j)}\overline{2\eta}^{k-k(i,j)}}^{-1}(\hat W_{i,j}) \subseteq Z_{k,\overline{2\eta}^k}. \end{equation}
We have shown that $W_{i,j}$ has the properties required of each $W_i$ in the statement of the lemma.
Moreover, since
\begin{equation} \rho_{k-1,\overline{2\eta}^{k-1},\overline\eta^{k-1}}^{-1}(V'_i) = \coprod_j W'_{i,j}, \end{equation}
it follows that
\begin{equation}
\label{eq:OpenCoverFactViDecomp} \rho_{k,\overline{2\eta}^k,\overline\eta^{k-1}}^{-1}(V_i) = \coprod_j W_{i,j}.
\end{equation}

Consequently, combining these for all $i$, we have
\begin{align} 
\notag
\rho_{k,\overline{2\eta}^k,\overline\eta^k}^{-1}(U) &= \coprod_{i=1}^p \rho_{k,\overline{2\eta}^k,\overline\eta^k}^{-1}(V_i) \\
&= \coprod_{i=1}^q \coprod_{j=1}^{q_i} W_{i,j}.
\end{align}
We conclude by relabelling
the family $(W_{i,j})_{i,j}$ as $W_1,\dots,W_q$.
\end{proof}

\subsection{What to do with a partiton of unity}

Let $(e_i)_{i=1}^p$ be a finite family of positive functions in $\mathrmm C(Z_{k_0,\bar\eta},D)$ for some unital $\mathrmm C^*$-algebra $D$, subordinate to the open cover $\mathcal U_{k_0}$, which means that for each $i=1,\dots,p$,
\begin{equation} U_i := \{z_{k_0} \in Z_{{k_0},\bar\eta} \mid e_i(z_{k_0}) \neq 0\} \in \mathcal U_{k_0}. \end{equation}
(We will soon ask that $e_i$ is an approximate partition of unity, although this is not required yet.)

By Lemma \ref{lem:OpenCoverFact}, for each $i$, we have the decomposition
\begin{equation} 
\label{eq:UiDecomp}
\rho_{\overline{2\eta}^{k_0}, \overline\eta^{k_0}}^{-1}(U_i) = W_{i,1} \amalg \cdots \amalg W_{i,q_i} \end{equation}
where, for each $j=1,\dots,q_i$, there is $k(i,j) \in \{0,\dots,{k_0}\}$ such that the open set
\begin{equation}
\label{eq:HatWijDef}
\hat W_{i,j} := \rho_{{k_0},\overline{2\eta}^{k_0},\overline\eta^{k(i,j)}\overline{2\eta}^{{k_0}-k(i,j)}}(W_{i,j}) \end{equation}
is a component of $\rho_{{k_0},\overline\eta^{k(i,j)}\overline{2\eta}^{{k_0}-k(i,j)}}^{-1}(U)$,
and
\begin{equation}
\label{eq:HatWijSubsetZ**}
\hat W_{i,j} \subseteq Z^{**}_{{k_0},\overline\eta^{k(i,j)}\overline{2\eta}^{{k_0}-{k(i,j)}},k(i,j)}. \end{equation}

Define
\begin{equation} 
\label{eq:eijDef}
e_{i,j} := (e_i \circ \rho_{{k_0},\overline\eta^{k(i,j)}\overline{2\eta}^{{k_0}-k(i,j)},\overline{\eta}^{k_0}})|_{\hat W_{i,j}} \in C_0(\hat W_{i,j},D)_+, \end{equation}
and use this to define a c.p.c.\ order zero map
\begin{equation} \phi_{i,j}\colon M_{m_{k(i,j)}} \to A_{k_0} \otimes D \end{equation}
by
\begin{equation} 
\label{eq:phiDef}
\phi_{i,j}(\kappa) := \xi_{{k_0},\overline\eta^{k(i,j)}\overline{2\eta}^{{k_0}-k(i,j)},k(i,j)}(\kappa \otimes e_{i,j}). \end{equation}
Also, pick some 
\begin{equation}
\label{eq:zijDef}
z_{k_0}^{(i,j)} \in \hat W_{i,j}
\end{equation}
and set
\begin{equation}
\label{eq:psiDef}
\psi_{i,j}:=\nu_{{k_0},\overline\eta^{k(i,j)}\overline{2\eta}^{{k_0}-k(i,j)},z_{k_0}^{(i,j)}}\colon A_{k_0} \to M_{m_{k(i,j)}} \end{equation}
(the range is correct by \eqref{eq:HatWijSubsetZ**} and Lemma \ref{lem:ZkPartition} (i)).
Also set 
\begin{equation} \phi_i := \sum_{j=1}^{q_i} \phi_{i,j}\colon \bigoplus_{j=1}^{q_i} M_{m_{k(i,j)}} \to A_{k_0} \otimes D. \end{equation}

To save space, let us write $\sigma_k$ to mean $(\sigma_k \otimes \id_D)\colon A_k \otimes D \to \mathrm C(CX_k, M_{m_k}) \otimes D$, and likewise for other maps such as $\lambda_{k'}^k$ and $\mu_{k,\bar\eta,\alpha,z_k}$.

\begin{lemma}
\label{lem:phiPtEval}
Let $i=1,\dots,p$, $j=1,\dots,q_i$, and $k =0,\dots,{k_0}$.
Consider the c.p.c.\ map $\sigma_k \circ \lambda_{k_0}^k \circ \phi_{i,j}\colon M_{m_{k(i,j)}}\to C(CX_k,M_{m_k}) \otimes D$.

(i) If $k<k(i,j)$ then
\begin{equation}
\sigma_{k} \circ \lambda_{k_0}^k(\phi_{i,j}(1_{m_{k(i,j)}})) = 0.
\end{equation}

(ii) If $k=k(i,j)$ then
\begin{equation}
\sigma_{k} \circ \lambda_{k_0}^k(\phi_{i,j}(1_{m_{k(i,j)}})) = 1_{m_{k}} \otimes e_{i,j}|_{CX_{k}|_{(\eta,1]}},
\end{equation}
where we canonically identify $M_{m_{k}} \otimes C_0(CX_{k}|_{(\eta,1]}, D)$ (to which the right-hand side belongs) with a subalgebra of $C(CX_{k}, M_{m_{k}}) \otimes D$ (for the left-hand side).

(iii) If $k\geq k(i,j)$ and $(x_{k},t) \in X_{k} \times (0,2\eta)$ then
\begin{gather}
\notag
\sigma_{k} \circ \lambda_{k_0}^k(\phi_{i,j}(1_{m_{k(i,j)}}))(x_{k},t) \\
\notag
= \sum_{z_{k-1}} \mu_{k-1,\overline{2\eta}^{k-1},\hat\theta_{k}(x_{k}),z_{k-1}}(1_{m_{k(i,j)}}) \\
\otimes e_{i} \circ \rho_{k,\overline\eta^{k},\overline{2\eta}^{k}}((x_{k},z_{k-1}),t).
\end{gather}
where the sum is over all
\begin{equation} z_{k-1} \in S_{k-1,\overline{2\eta}^{k-1},m_{k}}(\hat\theta_{k}(x_{k})) \end{equation}
for which $((x_k,z_{k-1}),t) \in W_{i,j}$.

(iv) If $k > k(i,j)$ then
\begin{equation} 
\sigma_{k} \circ \lambda_{k_0}^k(\phi_{i,j}(1_{m_{k(i,j)}})) \end{equation}
vanishes on $CX_{k}|_{[2\eta,1]}$.
\end{lemma}

\begin{proof}
We compute
\begin{eqnarray}
\notag
&&\hspace*{-15mm} \sigma_{k} \circ \lambda_{k_0}^k(\phi_{i,j}(1_{m_{k(i,j)}})) \\
\notag
&\stackrel{\eqref{eq:phiDef}}=& \sigma_{k} \circ \lambda_{k_0}^k \circ \xi_{{k_0},\overline{\eta}^{k(i,j)}\overline{2\eta}^{{k_0}-k(i,j)},k(i,j)} (1_{m_{k(i,j)}} \otimes e_{i,j}) \\
\label{eq:phiPtEvalComp1}
&\stackrel{\eqref{eq:xiDef}}=& \begin{cases}
0,\quad &\text{if }k<k(i,j), \\
\check\xi_{k,\overline{\eta}^{k(i,j)}\overline{2\eta}^{k-k(i,j)},k(i,j)} & \\
\quad (1_{m_{k(i,j)}} \otimes e_{i,j})|_{Z^{**}_{k,\overline{\eta}^{k(i,j)}\overline{2\eta}^{k-k(i,j)}}},\quad &\text{if }k\geq k(i,j),
\end{cases}
\end{eqnarray}
where, in the latter case, as in \eqref{eq:xiDef}, we view $Z^{**}_{k,\overline{\eta}^{k(i,j)}\overline{2\eta}^{k-k(i,j)}}$ as a subset of $Z_{{k_0},\overline\eta^{k(i,j)}\overline{2\eta}^{{k_0}-k(i,j)}}$.
The former case proves (i).
Point (iv) follows from the latter case and Lemma \ref{lem:checkXiDef}.

Picking up in the latter case, for $k = k(i,j)$, we have
\begin{eqnarray}
\notag
&&\hspace*{-15mm} \sigma_{k} \circ \lambda_{k_0}^k(\phi_{i,j}(1_{m_{k(i,j)}})) \\
\notag
&=& \check\xi_{k,\overline{\eta}^{k(i,j)}\overline{2\eta}^{{k_0}-k(i,j)},k(i,j)} (1_{m_{k(i,j)}} \otimes e_{i,j}|_{Z^{**}_{k,\overline{\eta}^{k(i,j)}\overline{2\eta}^{k-k(i,j)}}}) \\
&\stackrel{\eqref{eq:checkXiTrivDef}}=&
1_{m_{k(i,j)}} \otimes e_{i,j}|_{CX_{k}|_{(\eta,1]}},
\end{eqnarray}
thus establishing (ii).

On the other hand, for $k > k(i,j)$ and for $(x_{k},t) \in X_{k} \times (0,2\eta)$,
\begin{eqnarray}
\notag
&&\hspace*{-15mm} \sigma_{k} \circ \lambda_{k_0}^k(\phi_{i,j}(1_{m_{k(i,j)}}))(x_{k},t) \\
\notag
&\stackrel{\eqref{eq:phiPtEvalComp1}}=& \check\xi_{k,\overline{\eta}^{k(i,j)}\overline{2\eta}^{k-k(i,j)},k(i,j)} (1_{m_{k(i,j)}} \otimes e_{i,j}|_{Z^{**}_{k,\overline{\eta}^{k(i,j)}\overline{2\eta}^{k-k(i,j)}}})(x_{k},t) \\
\notag
&\stackrel{\eqref{eq:checkXiDef}}=& \sum_{z_{k-1}} \mu_{k-1,\overline\eta^{k(i,j)}\overline{2\eta}^{k-1-k(i,j)},\hat\theta_{k}(x_{k}),z_{k-1}}(1_{m_{k(i,j)}})\\
\label{eq:phiPtEvalComp2}
&&\quad \otimes e_{i,j}((x_{k},z_{k-1}),t),
\end{eqnarray}
where the sum is taken over all
\begin{align}
\notag
 z_{k-1} \in &S_{k-1,\overline\eta^{k(i,j)}\overline{2\eta}^{k-1-k(i,j)},m_{k}}(\hat\theta_{k}(x_{k})) \cap \\
\label{eq:phiPtEvalzQuant}
&\quad Z^{**}_{k-1,\overline\eta^{k(i,j)}\overline{2\eta}^{k-1-k(i,j)},k(i,j)}.
\end{align}
For $z_{k-1}$ as in \eqref{eq:phiPtEvalzQuant}, by Lemma \ref{lem:omegaMuCompat},
\begin{align}
\notag
&\mu_{k-1,\overline{\eta}^{k(i,j)}\overline{2\eta}^{k-1-k(i,j)},\hat\theta_k(x_k),z_{k-1}}(1_{m_{k(i,j)}}) \\
\label{eq:phiPtEvalComp3}
&\quad = \sum_{z_{k-1}'} \mu_{k-1,\overline{2\eta}^{k-1},\hat\theta_k(x_k),z_{k-1}'}(1),
\end{align}
where the sum is over all
\begin{align}
\notag z_{k-1}' \in &\rho_{k-1,\overline{2\eta}^{k-1},\overline{\eta}^{k(i,j)}\overline{2\eta}^{k-1-k(i,j)}}^{-1}(\{z_{k-1}\}) \\
&\subset Z_{k-1,\overline{2\eta}^{k-1}}.
\end{align}
For such $z_{k-1}$ and $z_{k-1}'$,
\begin{eqnarray}
\notag
&& e_{i,j}((x_k,z_{k-1}),t) \\
\notag
&=& e_{i,j}((x_k,\rho_{k-1,\overline{2\eta}^{k-1},\overline{\eta}^{k(i,j)}\overline{2\eta}^{k-1-k(i,j)}}(z_{k-1}')),t) \\
\notag
&\stackrel{\begin{tabular}{@{}c}{\scriptsize \eqref{eq:rhoDefiv},\eqref{eq:rhoFacts2},} \\ \scriptsize{\eqref{eq:eijDef}} \end{tabular}}=& 
\begin{cases} e_i \circ \rho_{k_0,\overline{2\eta}^{k_0},\overline{\eta}^{k_0}}&((x_k,z_{k-1}'),t),\\
&\quad\text{if }((x_k,z_{k-1}),t) \in \hat W_{i,j}; \\ 0, &\quad\text{otherwise} \end{cases} \\
\label{eq:phiPtEvalComp4}
&\stackrel{\eqref{eq:rhoFacts2},\eqref{eq:HatWijDef}}=& \begin{cases} e_i \circ \rho_{k_0,\overline{2\eta}^{k_0},\overline{\eta}^{k_0}}&((x_k,z_{k-1}'),t),\\
&\quad\text{if }((x_k,z_{k-1}'),t) \in W_{i,j}; \\ 0,&\quad\text{otherwise,} \end{cases}
\end{eqnarray}
where we are viewing $Z_{k,\overline{2\eta}^k}$ as a subset of $Z_{k_0,\overline{2\eta}^{k_0}}$.
Putting together \eqref{eq:phiPtEvalComp2}, \eqref{eq:phiPtEvalComp3}, and \eqref{eq:phiPtEvalComp4}, and using \eqref{eq:HatWijSubsetZ**}, we obtain
\begin{gather}
\notag
\sigma_{k} \circ \lambda_{k_0}^k(\phi_{i,j}(1_{m_{k(i,j)}}))(x_{k},t) \\
\notag
= \sum_{z_{k-1}} \mu_{k-1,\overline{2\eta}^{k-1},\hat\theta_{k}(x_{k}),z_{k-1}}(1_{m_{k(i,j)}}) \\
\otimes e_{i} \circ \rho_{k,\overline\eta^{k},\overline{2\eta}^{k}}((x_{k},z_{k-1}),t),
\end{gather}
where the sum is over all
\begin{equation} z_{k-1} \in S_{k-1,\overline{2\eta}^{k-1},m_{k}}(\hat\theta_{k}(x_{k})) \end{equation}
for which $((x_k,z_{k-1}),t) \in W_{i,j}$.
This establishes (iii), in the case $k>k(i,j)$.

In the case $k=k(i,j)$, (iii) follows from (ii) and Lemma \ref{lem:omegaMuCompat}.
\end{proof}

\begin{lemma}
\label{lem:CpcFacts}
(i) For each $i$, the map $\phi_i$ is c.p.c.\ order zero.

(ii) If $e_{i_1}$ and $e_{i_2}$ are orthogonal, then so are the ranges of $\phi_i$ and $\phi_i'$.
Therefore, if the elements $(e_i)_i$ are $(n+1)$-colourable then so also is the map $\phi=\sum_{i,j} \phi_{i,j}$.

(iii)
If $(e_i)_i$ is an exact partition of unity then $\sum_{i,j} \phi_i(1_{m_{k(i,j)}}) = 1_B$.

(iii)$'$
If $\sum_i e_i \approx_\e 1$ (i.e., $(e_i)_i$ is an \textbf{$\e$-approximate partition of unity}) then $\sum_{i,j} \phi_{i,j}(1_{m_{k(i,j)}}) \approx_\e 1_B$.
\end{lemma}

\begin{proof}
(i) and (ii):
$\phi_{i,j}$ is quite clearly a c.p.c.\ order zero map.
To show that $\phi_i$ is order zero, it therefore suffices to show that
\begin{equation} \phi_{i,j_1}(1_{m_{k(i,j_1)}}) \quad\text{and}\quad \phi_{i,j_2}(1_{m_{k(i,j_2)}}) \end{equation}
are orthogonal for $j_1\neq j_2$.
The key fact to be used here is
\begin{equation} W_{i,j_1} \cap W_{i,j_2} = \emptyset. \end{equation}

Similarly, for (ii), we need to show that, for all $j_1,j_2$,
\begin{equation} \phi_{i_1,j_1}(1_{m_{k(i_1,j_1)}}) \quad\text{and}\quad \phi_{i_1,j_2}(1_{m_{k(i_1,j_2)}}) \end{equation}
are orthogonal, where we can assume that 
\begin{equation} W_{i_1} \cap W_{i_2} = \emptyset. \end{equation}

Thus, (i) and (ii) both reduce to showing that, if
\begin{equation} W_{i_1,j_1} \cap W_{i_2,j_2} = \emptyset \end{equation}
then $\phi_{i_1,j_1}(1)$ and $\phi_{i_2,j_2}(1)$ are orthogonal, for which it is sufficient to show that they have orthogonal images in each irreducible representation, i.e., that for $k =0,\dots,{k_0}$ and $v_{k} \in CX_{k}|_{(0,1]}$, the two operators
\begin{equation}
\label{eq:CpcFactsToShowOrthog}
 \sigma_{k} \circ \lambda_{k_0}^k(\phi_{i_s,j_s}(1))(v_{k}),\quad s=1,2, \end{equation}
are orthogonal.

Three cases naturally arise.

(a) $k(i_s,j_s)>k$ for some $s$: By Lemma \ref{lem:phiPtEval} (i), one of the elements in \eqref{eq:CpcFactsToShowOrthog} is already 0.

(b) $k(i_s,j_s)=k$ for $s=1,2$:
In this case, $e_{i_1,j_1}$ and $e_{i_2,j_2}$ are orthogonal (their supports lie in disjoint sets $W_{i_1,j_1}$ and $W_{i_2,j_2}$ respectively), and so by Lemma \ref{lem:phiPtEval} (ii), the elements in \eqref{eq:CpcFactsToShowOrthog} are orthogonal.

(c) $k(i_s,j_s) \leq k$ for $s=1,2$, with at least one inequality strict:
If $v_{k} \in CX_{k}|_{[2\eta,0)}$ then by Lemma \ref{lem:phiPtEval} (iv), one of the elements in \eqref{eq:CpcFactsToShowOrthog} is already zero.
So assume that $v_{k} \in CX_{k}|_{(0,2\eta)}$, so that the elements in \eqref{eq:CpcFactsToShowOrthog} are both described in Lemma \ref{lem:phiPtEval} (iii).

The sum in Lemma \ref{lem:phiPtEval} (iii) is taken over a subset of $W_{i,j}$, and therefore, for $(i,j)=(i_1,j_1)$ and $(i,j)=(i_2,j_2)$, it is taken over disjoint sets.
Since the ranges of $\mu_{k-1,\overline{2\eta}^{k-1},\hat\theta_{k}(x_{k}),z_{k-1}}$ are orthogonal as $z_{k-1}$ varies, it follows that the two elements in \eqref{eq:CpcFactsToShowOrthog} are orthogonal.

(iii) follows from (iii)$'$, which we do presently.

(iii)$'$:
We work again in each irreducible representation, that is, we will show that for $k =0,\dots,{k_0}$ and $v_{k} \in CX_{k}|_{(0,1]}$,
\begin{equation}
\label{eq:CpcFactsToShowPOU}
\sum_{i,j} \sigma_{k} \circ \lambda_{k_0}^k(\phi_{i,j}(1))(v_{k}) \approx_\e 1_{m_{k}}.
\end{equation}

We must break our analysis into two cases.

(a)
$v_{k} \in CX_{k}|_{[2\eta,1]}$:
By Lemma \ref{lem:phiPtEval} (i) and (iv), the only contribution to the sum in \eqref{eq:CpcFactsToShowPOU} comes from $(i,j)$, for which $k(i,j)=k$, and then by Lemma \ref{lem:phiPtEval} (ii), we get
\begin{align}
\notag
\sum_{i,j} \sigma_{k} \circ \lambda_{k_0}^k(\phi_{i,j}(1))(v_{k}) &=
\sum_{i,j} 
1_{m_{k}} \otimes e_{i,j}(v_{k}) \\
&\approx_{\e} 1_{m_{k}} \otimes 1_D,
\end{align}
(by hypothesis).

(b)
$v_{k} = (x_{k},t) \in CX_{k}|_{(0,2\eta)}$:
By Lemma \ref{lem:phiPtEval} (i), the only contributions to the sum in \eqref{eq:CpcFactsToShowPOU} come from those $(i,j)$ for which $k(i,j)\geq k$.
From Lemma \ref{lem:phiPtEval} (iii) and \eqref{eq:UiDecomp}, we derive that
\begin{align}
\notag
&\sigma_{k} \circ \lambda_{k_0}^k(\phi_{i}(1))(x_{k},t) \\
&\quad= \sum_{z_{k-1}}  \mu_{k-1,\overline{2\eta}^{k-1},\hat\theta_{k}(x_{k}),z_{k-1}}(e_{i} \circ \rho_{k,\overline\eta^{k},\overline{2\eta}^{k}}((x_{k},z_{k-1}),t)),
\end{align}
where the sum is over all
\begin{equation} z_{k-1} \in S_{k-1,\overline{2\eta}^{k-1},m_{k}}(\hat\theta_{k}(x_{k})). \end{equation}
Thus,
\begin{eqnarray}
\notag
&&\hspace*{-15mm}\sum_{i,j} \sigma_{k} \circ \lambda_{k_0}^k(\phi_{i,j}(1))(v_{k}) \\
\notag
&=& \hspace*{-5mm}\sum_{i,z_{k-1}}  \mu_{k-1,\overline{2\eta}^{k-1},\hat\theta_{k}(x_{k}),z_{k-1}}(e_{i} \circ \rho_{k,\overline\eta^{k},\overline{2\eta}^{k}}((x_{k},z_{k-1}),t)) \\
\notag
&\approx_{\e}&
\sum_{z_{k-1}} \mu_{k-1,\overline{2\eta}^{k-1},\hat\theta_{k}(x_{k}),z_{k-1}}(1) \\
&\stackrel{\text{Lemma\ }\ref{lem:PointwiseApprox}}=& 1,
\end{eqnarray}
where the approximation to within $\e$ is possible since the summands are orthogonal.
\end{proof}

Recall that the open cover $\mathcal U_{k_0}$ depended on open covers $\mathcal V_k$ of $CX_k$ which in turn depended on an open set $\mathcal F$ and a tolerance $\dl>0$, by the fact that $\sigma_k\circ \lambda_{k_0}^k(a)$ is approximately constant on each set in $\mathcal V_k$, for $a \in \mathcal F$.
We exploit this relationship in proving the following.

\begin{lemma}
\label{lem:CpcApproximation}
Suppose that $(e_i)_i$ is an $(n+1)$-colourable $\e$-approximate partition of unity.
Let $a \in \mathcal F$.
Then
\begin{equation} \sum_{i,j} \phi_{i,j} \circ \psi_{i,j}(a) \approx_{(n+1)\dl+\e} a \otimes 1_D. \end{equation}
\end{lemma}

\begin{proof}
Let us first show that
\begin{equation}
\label{eq:CpcApproximationToShow}
 \phi_{i,j}(1)\cdot(a \otimes 1_D) = (a \otimes 1_D)\cdot\phi_{i,j}(1) \approx_\dl \phi_{i,j} \circ \psi_{i,j}(a) \end{equation}
for every $i,j$, by showing that this holds in every irreducible representation.

To this end, let $k =0,\dots,{k_0}$, let $v_{k} \in CX_{k}|_{(0,1]}$ and consider images under the representation
\begin{equation} \ev_{v_{k}} \circ \sigma_{k} \circ \lambda_{k_0}^k. \end{equation}
Set
\begin{equation} a_{k} := \lambda_{k_0}^k(a) \in A_{k}. \end{equation}

The analysis divides naturally into three cases.

(i) $k < k(i,j)$:
In this case, everything in \eqref{eq:CpcApproximationToShow} is zero in $A_{k}$, by Lemma \ref{lem:phiPtEval} (i).

(ii) $k=k(i,j)$:
By Lemma \ref{lem:phiPtEval} (ii),
\begin{eqnarray}
\notag
&&\hspace*{-20mm}\sigma_{k} \circ \lambda_{k_0}^k(\phi_{i,j}(1)\cdot (a \otimes 1_D))(v_{k}) \\
\notag
&=& (1_{m_{k}} \otimes e_{i,j}(v_{k}))\cdot(\sigma_{k}(a_{k})(v_{k}) \otimes 1_D) \\
\notag
&=& (\sigma_{k}(a_{k})(v_{k}) \otimes e_{i,j}(v_{k})) \\
\notag
&=& \nu_{k_0,\overline{\eta}^{k(i,j)}\overline{2\eta}^{k_0-k(i,j)},v_k}(a) \otimes e_{i,j}(v_k) \\
\notag
&\stackrel{\begin{tabular}{@{}c}{\scriptsize \eqref{eq:OpenCoverFactnuApprox},\eqref{eq:eijDef},} \\ \scriptsize{\eqref{eq:zijDef}} \end{tabular}}{\approx_\dl}& 
\nu_{k_0,\overline{\eta}^{k(i,j)}\overline{2\eta}^{k_0-k(i,j)},z_{k_0}^{(i,j)}}(a) \otimes e_{i,j}(v_k) \\
\notag
&\stackrel{\eqref{eq:psiDef}}=& (\psi_{i,j}(a) \otimes e_{i,j}(v_{k})) \\
&=& \sigma_{k} \circ \lambda_{k_0}^k(\phi_{i,j} \circ \psi_{i,j}(a)).
\end{eqnarray}
We clearly get the same thing in the second line if we start with the expression
\begin{equation}
\sigma_{k} \circ \lambda_{k_0}^k((a \otimes 1_D)\cdot\phi_{i,j}(1))(v_{k}), \end{equation}
and thus we have established \eqref{eq:CpcApproximationToShow} in this irreducible representation.

(iii) $k > k(i,j)$:
By Lemma \ref{lem:phiPtEval} (iv) and the fact that $\phi_{i,j}$ is order zero, everything in \eqref{eq:CpcApproximationToShow} is zero in this representation for $v_{k} \in CX_{k}|_{[2\eta,1]}$.
For $v_{k} = (x_{k},t) \in CX_{k}|_{(0,2\eta)}$, note that 
\begin{equation}
\label{eq:CpcApproximationQuicky}
\sigma_k(a_k) = \hat\theta_k(x_k)(a_{k-1}).
\end{equation}
As in Lemma \ref{lem:phiPtEval} (iii), we have
\begin{eqnarray}
\notag
&&\hspace*{-15mm}\sigma_{k} \circ \lambda_{k_0}^k(\phi_{i,j}(1)\cdot(a \otimes 1_D))(v_{k}) \\
\notag
&\stackrel{\eqref{eq:phiDef}}=&
\sigma_k \circ \lambda_{k_0}^k(\xi_{k_0,\bar\eta^{k(i,j)}\overline{2\eta}^{k_0-k(i,j)},k(i,j)}(1_{m_{k(i,j)}} \otimes e_{i,j})(v_k)\\
\notag
&&\quad\cdot (\sigma_k(a_k) \otimes 1_D)(v_k) \\
\notag
&\stackrel{\eqref{eq:xiDef}}=&
\check\xi_{k,\overline{\eta}^{k(i,j)}\overline{2\eta}^{k-k(i,j)},k(i,j)} \big(1 \otimes e_{i,j}|_{Z^{**}_{k,\overline{\eta}^{k(i,j)}\overline{2\eta}^{k-k(i,j)}}}\big)(v_k) \\
\notag
&&\quad \cdot(\sigma_{k}(a_{k}) \otimes 1_D)(v_{k}) \\
\notag
&\stackrel{\eqref{eq:checkXiDef},\eqref{eq:eijDef}}=& \textstyle{\sum_{z_{k-1} \in S_{k-1,\overline\eta^{k(i,j)}\overline{2\eta}^{k-1-k(i,j)},m_{k}}(\hat\theta_{k}(x_{k})) \cap \hat W_{i,j}}} \\
\notag
&& \hspace*{-25mm}\big(\mu_{k-1,\overline\eta^{k(i,j)}\overline{2\eta}^{k-1-k(i,j)},\hat\theta_{k}(x_{k}),z_{k-1}}(1) \otimes e_{i,j}(v_k)\big)\cdot(\sigma_k(a_{k-1}) \otimes 1_D)(v_k) \\
\notag
&\stackrel{\eqref{eq:CpcApproximationQuicky}}=& \textstyle{\sum_{z_{k-1} \in S_{k-1,\overline\eta^{k(i,j)}\overline{2\eta}^{k-1-k(i,j)},m_{k}}(\hat\theta_{k}(x_{k})) \cap \hat W_{i,j}}} \\
\notag
&& \hspace*{-25mm}\big(\mu_{k-1,\overline\eta^{k(i,j)}\overline{2\eta}^{k-1-k(i,j)},\hat\theta_{k}(x_{k}),z_{k-1}}(1) \otimes e_{i,j}(v_k)\big)\cdot(\hat\theta_k(x_k)(a_{k-1}) \otimes 1_D) \\
\notag
&=& \sum_{z_{k-1}} \big(\mu_{k-1,\overline\eta^{k(i,j)}\overline{2\eta}^{k-1-k(i,j)},\hat\theta_{k}(x_{k}),z_{k-1}}(1)\cdot\hat\theta_k(x_k)(a_{k-1})\big) \\
\notag
&& \quad \otimes e_{i,j}(v_k) \\
\label{eq:CpcApproximationMidComp}
&\stackrel{\eqref{eq:PointwiseApproxMult}}=& \sum_{z_{k-1}} \mu_{k-1,\overline\eta^{k(i,j)}\overline{2\eta}^{k-1-k(i,j)},\hat\theta_{k}(x_{k}),z_{k-1}} \\
\notag
&&\quad \circ \nu_{k-1,\overline\eta^{k(i,j)}\overline{2\eta}^{k-1-k(i,j)},z_{k-1}}(a_{k-1}) \otimes e_{i,j}(v_k) \\
\notag
&\stackrel{\eqref{eq:nuDef1}}=& \sum_{z_{k-1}} \mu_{k-1,\overline\eta^{k(i,j)}\overline{2\eta}^{k-1-k(i,j)},\hat\theta_{k}(x_{k}),z_{k-1}} \\
\notag
&&\quad \circ \nu_{k_0,\overline\eta^{k(i,j)}\overline{2\eta}^{k_0-k(i,j)},((x_k,z_{k-1}),t)}(a) \otimes e_{i,j}(v_k) \\
\notag
&\stackrel{\begin{tabular}{@{}c}{\scriptsize \eqref{eq:OpenCoverFactnuApprox},\eqref{eq:eijDef},} \\ \scriptsize{\eqref{eq:zijDef}} \end{tabular}}{\approx_\dl}& 
\sum_{z_{k-1}} \mu_{k-1,\overline\eta^{k(i,j)}\overline{2\eta}^{k-1-k(i,j)},\hat\theta_{k}(x_{k}),z_{k-1}} \\
\notag
&&\quad \circ \nu_{k_0,\overline\eta^{k(i,j)}\overline{2\eta}^{k_0-k(i,j)},z_{k_0}^{(i,j)}}(a) \otimes e_{i,j}(v_k) \\
\notag
&\stackrel{\eqref{eq:psiDef}}=& 
\sum_{z_{k-1}} \mu_{k-1,\overline\eta^{k(i,j)}\overline{2\eta}^{k-1-k(i,j)},\hat\theta_{k}(x_{k}),z_{k-1}}(\psi_{i,j}(a)) \\
\notag
&&\qquad \otimes e_{i,j}(v_{k}) \\
\notag
&\stackrel{\eqref{eq:checkXiDef}}=& \check\xi_{k_0,\overline{\eta}^{k(i,j)}\overline{2\eta},k}(\psi_{i,j}(a) \otimes e_{i,j}(v_k)) \\
\notag
&\stackrel{\eqref{eq:xiDef}}=& \left(\sigma_k \circ \lambda_{k_0}^k \circ \xi_{k_0,\overline{\eta}^{k(i,j)}\overline{2\eta},k}(\psi_{i,j}(a) \otimes e_{i,j})\right)(v_k) \\
&\stackrel{\eqref{eq:phiDef}}=& \sigma_k \circ \lambda_{k_0}^k \circ \phi_{i,j} \circ \psi_{i,j}(a),
\end{eqnarray}
where we have used the fact that the summands are orthogonal to get the approximation to within $\dl$.
We clearly get the same thing in \eqref{eq:CpcApproximationMidComp} if we start with the expression
\begin{equation}
\sigma_{k} \circ \lambda_{k_0}^k((a \otimes 1_D)\cdot\phi_{i,j}(1))(v_{k}), \end{equation}
and thus we have shown \eqref{eq:CpcApproximationToShow} in this irreducible representation.

We have finished establishing \eqref{eq:CpcApproximationToShow}.

By Lemma \ref{lem:CpcFacts} (ii), we may colour (i.e., partition) the set of all indices $(i,j)$ into
\begin{equation} I_0 \amalg \cdots \amalg I_n, \end{equation}
in such a way that for $c=0,\dots,n$,
\begin{equation} (\phi_{i,j})_{(i,j) \in I_c} \end{equation}
have orthogonal ranges.
It follows from \eqref{eq:CpcApproximationToShow} that
\begin{equation}
\label{eq:CpcApproximationOneColour} \sum_{(i,j) \in I_c} \phi_{i,j}(1)(a \otimes 1_D) \approx_{\dl} \sum_{(i,j) \in I_c} \phi_{i,j} \circ \psi_{i,j}(a). \end{equation}
Therefore,
\begin{eqnarray}
\notag
a \otimes 1_D &\stackrel{\mathrm{Lemma\ }\ref{lem:CpcFacts}\mathrm{\ (iii)}'}{\approx_\e}& \sum_{i,j} \phi_{i,j}(1)(a \otimes 1_D) \\
\notag
&=& \sum_{c=0}^n \sum_{(i,j) \in I_c} \phi_{i,j}(1)(a \otimes D) \\
\notag
&\stackrel{\eqref{eq:CpcApproximationOneColour}}{\approx_{(n+1)\dl}}& \sum_{c=0}^n \sum_{(i,j) \in I_c} \phi_{i,j}\circ\psi_{i,j}(a) \\
&=& \sum_{i,j} \phi_{i,j}\circ\psi_{i,j}(a),
\end{eqnarray}
as required.
\end{proof}

\subsection{Proof of Theorem \ref{thm:drBound} and applications}

\begin{proof}[Proof of Theorem \ref{thm:drBound}]
Since $A$ is an NC cell complex, we may assume it is $A_{k_0}$.

Since $A_{k_0}$ is locally approximated by $A_{{k_0},\bar\eta}$, it suffices to show that
\begin{equation} \dr (A_{{k_0},\bar\eta} \otimes 1_D \subseteq A_{k_0} \otimes D) \leq n. \end{equation}
Let $\mathcal F \subseteq A_{{k_0},\bar\eta}$ and let $\e >0$.
Using this $\mathcal F$ and $\dl:=\frac\e{2(n+1)}$, form the open cover $\mathcal U_{k_0}$ as in Section \ref{sec:OpenCover}.
By the hypothesis, applied to $Z:=Z_{{k_0},\bar\eta}$, and by \cite[Proposition 3.2]{TW:Zdr}, there exists an $(n+1)$-colourable $\e/2$-approximate partition of unity $(e_i)_i$ in $\mathrmm C(Z_{{k_0},\bar\eta},D)$ subordinate to $\mathcal U_{k_0}$.
Form the c.p.c.\ maps $\phi_{i,j}$ and $\psi_{i,j}$ as in \eqref{eq:phiDef}, \eqref{eq:psiDef}.
By Lemma \ref{lem:CpcFacts} (ii), $\phi=\sum_{i,j} \phi_{i,j}$ is $(n+1)$-colourable, and by Lemma \ref{lem:CpcFacts} (iii)$'$, $\|\phi\| \leq 1+\e/2$.
Rescaling, we may arrange that $\phi$ is c.p.c.

Finally, by Lemma \ref{lem:CpcApproximation}, 
\begin{equation} \sum_{i,j} \phi_{i,j} \circ \psi_{i,j}(a) \approx_\e a \otimes 1_D \end{equation}
for all $a \in \mathcal F$.
\end{proof}

\begin{cor}[Theorem \ref{thm:MainThmA}]
\label{cor:MainThmA}
Every $\mathcal Z$-stable, locally subhomogeneous algebra $A$ satisfies $\dr A \leq 2$.
\end{cor}

\begin{proof}
Since locally subhomogeneous algebras are locally approximated by NC cell complexes, and by \cite[Proposition 2.6]{TW:Zdr}, it suffices to show, for any NC cell complex $A$, that the inclusion
\begin{equation} A \otimes \mathcal Z \otimes 1_{\mathcal Z} \subset A \otimes \mathcal Z \otimes \mathcal Z \end{equation}
has decomposition rank at most $2$.
By approximating, it suffices to show this for the inclusion
\begin{equation} A \otimes \mathcal Z_{p,q} \otimes 1_{\mathcal Z} \subset A \otimes \mathcal Z_{p,q} \otimes \mathcal Z. \end{equation}
Note that $A \otimes \mathcal Z_{p,q}$ is itself an NC cell complex, so that this inclusion has decomposition rank at most $2$ follows from Theorem \ref{thm:drBound} and \cite{TW:Zdr}.
\end{proof}

Theorem \ref{thm:drBound} also gives a different proof of the following result, the main theorem of \cite{Winter:drSH}.

\begin{cor}[\cite{Winter:drSH}]
\label{cor:drSH}
Let $A$ be a separable subhomogeneous algebra such that $\mathrm{Prim}_k(A)$ has dimension at most $n$, for each $k$.
Then $\dr A \leq n$.
\end{cor}

\begin{proof}
By Theorem \ref{thm:ApproxCell}, $A$ is locally approximated by NC cell complexes where the cells have dimension at most $n$.
(Note that this does not use the result of \cite{Winter:drSH}, since the proof of Theorem \ref{thm:ApproxCell} only uses the fact that $\mathrm{Prim}_k(A)$ has dimension at most $n$.)
Therefore, by Proposition \ref{prop:drLocal}, we may assume that $A$ is such an NC cell complex.

This means that $A$ has the form given in Section \ref{sec:Setup}, where $\dim X_i \leq n-1$ for all $i=0,\dots,k$.

In order to apply Theorem \ref{thm:drBound}, we must show that $Z_{k,\bar\eta}$ has dimension at most $n$.
Evidently, there exists a finite-to-one continuous map $Y_{k,\bar\eta} \to X_i$ and so by \cite[Theorem 1.12.4]{Engelking:DimTheory},
\begin{equation} \dim Y_{k,\bar\eta} \leq \dim X_i \leq n \end{equation}
for every $k$.
Then since $Z_{k,\bar\eta}$ can be decomposed into a closed set homeomorphic to $Z_{k-1,\bar\eta}$, an open set homeomorphic to $Y_{k,\bar\eta} \times (0,\eta)$, and a closed set homeomorphic to $CX_k$, we can show inductively that
\begin{equation} \dim Z_{k,\bar\eta} \leq n. \end{equation}
\end{proof}

\section{Minimal $\mathbb Z$-crossed products}
\label{sec:CrossedProd}

Let us now consider crossed products associated to minimal homeomorphisms $\alpha\colon X \to X$, for compact metrizable spaces $X$.
Using $\mathrmm C(X) \rtimes_\alpha \mathbb Z$ to denote the crossed product, we show the following result.

\begin{thm}
\label{thm:UHFCrossedProd}
Let $X$ be a compact metrizable space, let $\alpha\colon X \to X$ be a minimal homeomorphism, and let $U$ be an infinite dimensional UHF algebra.
Then
\begin{equation} (\mathrmm C(X) \rtimes_\alpha \mathbb Z) \otimes U \end{equation}
is locally subhomogeneous.
\end{thm}

In particular, this result shows that Theorem \ref{thm:MainThmA} applies to $(\mathrmm C(X) \rtimes_\alpha \mathbb Z) \otimes U$, and then on using \cite[Lemma 3.1]{Carrion:RFproducts}, it follows that $(\mathrmm C(X) \rtimes_\alpha \mathbb Z) \otimes \mathcal Z$ has finite decomposition rank, as explained in Corollary \ref{cor:CrProddr}, below.
In fact, using the heavy classification machinery of Gong, Lin, and Niu \cite{GongLinNiu}, Lin has shown that $(\mathrmm C(X) \rtimes_\alpha \mathbb Z) \otimes \mathcal Z$ is approximately subhomogeneous (see Remark \ref{rmk:LinCrossedClass} and Corollary \ref{cor:CrProdClass}) \cite{Lin:CrossedClass}.

It would be nice to have a direct proof (such as the proof of Theorem \ref{thm:UHFCrossedProd}) that $\mathcal Z$-stable (or even all) minimal $\mathbb Z$-crossed products are locally subhomogeneous, not using lengthy classification results.

Let $X,\alpha$ as in Theorem \ref{thm:UHFCrossedProd}, and let $x \in X$.
Let $u \in \mathrmm C(X) \rtimes_\alpha \mathbb Z$ denote the canonical unitary, so that
\begin{equation} u^*fu = f \circ \alpha \text{ for }f \in \mathrmm C(X) \subset \mathrmm C(X) \rtimes_\alpha \mathbb Z. \end{equation}
Define
\begin{equation}
\label{eq:PutnamSugalgDef}
A_x := \mathrmm C^*(\mathrmm C(X) \cup uC_0(X \setminus \{x\})) \subset \mathrmm C(X) \rtimes_\alpha \mathbb Z.
\end{equation}
This algebra is known to be locally subhomogeneous (in fact, an inductive limit of fairly explicit subhomogeneous algebras), as well as simple; these facts are largely due to Q.~Lin \cite{QLin}; see \cite[Theorem 4.1]{Phillips:arsh} for a concise account.
Many times has the structure of $\mathrmm C(X) \rtimes_\alpha \mathbb Z$ been studied using $A_x$ \cite{ElliottEvans:IrratRotation,ElliottNiu:MeanDim0,GiordanoPutnamSkau:OrbitEq,LinPhillips:CrProd,Phillips:arsh,Putnam:CantorCrProd,Putnam:KTheoryGroupoids,LinPhillips:mindiffeos,StrungWinter:IJM,TomsWinter:MinClassification,TomsWinter:rigidity}!
In Lemma \ref{lem:UHFCrossedProdSpecific}, below, Berg's technique is used to show a new link: up to UHF stabilization, $\mathrmm C(X) \rtimes_\alpha \mathbb Z$ is locally approximated by the direct sum of a corner of $A_x$ and a circle-matrix algebra.
From this, Theorem \ref{thm:UHFCrossedProd} will immediately follow.

As pointed out to us by a referee, our use of Berg's technique here is similar to the argument by Strung and Winter in \cite{StrungWinter:IJM}, used to prove that tracial approximation (by a class $\mathcal S$) passes from a $A_x\otimes U$ to $(C(X)\rtimes_\alpha \mathbb Z) \otimes U$.
Combined with classification results (\cite{EGLN:ASH}), this can give an alternate proof of Theorem \ref{thm:UHFCrossedProd}; however, this route requires knowing that $A_x\otimes U$ has finite decomposition rank (which follows from Theorem \ref{thm:MainThmA}), whereas the present argument does not rely on Theorem \ref{thm:MainThmA}, and (perhaps more significantly) does not require classification.

First, a few (known) preparatory lemmas.

In the following, for positive elements $a,e$, we write $a \vartriangleleft e$ to mean $ea=ae=a$.

\begin{lemma}
\label{lem:ProjBetween}
Let $A$ be a $\mathrmm C^*$-algebra with stable rank one.
Suppose that $a,b,c \in A_+$ are positive contractions and $p \in A$ is a projection such that 
\begin{gather}
\notag
a \vartriangleleft b \vartriangleleft c \quad \text{and} \\
[b] \leq [p] \leq [c]
\end{gather}
in the Cuntz semigroup $W(A)$.
Then there exists a projection $q \in A$ such that $a \vartriangleleft q \vartriangleleft c$ and $[p]=[q]$.
\end{lemma}

\begin{proof}
Using \cite[Proposition 2.4]{Rordam:UHFII}, we may replace $p$ by a Murray-von Neumann equivalent projection in $A \cap \{1-c\}^\perp$.
Therefore, without loss of generality, $p \vartriangleleft c$.

Since stable rank one passes to hereditary subalgebras, we may apply \cite[Proposition 2.4 (iv)]{Rordam:UHFII} to obtain a unitary $u \in (A \cap \{1-c\}^\perp)^\sim$ such that
\begin{equation} uau^* \in \her(p). \end{equation}
Therefore, $q = upu^*$ satisfies the conclusion.
\end{proof}

\begin{lemma}
\label{lem:InAx}
Let $X,\alpha$ be as in Theorem \ref{thm:UHFCrossedProd}, let $x \in X$, $h \in \mathrmm C(X)$ and $k \in \mathbb N$.
\begin{enumerate}
\item If $h(\alpha^{-i}(x)) = 0$ for $i=0,\dots,k-1$ then $u^kh \in A_x$.
\item If $h(\alpha^{i}(x)) = 0$ for $i=1,\dots,k$ then $hu^k \in A_x$.
\end{enumerate}
\end{lemma}

\begin{proof}
(i):
It suffices to show that $u^kh$ is approximately contained in $A_x$.
Therefore, perturbing $h$, we may assume that $h$ is zero in a neighbourhood of $\alpha^{-i}(x)$ for each $i=0,\dots,k-1$.
Then there exists $e \in C_0(X \setminus \{x\})$ such that
\begin{equation} (e \circ \alpha^i)h=h \end{equation}
for $i=0,\dots,k-1$.
Thus,
\begin{align}
\notag
u^kh &= u^k(e \circ \alpha^{k-1}) \cdots (e \circ \alpha)eh \\
&= (ue)^kh \in A_x.
\end{align}

(ii): Note that $hu^k = u^k(h \circ \alpha^k)$.
In view of this, (ii) follows directly from (i).
\end{proof}

\begin{lemma}
\label{lem:CPProjs}
Let $X,\alpha$ be as in Theorem \ref{thm:UHFCrossedProd}, let $x \in X$, and let $U$ be an infinite dimensional UHF algebra.
Let $\mathcal G \subset C(X)$ be a finite set, let $\eta > 0$, and let $k \in \mathbb N$.
Then there exist $m \in \mathbb N$ with $k < m$, a finite set $\mathcal H \subset C(X)$, and orthogonal projections 
\begin{equation}
p_{-k}, p_{-k+1},\dots,p_{m-1},p_m \in A_x \otimes U \cap \mathcal H'
\end{equation}
such that:

(i) $\mathcal G \subset_\eta \mathcal H$;

(ii) $1-p_0 \in \her(C_0(X \setminus \{x\}))$;

(iii) for $j \neq 0$, $p_{j} \in \her(C_0(X \setminus \{x\}))$;

(iv) for $j=-k,\dots,m-1$, 
\begin{equation} p_{j+1} = u p_j u^*. \end{equation}

(v) for $f \in \mathcal H$ and $j=-k,\dots,m$, there exists $\lambda_{f,j} \in \mathbb C$ such that
\begin{equation} fp_j = \lambda_{f,j}p_j; \end{equation}

(vi) for $f \in \mathcal H$ and $j=0,\dots,k$, 
\begin{equation} \lambda_{f,-j} \approx_\eta \lambda_{f,m-j}.
\end{equation}
\end{lemma}

\begin{proof}
The crux of this argument is the use of Berg's techniques, inspired by the proof of \cite[Lemma 4.2]{LinPhillips:CrProd}.

Fix a compatible metric $d$ on $X$ and let $\dl > 0$ be such that, for $x,y \in X$, if $d(x,y) < \dl$ then
\begin{equation}
\label{eq:CPProjsdlDef}
 |f(x)-f(y)| < \eta, \text{ for all }f \in \mathcal G. \end{equation}
By minimality, there exists $m > k$ be such that
\begin{equation}
\label{eq:CPProjsmDef}
d(\alpha^{m-j}(x),\alpha^{-j}(x)) < \dl \end{equation}
for all $j=0,\dots,k$.

Let $V$ be a neighbourhood of $x$ such that
\begin{equation}
\label{eq:CPProjsVTranslates}
 \alpha^{-k}(V),\alpha^{-k+1}(V),\dots,\alpha^m(V) \end{equation}
are disjoint.
By possibly shrinking $V$, each function $f \in \mathcal G$ can be perturbed by at most $\eta$ to a function $f'$ which is constant on each set in \eqref{eq:CPProjsVTranslates}.
Set
\begin{equation} \mathcal H := \{f' \mid f \in \mathcal G\}, \end{equation}
and we see that (i) holds.
For $f \in \mathcal H$ and $j=-k,\dots,m$, set $\lambda_{f,j}$ equal to the complex number satisfying
\begin{equation}
\label{eq:CPProjslambdaDef}
f|_{\alpha^j(V)} \equiv \lambda_{f,j}. \end{equation}

It follows from \eqref{eq:CPProjsdlDef} and \eqref{eq:CPProjsmDef} that
\begin{equation}
\label{eq:CPProjslambdaClose}
|\lambda_{f,-j}-\lambda_{f,m-j}| < \eta
\end{equation}
for all $f \in \mathcal H$, $j=0,\dots,k$, establishing (vi).

Let $h_0 \in C_0(V)_+$ be a function that is constant equal to $1$ in a neighbourhood of $x$.
Since $(A_x \otimes U) \cap \{1-h_0\}^\perp$ is a non-zero full hereditary subalgebra of the UHF-stable, unital algebra $A_x \otimes U$, there exists a non-zero projection
\begin{equation} p_0' \in (A_x \otimes U) \cap \{1-h_0\}^\perp. \end{equation}
Since $\alpha$ is minimal, we may find $g_0',g_0 \in C(X)_+ \cap \{1-h_0\}^\perp$ such that
\begin{equation} d_\tau(g_0') < \tau(p_0')\text{ for all }\tau \in T(A_x), \end{equation}
$g_0 \vartriangleleft g_0'$, and $g_0$ is constantly equal to $1$ in a neighbourhood of $x$.
By strict comparison, it follows that $[g_0'] \leq [p_0']$ in the Cuntz semigroup.
Therefore, by Lemma \ref{lem:ProjBetween}, there exists a projection
\begin{align}
p_0 &\in A_x \otimes U
\end{align}
such that
\begin{align}
\label{eq:CPProjsg0p0h0}
g_0 &\vartriangleleft p_0 \vartriangleleft h_0.
\end{align}
(We can now forget $g_0'$ and $p_0'$.)

For $j = -k,\dots,m$, set
\begin{align}
\label{eq:CPProjsp0Def}
p_j &:= u^j p_0 u^{-j}.
\end{align}

Thus, we have defined projections $p_{-k},\dots,p_m \in A \otimes U$, and it is quite clearly the case that (iii)--(v) hold.
From (v), it follows that each $p_j$ commutes with $\mathcal H$.
Point (ii) follows since
\begin{equation} 0 \leq 1-p_0 \leq 1-g_0 \in C_0(X \setminus \{x\}). \end{equation}
The projections are orthogonal by disjointness of $(\alpha^j(V))_{j=-k,\dots,m}$.

Let us show that $p_j \in A_x \otimes U$ for $j=-k,\dots,m$.
For $j=0$, this is true by our choice of $p_0$.

For $j=-k,\dots,0$,
\begin{equation} p_j = u^j p_0 u^{-j} = (h_0u^{-j})^*p_0(h_0u^{-j}) \in A_x \otimes U \end{equation}
by Lemma \ref{lem:InAx} (ii).

$j=1$ is a special case; here, we note that since $h_0-g_0$ vanishes in a neighbourhood of $x$, and by \eqref{eq:CPProjsg0p0h0}, it follows that there exists $e \in C_0(X \setminus \{x\})_+$ such that $p_0-g_0 \vartriangleleft e$.
Therefore,
\begin{align}
\notag
p_1 &= up_0u^* \\
&= ue(p_0-g_0)(ue)^* + ug_0u^* \in A_x \otimes U.
\end{align}

Now, for $j=2,\dots,m$,
\begin{equation} p_j = u^{j-1}p_1 u^{-(j-1)} = (u^{j-1}(1-h_0)) p_1 (u^{j-1}(1-h_0))^* \in A_x \otimes U \end{equation}
by Lemma \ref{lem:InAx} (i).
\end{proof}

\begin{lemma}
\label{lem:NilPIPath}
Let $A$ be a $\mathrmm C^*$-algebra.
There exists a uniformly continuous map
\begin{align}
 \gamma\colon \{v \in A \mid v \text{ is a partial isometry and } v^2=0\} \times [0,1] &\to A
\end{align}
such that, for each $v$, the assignment $t \mapsto \gamma(v,t)$ defines a path of projections in $\mathrmm C^*(v)$, from $vv^*$ to $v^*v$.
Also, for any unitary $u \in A$,
\begin{equation} u\gamma(v,t)u^* = \gamma(uvu^*,t). \end{equation}
\end{lemma}

\begin{proof}
Simply set
\begin{equation} \gamma(v,t) := t^2v^*v + t(1-t)(v+v^*) + (1-t)^2vv^*. \end{equation}
In matrix form, this is
\begin{equation}
\left(\begin{array}{cc}
t^2 & t(1-t) \\
t(1-t) & (1-t)^2
\end{array}\right).
\end{equation}
The desired properties are easy to verify.
\end{proof}

\begin{lemma}
\label{lem:UniversalCircle}
The universal $\mathrmm C^*$-algebra generated by partial isometries $w_1,\dots,w_m$ such that
\begin{align}
\notag
w_i^*w_i &= w_{i+1}w_{i+1}^*,\quad i=1,\dots,m-1, \\
\notag
w_{m}^*w_{m} &= w_1w_1^*, \quad \text{and} \\
\label{eq:UniversalCircleRels}
w_i^*w_j&=0, i \neq j
\end{align}
is isomorphic to $\mathrmm C(\mathbb T) \otimes M_m$, via an isomorphism sending $w_i$ to $1 \otimes e_{i,i+1}$ for $i=1,\dots,m-1$ and $w_m$ to $z \otimes e_{m,1}$ where $z \in \mathrmm C(\mathbb T)$, is a generating unitary.
\end{lemma}

\begin{proof}
Let $A$ be the universal $\mathrmm C^*$-algebra generated by the $w_i$ satisfying the relations \eqref{eq:UniversalCircleRels}.
It is easy to see that the described map, call it $\pi$, is a surjective homomorphism.
To see that $\pi$ is injective, one derives from the relations that

(i) $\sum_{i=1}^m w_i^*w_i$ is a unit for $A$;

(ii) $u:= \sum_{i=1}^m w_iw_{i+1}\cdots w_mw_1 \cdots w_{i-1}$ is a unitary in the centre of $A$, so that under any irreducible representation, the image of $u$ is a scalar of modulus $1$;

(iii) $\{w_i \mid i=0,\dots,m-1\}$ generates a unital copy of $M_m$ in $A$ (therefore also in any irreducible representation); and

(iv) $w_m=w_{m-1}^*\cdots w_1^* u$.

From this, one sees that every irreducible representation of $A$ is equivalent to $\ev_z \circ \pi$ for some point $z \in \mathbb T$, and therefore, $\pi$ is injective.
\end{proof}

\begin{lemma}
\label{lem:UHFCrossedProdSpecific}
Let $X,\alpha$ be as in Theorem \ref{thm:UHFCrossedProd}, let $x \in X$, and let $U$ be an infinite dimensional UHF algebra.
Let $\mathcal C$ denote the class of all $\mathrmm C^*$-algebras of the form
\begin{equation} p(A_x \otimes U)p \oplus \mathrmm C(\mathbb T, M_k), \end{equation}
where $p \in A_x \otimes U$ is a projection and $k \in \mathbb N$.

Then $(\mathrmm C(X) \rtimes_\alpha \mathbb Z) \otimes U$ is locally $\mathcal C$.
\end{lemma}

\begin{proof}
This proof uses ideas from \cite[Lemma 4.2]{LinPhillips:CrProd}.

Let $\mathcal F \subset (\mathrmm C(X) \rtimes_\alpha \mathbb Z) \otimes U$ be a finite subset to be approximated and let $\e > 0$ be a specified tolerance.
By approximating within $(\mathrmm C(X) \rtimes_\alpha \mathbb Z) \otimes M$ for some matrix subalgebra $M$ of $U$, and then observing that it suffices to approximate the matrix entries in $(\mathrmm C(X) \rtimes_\alpha \mathbb Z) \otimes (U \cap M')$, we see that without loss of generality, $\mathcal F \subset (\mathrmm C(X) \rtimes_\alpha \mathbb Z) \otimes 1_U$, which we hereby identify with $\mathrmm C(X) \rtimes_\alpha \mathbb Z$.
Without loss of generality again, $\mathcal F = \{u\} \cup \mathcal G$ where $\mathcal G \subset \mathrmm C(X)$.

Let $k \in \mathbb N$ be such that, for the map $\gamma$ defined in Lemma \ref{lem:NilPIPath} (for $A=(\mathrmm C(X) \rtimes_\alpha \mathbb Z) \otimes U$), if $|t-t'|<1/k$ then 
\begin{equation}
\label{eq:UHFCrossedProdSpecifickDef}
 \|\gamma(v,t)-\gamma(v,t')\| < \e/16.
\end{equation}

Apply Lemma \ref{lem:CPProjs} with $\eta:=\e/16$, to obtain $m, \mathcal H$, and $p_{-k},\dots,p_m \in A_x \otimes U$ satisfying the conclusions of that lemma.

Note that $p_m=u^m p_0 u^{-m} \sim p_0$ in $A \otimes U$.
Since the embedding of $A_x \otimes U$ in $A \otimes U$ induces an isomorphism between ordered $K_0$-groups \cite[Theorem 4.1 (3)]{Phillips:arsh}, and since $A_x \otimes U$ has cancellation of projections, there exists $v_0 \in A_x \otimes U$ such that
\begin{equation} p_0 = v_0v_0^* \quad \text{and} \quad v_0^*v_0 = p_m. \end{equation}

For $j=k,\dots,1$, set 
\begin{equation} v_j := u^{-j}v_0u^j = u^{-j}h_0v_0h_mu^j. \end{equation}
By Lemma \ref{lem:InAx} $v_j \in A_x \otimes U$.
Also, an easy calculation yields
\begin{equation} p_{-j} = v_jv_j^* \quad \text{and} \quad v_0^*v_0 = p_{m-j}. \end{equation}
Using $\gamma$ from Lemma \ref{lem:NilPIPath}, set
\begin{equation} \tilde p_{-j} := \gamma(v_j,j/k) \in (A_x \otimes U) \cap \her(p_{-j}+p_{m-j}). \end{equation}

From the definition,
\begin{equation}
\label{eq:UHFCrossedProdSpecifictildepEndpts}
\tilde p_0=p_0, \quad \tilde p_{-k} = p_{m-k},
\end{equation}
and for each $j=1,\dots,k$,
\begin{eqnarray}
\label{eq:UHFCrossedProdSpecificApproxConj}
\notag
\|u\tilde p_{-j}u^* - \tilde p_{-j+1}\| &=& \|u\gamma(v_{-j},j/k)u^* - \gamma(v_{-j+1},(j-1)/k)\| \\
\notag
&=& \|\gamma(uv_{-j}u^*,j/k) - \gamma(v_{-j+1},(j-1)/k)\| \\
\notag
&=& \|\gamma(v_{-j+1},j/k) - \gamma(v_{-j+1},(j-1)/k)\| \\
&\stackrel{\eqref{eq:UHFCrossedProdSpecifickDef}}<& \e/16,
\end{eqnarray}
and, therefore, there exists $w_{-j} \in A \otimes U$ such that
\begin{equation}
\label{eq:UHFCrossedProdSpecificwjDef1}
 w_{-j}^*w_{-j} = \tilde p_{-j+1}, \quad w_{-j}w_{-j}^* = \tilde p_{-j}, \quad\text{and}\quad \|w_{-j}-\tilde p_{-j}u^*\| < \e/4. \end{equation}
For $j=0,\dots,m-k-1$, set
\begin{equation}
\label{eq:UHFCrossedProdSpecificwjDef2}
 w_j := p_ju^* \in A \otimes U. \end{equation}

Now, set
\begin{equation} C := \mathrmm C^*(\{w_{-k},\dots,w_{m-k-1}\}) \subseteq A \otimes U. \end{equation}
To complete the proof, we shall show that:
\begin{enumerate}
\item $C \cong \mathrmm C(\mathbb T, M_m)$;
\item $1_C \in A_x \otimes U$ and therefore $1_{A\otimes U}-1_C \in A_x \otimes U$;
\item $\|[1_C,\mathcal F]\| < \e/2$;
\item $\mathcal F(1_{A\otimes U}-1_C) \subset A_x \otimes U$; and
\item $\mathcal F1_C \approx_{\e/2} C$.
\end{enumerate}

(i):
Lemma \ref{lem:UniversalCircle} shows that there is a surjective map from $\mathrmm C(\mathbb T) \otimes M_m$ to $C$.
Moreover, since $u1_C$ has non-zero $K_1$-class (in $1_\mathrmm C((\mathrmm C(X) \rtimes_\alpha \mathbb Z) \otimes U)1_C$) and is approximately contained in $C$ (as will be shown in (v)), $K_1(C) \neq 0$, and, therefore, this map must be surjective.

(ii):
We have already seen that $\tilde p_{-k},\dots,\tilde p_0, p_1,\dots,p_{m-k-1} \in A_x$.
It is not hard to see that the sum of these projections is $1_C$.

(iii):
For $g \in \mathcal G$, let $f \in \mathcal H$ be such that $g \approx_\eta f$.
We have 
\begin{align}
\notag
[1_C,f] &= \left[ \sum_{j=0}^k \tilde p_{-j} + \sum_{j=1}^{m-k} p_j, f \right] \\
\label{eq:UHFCrossedProdSpecific1CCommuteA}
&= \sum_{j=0}^k [\tilde p_{-j},f] + \sum_{j=1}^{m-k-1} [p_j,f].
\end{align}
From Lemma \ref{lem:CPProjs}, each $p_j$ commutes with $f$, so that the second term vanishes.
By Lemma \ref{lem:CPProjs} (v) and (vi), and since $\tilde p_{-j}$ is in the hereditary subalgebra generated by $p_{-j}+p_{m-j}$, we see that for $f \in \mathcal F'$,
\begin{equation}
\label{eq:UHFCrossedProdSpecific1CCommuteB}
\|[\tilde p_{-j},f]\| < 2\eta. \end{equation}
Also, since $p_{-j}+p_{m-j}$ commutes with $f$, $[\tilde p_{-j},f]$ is in the hereditary subalgebra generated by $p_{-j}+p_{m-j}$, and therefore, the first sum in \eqref{eq:UHFCrossedProdSpecific1CCommuteA} is orthogonal.
Hence,
\begin{eqnarray}
\|[1_C,f]\| &=& \max_{j=0,\dots,k} \|[\tilde p_{-j},f]\| \\
&\stackrel{\eqref{eq:UHFCrossedProdSpecific1CCommuteB}}<& 2\eta,
\end{eqnarray}
and thus $\|[1_C,g]\| < 4\eta < \e/2$.

Next, we compute
\begin{eqnarray}
\notag
&& u1_Cu^* - 1_C \\
\notag
&=& \sum_{j=0}^k (u\tilde p_{-j}u^*-\tilde p_{-j}) + \sum_{j=1}^{m-k-1} (up_ju^*-p_j) \\
\notag
&\stackrel{\eqref{eq:UHFCrossedProdSpecifictildepEndpts}}=& \sum_{j=1}^k (u\tilde p_{-j}u^*-\tilde p_{-j+1}) + up_0u^* - p_{m-k} + \sum_{j=1}^{m-k-1} (up_ju^*-p_j) \\
\notag
&=& \sum_{j=1}^k (u\tilde p_{-j}u^*-\tilde p_{-j+1}) + \sum_{j=0}^{m-k-1} (up_ju^*-p_{j+1}) \\
&\stackrel{\text{Lemma \ref{lem:CPProjs} (iv)}}=& \sum_{j=1}^k (u\tilde p_{-j}u^* - \tilde p_{-j+1}).
\end{eqnarray}
Note that $\tilde p_{-j}$ is in the hereditary subalgebra generated by $p_{-j}+p_{m-j}$, so that by Lemma \ref{lem:CPProjs} (iv), $u\tilde p_{-j}u^*$ is in the hereditary subalgebra generated by $p_{-j+1}+p_{m-j+1}$, and, therefore, this last sum is orthogonal.
Consequently, we obtain
\begin{eqnarray}
\|u1_Cu^* - 1_C\| &=& \max_{j=1,\dots,k} \|u\tilde p_{-j}u^* - \tilde p_{-j+1}\| \\
\notag
&\stackrel{\eqref{eq:UHFCrossedProdSpecificApproxConj}}<& \e/16,
\end{eqnarray}
as required.

(iv):
Since $(1-1_C) \in A_x \otimes U$, it follows that $\mathcal G(1-1_C) \subset A_x \otimes U$.
The non-trivial part is showing $u(1-1_C) \in A_x \otimes U$.

For $j=1,\dots,m$, since $\tilde p_{-j} \in \her(p_{-j}+p_{m-j})$ and by Lemma \ref{lem:CPProjs} (iii), we see that $\tilde p_{-j} \in \her(C_0(X \setminus \{x\})$.
Using this and Lemma \ref{lem:CPProjs} (ii) and (iii), we find that
\begin{equation} 1-1_C = 1-p_0 + \sum_{j=1}^k \tilde p_{-j} + \sum_{j=1}^{m-k-1} p_j \in \her(C_0(X \setminus \{x\})). \end{equation}
Therefore, for $\dl > 0$, there exists $e \in C_0(X \setminus \{x\})$ such that $1-1_C \approx_\dl e(1-1_C)$.
Hence,
\begin{equation} u(1-1_C) \approx_{\dl} ue(1-1_C) \in A_x \otimes U. \end{equation}
Since $\dl$ is arbitrary, $u(1-1_C) \in A_x \otimes U$ as required.

(v):
For $g \in \mathcal G$, let $f \in \mathcal H$ be such that $g \approx_\eta f$.
We shall show that 
\begin{equation} g1_C \approx_{\e/2} \sum_{j=1}^{k} \lambda_{f,-j} \tilde p_{-j} + \sum_{j=0}^{m-k-1} \lambda_{f,j} p_j. \end{equation}
Of course, it suffices to show this with $f$ in place of $g$ and $\eta$ in place of $\e/2$ (since $\eta < \e/2$).

We have
\begin{eqnarray}
\notag
&& \hspace*{-20mm} f1_C - \left(\sum_{j=1}^{k} \lambda_{f,-j} \tilde p_{-j} + \sum_{j=0}^{m-k-1} \lambda_{f,j} p_j\right) \\
\notag
&=& \sum_{j=1}^{k} (f - \lambda_{f,-j}) \tilde p_{-j} + \sum_{j=0}^{m-k-1} (f-\lambda_{f,j})p_j \\
\label{eq:UHFCrossedProdSpecificvStep}
&\stackrel{\text{Lemma \ref{lem:CPProjs} (v)}}=& \sum_{j=1}^{k} (f - \lambda_{f,-j}) \tilde p_{-j}
\end{eqnarray}
Since $\tilde p_{-j} \in \her(p_{-j}+p_{m-j})$, by Lemma \ref{lem:CPProjs} (v) and (vi),
\begin{equation}
\|(f-\lambda_{f,-j})\tilde p_{-j}\| < \eta.
\end{equation}
Moreover since $f$ commutes with $p_{-j}+p_{m-j}$, it follows that $(f-\lambda_{f,-j})\tilde p_{-j} \in \her(p_{-j}+p_{m-j})$, whence the sum in \eqref{eq:UHFCrossedProdSpecificvStep} is an orthogonal sum, so that
\begin{eqnarray}
\notag
&& \hspace*{-20mm} \|f1_C - \left(\sum_{j=1}^{k} \lambda_{f,-j} \tilde p_{-j} + \sum_{j=0}^{m-k-1} \lambda_{f,j} p_j\right)\| \\
\notag
&=& \max_{j=1,\dots,k} \|(f-\lambda_{f,-j}) \tilde p_{-j}\| \\
&<&\eta.
\end{eqnarray}

Finally, we have
\begin{eqnarray}
&& \hspace*{-10mm} \| u1_C - (w_{-k}^* + \cdots + w_{m-k-1}^*) \| \\
\notag
&\stackrel{\eqref{eq:UHFCrossedProdSpecificwjDef2}}=& \| (u\tilde p_{-k} - w_{-k}^*) + \cdots + (u\tilde p_{-1} - w_{-1}^*)\| \\
\notag
&\leq&
\| \sum_{j \text{ even}} (u\tilde p_{-j} - w_{-j}^*) \| + 
\| \sum_{j \text{ odd}} (u\tilde p_{-j} - w_{-j}^*) \|.
\end{eqnarray}
Now, note that $p_{-j+1}+p_{m-j+1}$ (respectively, $p_{-j}+p_{m-j}$) acts as the unit on $u\tilde p_{-j} - w_{-j}^*$ on the left (respectively, right).
Therefore, we can see that each of the two sums in the last inequality above is an orthogonal sum, and we continue,
\begin{eqnarray}
\notag
\| u1_C - (w_{-k}^* + \cdots + w_{m-k-1}^*) \|
&\leq&
2\max_{j=1,\dots,k} \|(u\tilde p_{-k} - w_{-k}^*) \| \\
&\stackrel{\eqref{eq:UHFCrossedProdSpecificwjDef1}}{\leq}&
\e/2,
\end{eqnarray}
as required.
\end{proof}

\begin{proof}[Proof of Theorem \ref{thm:UHFCrossedProd}]
In the context of Lemma \ref{lem:UHFCrossedProdSpecific}, $A_x$ is locally subhomogeneous, and therefore so also is any unital corner of $A_x$
(any subhomogeneous approximation that approximately contains a projection $p$ will exactly contain a projection $p'$ which is near to $p$).
Thus, it follows from Lemma \ref{lem:UHFCrossedProdSpecific} that $(\mathrmm C(X) \rtimes_\alpha \mathbb Z) \otimes U$ is locally subhomogeneous.
\end{proof}

\begin{cor}
\label{cor:CrProddr}
Let $X$ be an infinite compact metrizable space, let $\alpha\colon X \to X$ be a minimal homeomorphism, and let $U$ be an infinite dimensional UHF algebra.
Then
\begin{align}
\dr \left((\mathrmm C(X) \rtimes_\alpha \mathbb Z) \otimes U\right) &\leq 2 \quad \text{and} \\
\notag
\dr \left((\mathrmm C(X) \rtimes_\alpha \mathbb Z) \otimes \mathcal Z\right) &\leq 5.
\end{align}
In particular, if $\alpha$ has mean dimension zero then $\dr(\mathrmm C(X) \rtimes_\alpha \mathbb Z) \leq 5$.
\end{cor}

\begin{remark}
\label{rmk:LinCrossedClass}
In fact, a classification of $\mathcal Z$-stabilized crossed products implies that $(C(X) \rtimes_\alpha \mathbb Z) \otimes \mathcal Z$ is approximately subhomogeneous has decomposition rank at most two.
See Corollary \ref{cor:CrProdClass} and its proof below.
\end{remark}

\begin{proof}
The first inequality follows directly from Theorems \ref{thm:UHFCrossedProd} and \ref{thm:MainThmA}, upon noting that $U$ is itself $\mathcal Z$-stable.
For the second inequality, note that by \cite[Theorem 3.4]{RordamWinter:Z}, $(\mathrmm C(X) \rtimes_\alpha \mathbb Z) \otimes \mathcal Z$ is an inductive limit of copies of the $\mathrmm C^*$-algebra
\begin{equation} (\mathrmm C(X) \rtimes_\alpha \mathbb Z) \otimes \mathcal Z_{2^\infty,3^\infty}, \end{equation}
which is a continuous field over $[0,1]$, whose fibres are each equal to $\mathrmm C(X) \rtimes_\alpha \mathbb Z \otimes U$ for a UHF algebra $U$.
Therefore, the second line follows from the first line and \cite[Lemma 3.1]{Carrion:RFproducts}.
The last statement follows from the main result of \cite{ElliottNiu:MeanDim0}, which is that when $\alpha$ has mean dimension zero, the crossed product is $\mathcal Z$-stable.
\end{proof}

\begin{proof}[Proof of Corollary \ref{cor:CrProdClass}]
Let $\mathcal C$ denote the class of simple, separable, unital, nuclear, infinite dimensional $\mathrm C^*$-algebras with finite decomposition rank which satisfy the Universal Coefficient Theorem.
The main result of \cite{EGLN:Class2} says that, given $A,B \in \mathcal C$,
\begin{equation} A \cong B \quad \text{if and only if} \quad \mathrm{Ell}(A) \cong \mathrm{Ell}(B). \end{equation}
(This result makes heavy use of the classification result \cite{GongLinNiu}, and builds on the kep step taken in \cite{EN:Class1}.)

By \cite{RosenbergSchochet:UCT}, the $\mathrm C^*$-algebra $A$ in the statement of Corollary \ref{cor:CrProdClass} satisfies the Universal Coefficient Theorem, and by Corollary \ref{cor:CrProddr}, it has finite decomposition rank; thus $A \in \mathcal C$.
Also $B \in \mathcal C$ either by the same argument or by assumption.
Hence $A \cong B$ if and only if $\mathrm{Ell}(A) \cong \mathrm{Ell}(B)$, proving the last part of Corollary \ref{cor:CrProdClass}.

Here is the argument showing that $A$ is an inductive limit of subhomogeneous $\mathrmm C^*$-algebras.
By \cite{Elliott:ashrange}, there exists a simple separable unital $\mathrmm C^*$-algebra $B$ which is an inductive limit of subhomogeneous $\mathrmm C^*$-algebras of topological dimension at most two, satisfying $\mathrm{Ell}(B) \cong \mathrmm{Ell}(A)$.
It follows by \cite{Winter:drSH} and Proposition \ref{prop:drLocal}, $\dr(B) \leq 2$, and therefore $B \in \mathcal C$.
Consequently, $A \cong B$ which shows that $A$ is also an inductive limit of subhomogeneous $\mathrmm C^*$-algebras.
\end{proof}

\begin{remark}
Prior to the far-reaching classification result of \cite{EGLN:Class2} used in the above proof, Lin showed, using Corollary \ref{cor:CrProddr}, that $\mathcal Z$-stabilized minimal $\mathbb Z$-crossed products are classifiable (in the sense of \cite{GongLinNiu}) \cite{Lin:CrossedClass}.
\end{remark}

As mentioned earlier, the Putnam algebra $A_x$ associated to a minimal crossed product has been known for some time to be approximately subhomogeneous, simple, separable, and unital.
It is a consequence of Theorem \ref{thm:MainThmA} that these $\mathrm C^*$-algebras are classifiable, i.e., satisfy the hypotheses of the classification result in \cite{EGLN:Class2}.
The following is a natural question.

\begin{question}
Let $B$ be a simple separable unital $\mathrmm C^*$-algebra with finite decomposition rank and which satisfies the Universal Coefficient Theorem.
Does there exist a compact Hausdorff space $X$, a minimal homeomorphism $\alpha:X \to X$, and a point $x \in X$ such that $B \cong A_x$?
(Here, $A_x$ is the Putnam subalgebra of $C(X) \rtimes_\alpha \mathbb Z$, as defined in \eqref{eq:PutnamSugalgDef}.)
\end{question}

\newcommand{\cstar}{$\mathrmm C^*$}

\end{document}